\documentclass{article}
\date{\vspace{-5ex}}
\usepackage[margin=1in]{geometry}
\usepackage{color,soul}
\usepackage{setspace}
\usepackage{float}
\usepackage{subfloat}
\usepackage[utf8]{inputenc} 
\usepackage[T1]{fontenc}    
\usepackage{url}            
\usepackage{booktabs}       
\usepackage{amsfonts}       
\usepackage{nicefrac}       
\usepackage{microtype}      

\usepackage{mathtools,amssymb,graphicx,comment}
\usepackage{amsmath,amsthm}
\usepackage{bbm}
\usepackage{subcaption}
  {
      
  }
  \usepackage{breqn}
\usepackage{enumitem}
  \usepackage{tcolorbox}

\usepackage{hyperref}
\hypersetup{%
  colorlinks=false,
  linkbordercolor=red
}  

\definecolor{darkred}{RGB}{150,0,0}
\definecolor{darkgreen}{RGB}{0,150,0}
\definecolor{darkblue}{RGB}{0,0,150}
\hypersetup{colorlinks=true, linkcolor=red, citecolor=darkgreen, urlcolor=darkblue}
\newcommand\blfootnote[1]{%
  \begingroup
  \renewcommand\thefootnote{}\footnote{#1}%
  \addtocounter{footnote}{-1}%
  \endgroup
}

\newcommand*{\email}[1]{\texttt{#1}}

\newcommand{\Gamc}{\Upsilon}

\newcommand{\gammat}{\widetilde{\gamma}}
\newcommand{\gammaone}{\gamma_1}
\newcommand{\gammatwo}{\gamma_2}

\newcommand{\thetao}{\overline{\theta}}

\newcommand{\mathleft}{\@fleqntrue\@mathmargin0pt}
\newcommand{\mathcenter}{\@fleqnfalse}
\newcommand{\corr}[2]{{\rm{corr}}\left(\,{#1}\,;\,{#2}\,\right)}

\newcommand{\env}[3]{\mathcal{M}_{{#1}}\left({#2};{#3}\right)}
\newcommand{\prox}[3]{\mathrm{prox}_{{#1}}\left({#2};{#3}\right)}
\newcommand{\proxri}[3]{p_{#1}\left({#2},{#3}\right)}

\newcommand{\ellprox}[3]{L_{#1}\left({#2},{#3}\right)}

\newcommand{\ellp}{\ell^\prime}

\newcommand{\envl}[2]{\mathcal{M}_{\ell}\left({#1};{#2}\right)}
\newcommand{\proxl}[2]{\mathrm{prox}_{\ell}\left({#1};{#2}\right)}

\newcommand{\elldd}{\ell^{\prime\prime}}
\newcommand{\elld}{\ell^{\prime}}

\newcommand{\envdx}[3]{\mathcal{M}^{\prime}_{{#1},1}\left({#2};{#3}\right)}

\newcommand{\envdla}[3]{\mathcal{M}^{\prime}_{{#1},2}\left({#2};{#3}\right)}

\newcommand{\ourx}{\al G + \mu S Y}

\newcommand{\R}{\mathbb{R}}

\newcommand{\al}{\alpha}

\newcommand{\xh}{\widehat{\x}}

\providecommand{\abs}[1]{\lvert#1\rvert}
\providecommand{\norm}[1]{\lVert#1\rVert}

\DeclarePairedDelimiterX{\inp}[2]{\langle}{\rangle}{#1, #2}

\newcommand{\Id}{\mathbf{I}}

\newcommand{\ksi}{\xi}

\usepackage{dsfont}

\newcommand{\simiid}{\stackrel{\text{iid}}{\sim}}

\newcommand{\Pro}{\mathbb{P}}

\newcommand{\soft}[2]{{\Hc}\left({#1};{#2}\right)}

\theoremstyle{theorem}
\newtheorem{propo}{Proposition}[section]
\newtheorem{thm}{Theorem}[section]
\newtheorem{lem}{Lemma}[section]
\newtheorem{cor}{Corollary}[section]
\newtheorem{ass}{Assumption}
\theoremstyle{remark}
\newtheorem{remark}{Remark}

\theoremstyle{definition}

\newcommand{\eps}{\varepsilon}

\newcommand{\sign}{\mathrm{sign}}

\newcommand{\Exp}{\mathbb{E}}               
\newcommand{\E}{\mathbb{E}}                    
\newcommand{\la}{{\lambda}}                     

\newcommand{\sig}{\sigma}

\newcommand{\nn}{\notag}

\newcommand{\G}{\mathbf{G}}

\newcommand{\A}{\mathbf{A}}

\newcommand{\x}{\mathbf{x}}
\newcommand{\ub}{\mathbf{u}}

\newcommand{\w}{\mathbf{w}}

\newcommand{\g}{\mathbf{g}}
\newcommand{\vb}{\mathbf{v}}

\newcommand{\y}{\mathbf{y}}
\newcommand{\s}{\mathbf{s}}

\newcommand{\ab}{\mathbf{a}}

\newcommand{\h}{\mathbf{h}}

\newcommand{\betab}{\boldsymbol{\beta}}

\newcommand{\Fc}{\mathcal{F}}
\newcommand{\Mc}{\mathcal{M}}

\newcommand{\Sc}{{\mathcal{S}}}

\newcommand{\Rc}{\mathcal{R}}
\newcommand{\Nn}{\mathcal{N}}

\newcommand{\Hc}{\mathcal{H}}

\newcommand{\Jc}{\mathcal{J}}
\newcommand{\Ic}{\mathcal{I}}

\newcommand{\beq}{\begin{equation}}
\newcommand{\eeq}{\end{equation}}
\newcommand{\bea}{\begin{align}}
\newcommand{\eea}{\end{align}}

\newcommand{\vp}{\vspace{4pt}}

\usepackage{amsmath}
\def\bea#1\eea{\begin{align}#1\end{align}}
\title{Sharp Asymptotics and Optimal Performance\\ for Inference in Binary Models}

\usepackage{colortbl}

\author{%
Hossein Taheri, Ramtin Pedarsani, and  Christos Thrampoulidis\\
{\small{Department of Electrical and Computer Engineering, }}
\\\small{University of California, Santa Barbara}.
\blfootnote{\small{Part of this work to appear in the 23rd International Conference on 
Artificial Intelligence and Statistics (AISTATS), 2020. Emails: \email  \{hossein, ramtin, cthrampo\}@ucsb.edu}}}

\begin{document}
\maketitle

%
\begin{abstract}
%
We study convex empirical risk minimization for high-dimensional inference in binary models. Our first result sharply predicts the statistical performance of such estimators in the linear asymptotic regime under isotropic Gaussian features. Importantly, the predictions hold for a wide class of convex loss functions, which we exploit in order to prove a bound on the best achievable performance among them. Notably, we show that the proposed bound is tight for popular binary models (such as Signed, Logistic or Probit), by constructing appropriate loss functions that achieve it. More interestingly, for binary linear classification under the Logistic and Probit models, we prove that the performance of least-squares is no worse than 0.997 and 0.98 times the optimal one. Numerical simulations corroborate our theoretical findings and suggest they are accurate even for relatively small problem dimensions. 

%
\end{abstract}

%
\section{Introduction}

%
%
%

\subsection{Motivation}
Classical estimation theory studies problems in which the number of unknown parameters $n$ is small compared to the number of observations $m$. In contrast, modern inference problems are typically \emph{high-dimensional}, that is $n$ can be of the same order as $m$. Examples are abundant in a wide range of signal processing and machine learning applications such as medical imaging, wireless communications, recommendation systems and so on. Classical tools and theories are not applicable in these modern inference problems. As such, over the last two decades or so, the study of high-dimensional estimation problems has received significant attention. 

Perhaps the most well-studied setting is that of noisy linear observations (aka, linear regression). The literature on the topic is vast with remarkable contributions from the statistics, signal processing and machine learning communities. Several recent works focus on the \emph{linear asymptotic regime} and derive \emph{sharp} results on the inference performance of appropriate convex optimization methods, e.g., \cite{donoho2006compressed,Sto,Cha,DMM,tropp2014convex,TroppUniversal,montanariLasso,StoLASSO,OTH13,karoui2013asymptotic,bean2013optimal,COLT,donoho2016high,Master,advani2016statistical,weng2018overcoming,TSP18,miolane2018distribution,bu2019algorithmic,xu2019consistent,celentano2019fundamental}. These works show that, albeit challenging, \emph{sharp} results are advantageous over loose order-wise bounds. Not only do they allow for accurate comparisons between different choices of the optimization parameters, but they also form the basis for establishing optimal such choices as well as fundamental performance limitations. 

This paper takes this recent line of work a  step further by demonstrating that results of this nature can be achieved in binary observation models. While we depart from the previously studied linear regression model, we remain faithful to the requirement and promise of sharp results. Binary models are popularly applicable in a wide range of signal-processing (e.g., highly quantized measurements) and machine learning (e.g., binary classification) problems. 
We derive sharp asymptotics for a rich class of convex optimization estimators, which includes least-squares, logistic regression and hinge-loss as special cases. Perhaps more interestingly, we use these results to derive fundamental performance limitations and design optimal loss functions that provably outperform existing choices.

In Section \ref{sec:form} we formally introduce the problem setup. The paper's main contributions and organization are presented in Section \ref{sec:over}. A detailed discussion of prior art follows in Section \ref{sec:prior}. 

\vp
\noindent\textbf{Notation.}~~
The symbols $\Pro(\cdot)$, $\Exp\left[\cdot\right]$ and $\text{Var}[\cdot]$ denote probability, expectation and variance. We use boldface notation for vectors. $\|\vb\|_2$ denotes the Euclidean norm of a vector $\vb$. We write $i\in[m]$ for $i=1,2,\ldots,m$. 
 When writing $x_* = \arg\min_x f(x),$ we let the  operator $\arg\min$ return any one of the possible minimizers of $f$. For all $x\in\mathbb{R}$, $\Phi(x)$ is the cumulative distribution function of standard normal and Gaussian $Q$-function at $x$ is defined as $Q(x)=1-\Phi(x).$
\\
\subsection{Problem statement}\label{sec:form}
Consider the problem of recovering $\x_0\in\R^n$ from observations $y_i=f(\ab_i^T\x_0),~i\in[m]$, where $f:\mathbb{R}\rightarrow\{-1,+1\}$ is a (possibly random) binary function. We study the performance of \emph{empirical-risk minimization (ERM)} estimators $\hat\x_\ell$ that solve the following optimization problem for some \emph{convex} loss function $\ell:\R\rightarrow\R$
\bea\label{eq:gen_opt}
\xh_\ell := \arg\min_\x \frac{1}{m}\sum_{i=1}^{m} \ell(y_i\ab_i^T\x).
\eea 
\noindent\textbf{Model.}~The binary observations $y_i, i\in[m]$ are determined by a label function $f:\mathbb{R}\rightarrow\{-1,1\}$ as follows:
\begin{align}\label{eq:gen_model}
y_i=f(\ab_i^T\x_0),\;\;~i\in[m],
\end{align}
where $\ab_i$'s are known measurement vectors with i.i.d. Gaussian entries; and $\x_0\in\R^n$ is an unknown vector of coefficients. Some  popular examples for the  label function $f$ include the following:
 \\
 \\
\noindent$\bullet~${\emph{(Noisy) Signed}}: $y_i=
 \begin{cases}  \sign(\ab_i^T\x_0) &, \text{w.p.}~~1-\eps, \\  -\sign(\ab_i^T\x_0) &, \text{w.p.}~~\eps,  \end{cases}
~~~~\text{where}~~\eps \in [0,1/2]. 
$ 
\\
\\
\noindent$\bullet~${\emph{Logistic}}: 
$
y_i = \begin{cases}  +1 &, \text{w.p.}~~\frac{1}{1+\exp({-\ab_i^T\x_0})}, \\  -1 &, \text{w.p.}~~1- \frac{1}{1+\exp({-\ab_i^T\x_0})}.  \end{cases}
$
\\

\noindent$\bullet~${\emph{Probit}}:
$
y_i = \begin{cases}  +1 &, \text{w.p.}~~\Phi(\ab_i^T\x_0), \\  -1 &, \text{w.p.}~~1-\Phi(\ab_i^T\x_0).  \end{cases}
$
\\

\noindent\textbf{Loss function.}~We study the recovery performance of estimates $\xh_{\ell}$ of $\x_0$ that are obtained by solving \eqref{eq:gen_opt}
for proper convex loss functions $\ell:\R\rightarrow\R$. Different choices for $\ell$ lead to popular specific estimators including the following:
\\
\\
\noindent$\bullet~${\emph{Least Squares (LS):}} $\ell(t)=\frac{1}{2}(t-1)^2$,

\noindent$\bullet~${\emph{Least-Absolute Deviations (LAD):}} $\ell(t)=|t-1|$,

\noindent$\bullet~${\emph{Logistic Loss:}} $\ell(t)=\log(1+\exp({-t}))$,

\noindent$\bullet~${\emph{Exponential Loss:}} $\ell(t)=\exp({-t})$,

\noindent$\bullet~${\emph{Hinge Loss:}} $\ell(t)=\max\{1-t\,,\,0\}$.
\\
\\
\noindent\textbf{Performance measure.}~We measure 
performance of the estimator $\xh_{\ell}$ by the value of its correlation to $\x_0$, i.e., 
\bea\label{eq:corr}
\corr{\xh_\ell}{\x_0}:=\frac{\,{\inp{\xh_\ell}{\x_0}}\,}{\|\xh_\ell\|_2 \|\x_0\|_2} \in [-1,1].
\eea

Obviously, we seek estimates that maximize correlation. While correlation is the measure of primal interest, our results extend rather naturally to other parameter estimation metrics, such as square error, as well as prediction metrics, such as classification error.

\vp
\noindent\textbf{Model assumptions.}~All our results are valid under the assumption that the measurement vectors have i.i.d. Gaussian entries. 

\begin{ass}[Gaussian feature vectors]\label{ass:Gaussian}
 The vectors \,$\ab_i\in\R^n,\,i\in[m]$ have entries i.i.d.\,\,standard normal.
\end{ass}
\noindent We further assume that $\|\x_0\|_2=1$. This assumption is without loss of generality since the norm of $\x_0$ can always be absorbed in the link function. Indeed, letting $\|\x_0\|_2=r$, we can always write the measurements as $f(\ab^T\x_0) = \widetilde{f}\big(\ab^T\widetilde{\x}_0\big)$, where $\widetilde{\x}_0=\x_0/r$ (hence, $\|\widetilde{\x}_0\|_2=1$) and $\widetilde{f}(t) = f\big(r t\big)$.
We make no further assumptions on the distribution of the true vector $\x_0$.

\subsection{Contributions and organization}\label{sec:over}
This paper's main contributions are summarized below.

\vp
\noindent$\bullet$~\textbf{Sharp asymptotics}:~We show that the absolute value of correlation of $\xh_\ell$ to the true vector $\x_0$ is sharply predicted by $\sqrt{1/(1+\sig_{\ell}^2)}$ where the "effective noise" parameter $\sig_{\ell}$ can be explicitly computed by solving a system of three non-linear equations in three unknowns. We find that the system of equations (and thus, the value of $\sig_{\ell}$) depends on the loss function $\ell$ through its Moreau envelope function. Our prediction holds in the 
linear asymptotic regime in which $m,n\rightarrow\infty$ and $m/n\rightarrow \delta >1$. See Section \ref{sec:gen}.
\begin{figure*}[t!]
\centering
	\begin{subfigure}{0.45\textwidth}
		\centering
    		\includegraphics[width=6.7cm, height=5.5cm]{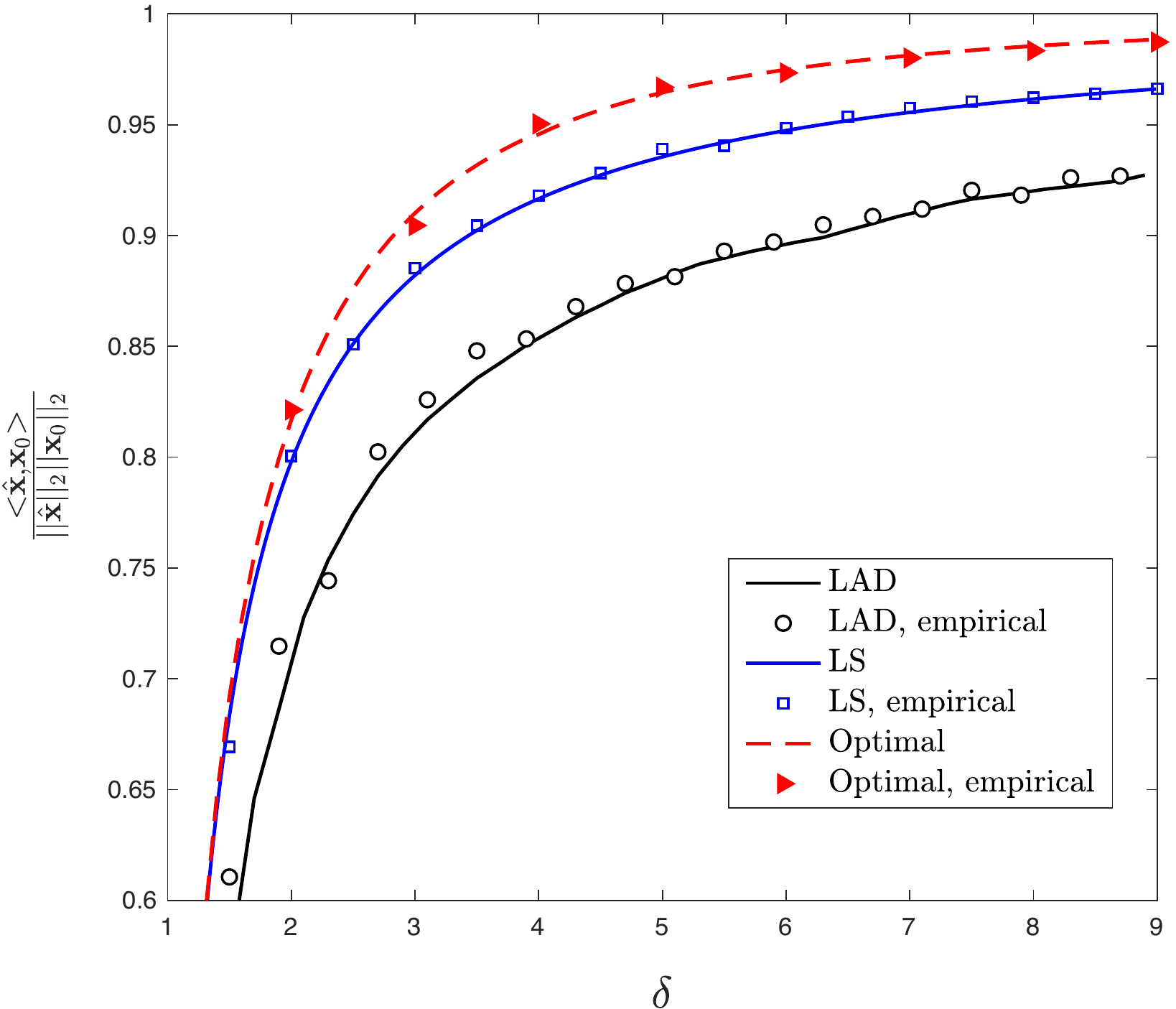}
    		\caption{{\footnotesize}}
    		\label{fig:figure1}
    \end{subfigure}
    	\begin{subfigure}{0.45\textwidth}
		\centering
    		\includegraphics[width=6.7cm, height=5.5cm]{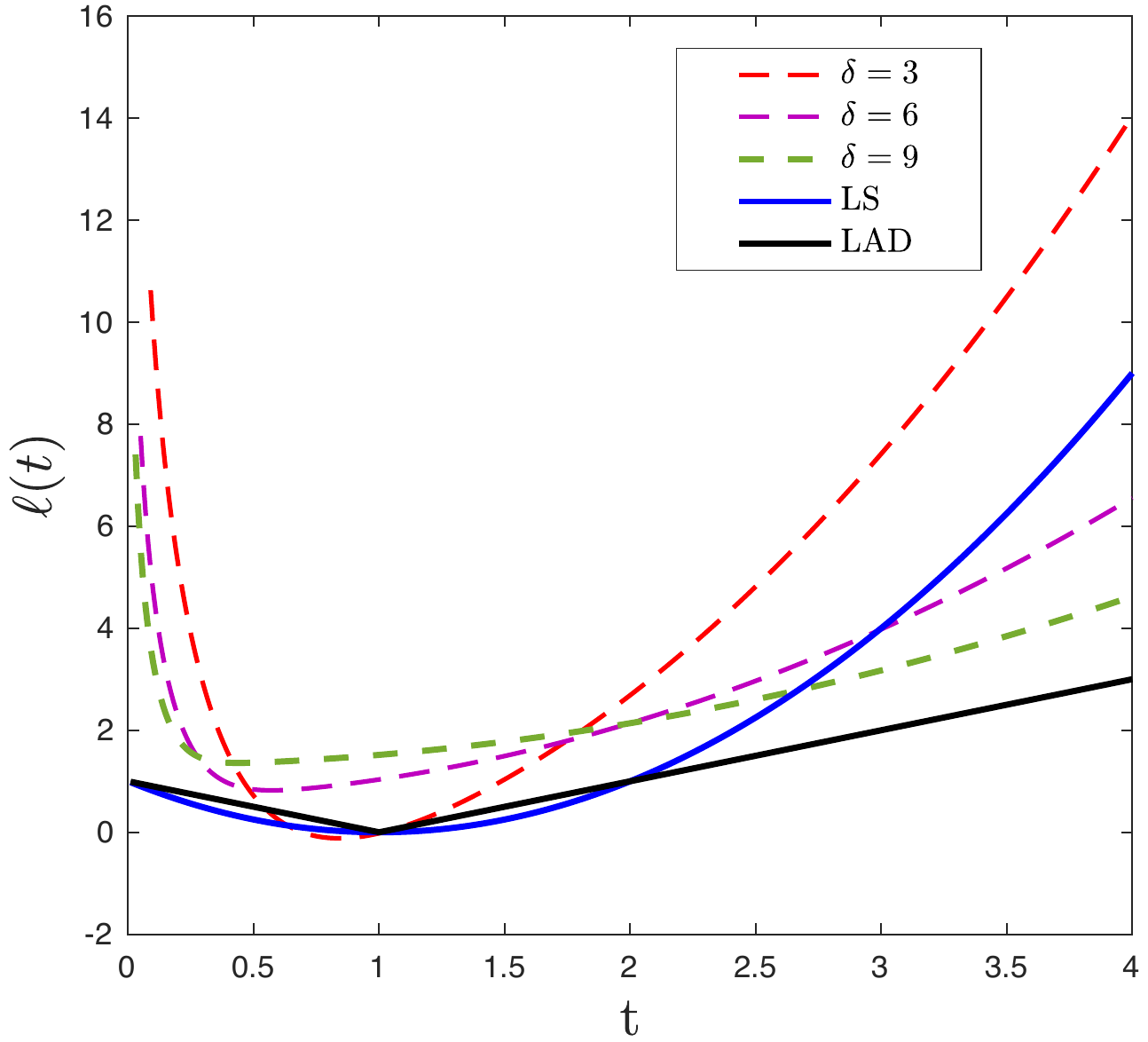}
    		\caption{{\footnotesize}}
    		\label{fig:optimal_loss}
    \end{subfigure}
    
    \caption{{\footnotesize Left: Comparison between analytical (solid lines) and empirical (markers) performance for least-squares (LS) and least-absolute deviations (LAD), along with optimal performance (dashed line) as predicted by the upper bound of Theorem \ref{sec:lem} for the Signed model($\eps=0$). The red markers depict the empirical performance of the optimal loss function that attains the upper bound.
     Right: Illustrations of optimal loss functions for the Signed model for different values of $\delta$ according to Theorem \ref{thm:opt_loss}.}} 
\label{fig:intro}
    \end{figure*}
\begin{footnotesize}
\begin{table*}
\centering
\begin{tabular}{c | c c c c c c c c}
$\delta$ & $2$ & $3$ & $4$ & $5$ & $6$  & $7$ & $8$ & $9$ \\ \hline
\text{\footnotesize Predicted Performance} & $0.8168$ & $0.9101$ & $0.9457$ & $0.9645$ & $0.9748$ & $0.9813$ & $0.9855$ & $0.9885$ \\ \hline
\text{\footnotesize Empirical Performance} & $0.8213$ & $0.9045$ & $0.9504$ & $0.9669$ & $0.9734$ & $0.9801$ & $0.9834$ & $0.9873$
\end{tabular}
\caption{\footnotesize Analytical predictions and empirical performance of the optimal loss function for Signed model. Empirical results are averaged over 20 independent experiments for $n=128$.}
\label{tab:1}
\end{table*}
\end{footnotesize}
\vp
\\
\noindent$\bullet$~\textbf{Fundamental limits}:~We establish fundamental limits on the performance of convex optimization-based estimators by computing an upper bound on the best possible correlation performance among all convex  loss functions. We compute the upper bound by solving a certain nonlinear equation and we show that such a solution exists for all $\delta>1$. See Section \ref{sec:upper}.

\vp
\noindent$\bullet$~\textbf{Optimal performance}:~For certain models including Signed and Logistic, we find the loss functions that achieve the optimal performance, i.e., they attain the previously derived upper bound. See Section \ref{sec:opt_func}.

\vp
\noindent$\bullet$~\textbf{Optimality of LS}:~For binary logistic and sigmoid models, we prove that the correlation performance of least-squares (LS) is at least as good as 0.9972 and 0.9804 times the optimal performance. See Section \ref{sec:LS}.


%
\vp
\noindent$\bullet$~\textbf{Numerical simulations}:~We specialize our results on general models and loss functions to popular instances, for which we provide simulation results that demonstrate the accuracy of the theoretical predictions. See Section \ref{sec:numerical}.

\vp
Figure \ref{fig:intro} contains a pictorial preview of our results described above for the special case of Signed measurements. First, Figure \ref{fig:figure1} depicts the correlation performance of LS and LAD estimators as a function of the aspect ratio $\delta$. Both theoretical predictions and numerical results are shown; note the close match for even small dimensions. Second, the dashed line on the same figure shows the upper bound derived in this paper -- there is no convex loss function that results in correlation exceeding this line. Third, we show that the upper bound can be achieved by the loss functions depicted in Figure \ref{fig:optimal_loss} for several values of $\delta$. We solve \eqref{eq:gen_opt} for this choice of loss functions using gradient descent and numerically evaluate the achieved correlation performance. The recorded values are compared in Table \ref{tab:1} to the corresponding values of the upper bound; again, note the close agreement between the values as predicted by the findings of this paper. We present corresponding results for the Logistic and Probit models in Section \ref{sec:numerical} and for the Noisy-signed model in Appendix \ref{sec:noisysigned}. 
%
\\
\subsection{Related work}\label{sec:prior}
Over the past two decades there has been a  long list of works that derive statistical guarantees for high-dimensional estimation problems. Many of these are concerned with convex optimization-based inference methods. Our work is most closely related to the following three lines of research.

\vp
\noindent\emph{(a)\,Sharp asymptotics for linear measurements.} Most of the results in the literature of high-dimensional statistics are order-wise in nature. Sharp asymptotic predictions have only more recently appeared in the literature for the case of noisy linear measurements with Gaussian measurement vectors. There are by now three different approaches that have been used  towards asymptotic analysis of convex regularized estimators: \\
i) the one that is based on the approximate message passing (AMP) algorithm and its state-evolution analysis, e.g., \cite{AMP,DMM,bayati2011dynamics,montanariLasso,donoho2016high,bu2019algorithmic,mousavi2018consistent}.\\
ii) the one that is based on Gaussian process (GP) inequalities, specifically the convex Gaussian min-max Theorem (CGMT) e.g., \cite{StoLASSO,OTH13,COLT,Master,TSP18,miolane2018distribution}.\\
iii) the ``leave-one-out" approach \cite{karoui2013asymptotic,karoui15}. The three approaches are quite different to each other and each comes with its unique distinguishing features and disadvantages. A detailed comparison is beyond our scope.

Our results in Theorems \ref{sec:lem} and \ref{thm:opt_loss} for achieving the best performance across all loss functions is complementary to \cite[Theorem 1]{bean2013optimal} and \cite{advani2016statistical} in which the authors also proposed a method for deriving optimal loss function and measuring its performance, albeit for \emph{linear} models. Instead, we study binary models. The optimality of regularization for linear measurements, is recently studied in \cite{celentano2019fundamental}. 

In terms of analysis, we follow the GP approach and build upon the CGMT. 
Since the previous works are concerned with linear measurements, they consider estimators that solve minimization problems of the form
\bea\label{eq:gen_opt2}
\xh := \arg\min_\x \sum_{i=1}^{m} \widetilde{\ell}(y_i-\ab_i^T\x) + r R(\x)
\eea
Specifically, the loss function $\widetilde\ell$ penalizes the residual. In this paper, we show that the CGMT is applicable to optimization problems in the form of \eqref{eq:gen_opt}. For our case of binary observations, \eqref{eq:gen_opt} is more general than \eqref{eq:gen_opt2}. To see this, note that for $y_i\in\pm 1$ and popular symmetric loss functions $\widetilde{\ell}(t)=\widetilde{\ell}(-t)$, e.g. least-squares (LS), \eqref{eq:gen_opt} results in \eqref{eq:gen_opt2} by choosing $\ell(t)=\widetilde\ell(t-1)$ in the former. Moreover, \eqref{eq:gen_opt} includes several other popular loss functions such as the logistic loss and the hinge-loss which cannot be expressed by \eqref{eq:gen_opt2}.


\vp
\noindent\emph{(b)\,One-bit compressed sensing.} Our work naturally relates to the literature on one-bit compressed sensing (CS) \cite{boufounos20081}. The vast majority of performance guarantees for one-bit CS are order-wise in nature, e.g., \cite{jacques2013robust,plan2013one,plan2012robust,PV15,genzel2017high,xu2018taking}. To the best of our knowledge, the only existing sharp results are presented in \cite{NIPS} for Gaussian measurement vectors, which studies the  asymptotic performance of regularized LS. Our work can be seen as a direct extension of \cite{NIPS} to loss functions beyond least-squares; see Section  \ref{sec:LS} for details.

Similar to the generality of our paper, \cite{genzel2017high} also studies the high-dimensional performance of general loss functions. However, in contrast to our results, their performance bounds are loose (order-wise); as such, they are not informative about the question of optimal performance which we also address here.

\vp
\noindent\emph{(c)\,Classification in high-dimensions.} In \cite{candes2018phase,sur2019modern} the authors study the high-dimensional performance of maximum-likelihood (ML) estimation for the logistic model. The ML estimator is a special case of \eqref{eq:gen_opt} and we consider general binary models. Also, their analysis is based on the AMP. The asymptotics of logistic loss under different classification models has also been recently studied in \cite{mai2019large}. In yet another closely related recent work \cite{salehi2019impact}, the authors extend the results of \cite{sur2019modern} to regularized ML by using the CGMT. Instead, we present results for general loss functions and for general measurement models. Importantly, we also study performance bounds and optimal loss functions. 

Finally, we remark on the following closely related parallel works. While this paper was being under review, the CGMT has been applied by \cite{montanari2019generalization} and \cite{deng2019model} to determine the generalization performance of max-margin linear classifiers in a binary classification setting. In essence, these results are complementary to the results of our paper in the following sense. Consider a binary classification setting under the logistic model and Gaussian regressors. As discussed in Section \ref{sec:logloss}, the optimal set of  \eqref{eq:gen_opt} is bounded with probability approaching one if and only if $\delta>\delta_f^\star$, for appropriate threshold $\delta_f^\star$ determined for first time in \cite{candes2018phase} (see also Figure \ref{fig:figure2}). Our results hold in this regime. In contrast, the papers \cite{deng2019model} and \cite{montanari2019generalization} study the regime $\delta<\delta_f^\star$.	
Also a preliminary version of the results of this paper was published in \cite{taheri2019sharp}.

%

%
\section{Sharp performance guarantees}\label{sec:gen}

\subsection{Definitions}
\noindent\textbf{Moreau envelopes.}~Before stating the first result we need a definition. We write
$$\env{\ell}{x}{\la}:=\min_{v}\frac{1}{2\la}(x-v)^2 + \ell(v),$$
for the \emph{Moreau envelope function} of the loss $\ell:\R\rightarrow\R$ at $x$ with parameter $\la>0$. The minimizer (which is unique by strong convexity) is known as the \emph{proximal operator} of $\ell$ at $x$ with parameter $\la$ and we denote it as $\prox{\ell}{x}{\la}$.
%
A useful property of the Moreau envelope function is that it is continuously differentiable with respect to both $x$ and $\la$ \cite{rockafellar2009variational}. We denote these derivatives as follows
\bea\nn
\envdx{\ell}{x}{\la}&:=\frac{\partial{\env{\ell}{x}{\la}}}{\partial x},\nn\\
\envdla{\ell}{x}{\la}&:=\frac{\partial{\env{\ell}{x}{\la}}}{\partial \la}.\nn
\eea


\noindent\textbf{A system of equations.}~As we show shortly, the asymptotic performance of the optimization in \eqref{eq:gen_opt} is tightly connected to the solution of a certain system of nonlinear equations, which we introduce here. 
Specifically, define random variables $G,S$ and $Y$ as follows:
\bea\label{eq:GSY}
G,S\overset{\tiny \text{i.i.d.}}{\sim}\Nn(0,1) \quad\text{and}\quad Y=f(S),
\eea
and consider the following system of non-linear equations in three unknowns $(\mu,\alpha\geq0,\la \geq 0)$:
\begin{subequations}\label{eq:eq_main}
\bea
 \Exp\bigg[Y\, S \cdot\envdx{\ell}{\ourx}{\la}  \bigg]&=0 , \label{eq:mu_main}\\
 {\la^2}\,{\delta}\,\Exp\bigg[\,\left(\envdx{\ell}{\ourx}{\la}\right)^2\,\bigg]&=\alpha^2 ,
\label{eq:alpha_main}\\
\lambda\,\delta\,\E\bigg[ G\cdot \envdx{\ell}{\ourx}{\la}  \bigg]&=\alpha.
\label{eq:lambda_main}
\eea
\end{subequations}
The expectations are with respect to the randomness of the random variables $G$, $S$ and $Y$. We remark that the equations are well defined even if the loss function $\ell$ is not differentiable. In Section \ref{sec:ME} we summarize some well-known properties of the Moreau Envelope function and use them to simplify \eqref{eq:eq_main} for differentiable loss functions. 

%


\subsection{Asymptotic prediction}
We are now ready to state our first main result.
\begin{thm}[Sharp asymptotics]\label{thm:main}
 Let Assumption \ref{ass:Gaussian} hold and assume $\delta>1$ such that the set of minimizers in \eqref{eq:gen_opt} is bounded and the system of equations \eqref{eq:eq_main} has a unique solution $(\mu,\al\geq 0,\la\geq 0)$, such that $\mu\neq 0$.  Let $\xh_\ell$ be as in \eqref{eq:gen_opt}. Then, in the limit of $m,n\rightarrow+\infty$, $m/n\rightarrow\delta$, it holds with probability one that

\bea\label{eq:corr_thm}
\lim_{n\rightarrow \infty} \corr{\xh_\ell}{\x_0} = \frac{\mu}{\sqrt{\mu^2+\al^2}}.
\eea

\noindent Moreover, 
\bea\label{eq:norm_thm}
\lim_{n\rightarrow \infty} \left\|\xh_\ell-\mu\cdot\frac{\x_0}{\|\x_0\|_2}\right\|_2^2 = \al^2.
\eea
\end{thm}

%

Theorem \ref{thm:main} holds for general loss functions. In Section \ref{sec:cases} we specialize the result to specific popular choices and also present numerical simulations that confirm the validity of the predictions (see Figures. \ref{fig:figure1}--\ref{fig:figure3} and \ref{fig:eppointone}--\ref{fig:eppointtwentyfive}). Before that, we include a few remarks on the conditions, interpretation and implications of the theorem. The proof is deferred to Appendix \ref{sec:proof} and uses the convex Gaussian min-max theorem (CGMT) \cite{COLT,Master}.

\begin{remark}[The role of $\mu$ and $\alpha$]
According to \eqref{eq:corr_thm}, the prediction for the limiting behavior of the correlation value is given in terms of  an effective noise parameter $\sigma_\ell:=\alpha\big/\mu$, where $\mu$ and $\alpha$ are unique solutions of \eqref{eq:eq_main}. \emph{The smaller the value of $\sigma_\ell$ is, the larger becomes the correlation value.} While the correlation value is fully determined by the ratio of $\alpha$ and $\mu$, their individual role is clarified in \eqref{eq:norm_thm}. Specifically, according to \eqref{eq:norm_thm}, $\xh_{\ell}$ is a biased estimate of the true $\x_0$ and $\mu$ represents exactly that bias term. In other words, solving \eqref{eq:gen_opt} returns an estimator that is close to a $\mu$--scaled version of $\x_0$. When $\x_0$ and $\xh_{\ell}$ are scaled appropriately, the $\ell_2$-norm of their difference converges to $\alpha$.
\end{remark}

\begin{remark}[Why $\delta>1$] The theorem requires that $\delta>1$ (equivalently, $m>n$). Here, we show that this condition is \emph{necessary} for the equations \eqref{eq:eq_main} to have a bounded solution.
To see this, take squares in both sides of \eqref{eq:lambda_main} and divide by \eqref{eq:alpha_main}, to find that 
$$
\delta = \frac{\Exp\left[\,\left(\envdx{\ell}{\ourx}{\la}\right)^2\,\right]}{\left(\E\left[ G\cdot \envdx{\ell}{\ourx}{\la} \right]\right)^2} \geq 1.
$$
The inequality follows by applying Cauchy-Schwarz and using the fact that $\E[G^2]=1$.
\end{remark}
\begin{remark}[On the existence of a solution to \eqref{eq:eq_main}]\label{rem:delta_f}
 While $\delta>1$ is a necessary condition for the equations in \eqref{eq:eq_main} to have a solution, it is \emph{not} sufficient in general. This depends on the specific choice of the loss function. For example, in Section \ref{sec:LS}, we show that for the squared loss $\ell(t)=(t-1)^2$, the equations have a unique solution iff $\delta>1$. On the other hand, for  logistic-loss and  hinge-loss, it is argued in Section \ref{sec:logloss} that there exists a threshold value $\delta^\star_{f}>2$ such that the set of minimizers in \eqref{eq:gen_opt} is unbounded if $\delta<\delta^\star_{f}$. In this case, Theorem \ref{thm:main} does not hold. 
 We conjecture that for these choices of loss, the equations \eqref{eq:eq_main} are solvable iff  $\delta>\delta^\star_{f}$. Justifying this conjecture
and further studying more general sufficient and necessary conditions under which the equations \eqref{eq:eq_main} admit a solution is left to future work. However, in what follows, given such a solution, we prove that it is unique for a wide class of convex-loss functions of interest.
\end{remark}


\begin{remark}[On the uniqueness of solution to \eqref{eq:eq_main}]\label{rem:unique}
We show that if the system of equations in \eqref{eq:eq_main} has a solution, then it is unique provided that $\ell$ is strictly convex, continuously differentiable and its derivative satisfies $\elld(0)\neq 0$. For instance, this class includes the square, the logistic and the exponential losses. However, it excludes non-differentiable functions such as the LAD and hinge-loss. We believe that the differentiability assumption can be relaxed without major modification in our proof, but we leave this for future work. Our result is summarized in Proposition \ref{propo:Unique} below.
\begin{propo}\label{propo:Unique}
 Assume that the loss function $\ell:\R\rightarrow\R$ has the following properties: (i) it is proper strictly convex; (ii) it is continuously differentiable and its derivative $\ellp$ is such that $\ellp(0)\neq 0$. Further assume that the (possibly random) link function $f$ is such that $SY=Sf(S),~S\sim\Nn(0,1)$ has strictly positive density on the real line. The following statement is true. For any $\delta>1$, if the system of equations in \eqref{eq:eq_main} has a bounded solution, then it is unique. 
\end{propo}

\noindent The detailed proof of Proposition \ref{propo:Unique} is deferred to Appendix \ref{sec:unique}. Here, we highlight some key ideas. The CGMT relates --in a rather natural way--  the original ERM optimization \eqref{eq:gen_opt} to the following deterministic min-max optimization on four variables
\bea\label{eq:DO}
\min_{\alpha>0,\mu,\tau>0}~\max_{\gamma>0}~F(\alpha,\mu,\tau,\gamma) := \frac{\gamma\tau}{2}-\frac{\alpha\gamma}{\sqrt{\delta}}+ \mathbb{E}\left[\env{\ell}{\alpha G+\mu YS}{\frac{\tau}{\gamma}} \right].
\eea
In Appendix \ref{sec:B4}, we show that the optimization above is convex-concave for any lower semi-continuous, proper, convex function $\ell:\R\rightarrow\R$. Moreover, it is shown that one arrives at the system of equations in \eqref{eq:eq_main} by simplifying the first-order optimality conditions of the min-max optimization in \eqref{eq:DO}. This connection is key to the proof of Proposition \ref{propo:Unique}. Indeed, we prove uniqueness of solution (if such a solution exists) to \eqref{eq:eq_main}, by proving instead that the function $F(\alpha,\mu,\tau,\gamma)$ above is (jointly) \emph{strictly} convex in $(\alpha,\mu,\tau)$ and \emph{strictly} concave in $\gamma$, provided that $\ell$ satisfies the conditions of the proposition. Next, let us briefly discuss how strict convex-concavity of \eqref{eq:DO} can be shown. For concreteness, we only discuss strict convexity here; the ideas are similar for strict concavity. At the heart of the proof of strict convexity of $F$ is understanding the properties of the \emph{expected Moreau envelope function} $\Omega:\R_+\times\R\times\R_+\times\R_+\rightarrow\R$ defined as follows:
$$\Omega(\al,\mu,\tau,\gamma):=\E\left[\env{\ell}{\al G + \mu YS}{\frac{\tau}{\gamma}}\right].$$
Specifically, we prove in Proposition \ref{propo:EME_sum} in Appendix \ref{sec:strict_EME} that  if $\ell$ is strictly convex, differentiable and does not attain its minimum at $0$, then $\Omega$ is strictly convex in $(\alpha,\mu,\tau)$ and strictly concave in $\gamma$. It is worth noting that the Moreau envelope function $\env{\ell}{\al g + \mu ys}{\tau}$ for fixed $g,s$ and $y=f(s)$ is \emph{not} necessarily strictly convex. Interestingly, we show that the \emph{expected} Moreau envelope has this desired feature. We refer the reader to Appendices \ref{sec:strict_EME} and   \ref{sec:unique} for more details.
\end{remark}

\section{On optimal performance}\label{sec:opt}

\subsection{Fundamental limitations}\label{sec:upper}
In this section, we establish fundamental limits on the performance of \eqref{eq:gen_opt} by deriving an upper bound on the absolute value of correlation $\corr{\xh_\ell}{\x_0}$ that holds for \emph{all} choices of  loss functions satisfying Theorem \ref{thm:main}. The result builds on the  prediction of Theorem \ref{thm:main}.
In view of \eqref{eq:corr_thm} upper bounding correlation is equivalent to lower bounding the effective noise parameter $\sigma_\ell=\alpha/\mu$. Theorem \ref{sec:lem} below derives such a lower bound. 
%

For a random variable $H$ with density $p_H(h)$ that has a derivative $p_H^{\prime}(h), \forall h\in\mathbb{R}$, we denote its score function $\ksi_H(h):=\frac{\partial}{\partial h}{\log p_H(h)}=\frac{p_H^{\prime}(h)}{p_H(h)}$. Then, the Fisher information of $H$  is defined as follows (e.g. \cite[Sec.~2]{barron1984monotonic}):
$$\Ic(H):=\Exp\Big[\,(\ksi_H(H))^2\,\Big].$$

\begin{thm}[Best achievable performance]\label{sec:lem} Let the assumptions and notation of Theorem \ref{thm:main} hold and recall the definition of random variables $G,S$ and $Y$ in \eqref{eq:GSY}. For $\sigma > 0$, define a new random variable $W_\sigma := \sig G+SY,$ and the function $\kappa:(0,\infty]\rightarrow[0,1]$ as follows,
\begin{align*}
\kappa(\sigma):= \frac{\sig^2\left(\sig^2 \Ic(W_\sig)+\Ic(W_\sig)-1\right)}{1+\sig^2\left(\sig^2\Ic(W_\sig)-1\right)}.
\end{align*}
%
Further define $\sigma_{\rm{opt}}$ as follows,
\bea \label{eq:lem}
\sig_{\rm{opt}} := \min\left\{\sig\geq 0: \kappa(\sigma)  = \frac{1}{\delta}\right\}.
\eea
Then, for $\sigma_{\ell}:= \frac{\alpha}{\mu}$ it holds that $\sigma_\ell\geq \sig_{\rm{opt}}$.
%
\end{thm}

The theorem above establishes an upper bound on the best possible correlation performance among all convex loss functions. In Section \ref{sec:opt_func}, we show that this bound is often tight, i.e. there exists a loss function that achieves the specified best possible performance.

\begin{remark} Theorem \ref{sec:lem} complements the results of \cite{bean2013optimal}, \cite[Lem.~3.4]{donoho2016high} and \cite[Rem.~5.3.3]{Master} in which they consider only linear measurements. In particular, Theorem \ref{sec:lem} shows that it is possible to achieve results of this nature for the more challenging setting of binary observations considered here. 
\end{remark}

\begin{proof}[{Proof of Theorem \ref{sec:lem}}]
Fix a loss function $\ell$ and let $(\mu\neq 0,\alpha>0,\la\geq 0)$ be a solution to \eqref{eq:eq_main}, which by assumptions of Theorem \ref{thm:main} is unique. The first important observation is that the error of a loss function is unique up to a multiplicative constant. To see this, consider an arbitrary loss function $\ell(t)$ and let $\xh_\ell$ be a minimizer in \eqref{eq:gen_opt}. Now consider \eqref{eq:gen_opt} with the following loss function instead, for some arbitrary constants $C_1>0, C_2\neq 0$: 
\begin{align}\label{eq:irrelevant}
\widehat\ell(t):=\frac{1}{C_1}\ell\big(C_2{t}\big).
\end{align}
 It is not hard to see that $\frac{1}{C_2}\,\xh_\ell$ is the minimizer for $\widehat\ell$. Clearly, $\frac{1}{C_2}\xh_\ell$ has the same correlation value with $\x_0$ as $\xh_\ell$, showing that the two loss functions $\ell$ and $\widehat\ell$ perform the same.
With this observation in mind, consider the function $\widehat\ell:\mathbb{R}\rightarrow\mathbb{R}$ such that $\widehat\ell(t)= \frac{\la}{\mu^2}\ell(\mu\,t)$.  Then, notice that
$$
\envdx{\ell}{x}{\la} = \frac{\mu}{\la}\envdx{\widehat\ell}{x/\mu}{1}.
$$
Using this relation in \eqref{eq:eq_main} and setting $\sig:=\sig_{\ell}=\alpha/\mu$, the system of equations in \eqref{eq:eq_main} can be equivalently rewritten in the following convenient form, 
\begin{subequations}\label{eq:eq_main2}
\bea
&\Exp\bigg[Y\, S \cdot\envdx{\widehat\ell}{W_{\sigma} }{1} \bigg]=0 , \label{eq:mu_main2}\\
&\Exp\bigg[\,\left(\envdx{\widehat\ell}{W_{\sigma} }{1}\right)^2\,\bigg]=\sig^2/\delta\,,
\label{eq:alpha_main2}\\
&\E\bigg[ G\cdot \envdx{\widehat\ell}{W_{\sigma} }{1}  \bigg]=\sig/\delta\,.
\label{eq:lambda_main2}
\eea
\end{subequations}
Next, we show how to use \eqref{eq:eq_main2} to derive an equivalent system of equations based on $W_\sig$. Starting with \eqref{eq:lambda_main2} we have
\begin{align}\label{eq:int3}
\E\bigg[ G\cdot \envdx{\widehat\ell}{W_{\sigma} }{1}  \bigg] = \frac{1}{\sig} \iint u\,\envdx{\widehat\ell}{ u+z}{1}  \phi_{\sig}(u)p_{SY}(z)\mathrm{d}u\mathrm{d}z,
\end{align}
where $\phi_{\sig}(u):=p_{\sig G}(u)= \frac{1}{\sig\sqrt{2\pi}}e^{-\frac{u^2}{2\sig^2}}$. Since it holds that $ \phi_{\sig}(u)=\frac{-\sig^2}{u}\phi^\prime_{\sig}(u) $, using \eqref{eq:int3} it follows that 
\begin{align}\label{eq:int4}
\begin{split}
\E\bigg[ G\cdot \envdx{\widehat\ell}{W_{\sigma} }{1}  \bigg]  &= -\sig\iint\envdx{\widehat\ell}{ u+z}{1}  \phi^\prime_{\sig}(u)p_{SY}(z)\mathrm{d}u\mathrm{d}z \\
&= -\sig\iint\envdx{\widehat\ell}{ w}{1}  \phi^\prime_{\sig}(u)p_{SY}(w-u)\mathrm{d}u\mathrm{d}w  =  -\sig\int\envdx{\widehat\ell}{ w}{1}  p^\prime_{W_\sig}(w)\mathrm{d}w, 
\end{split}
\end{align}
where in the last step we used 
$$
p_{W_\sig}^{\prime}(w) = \int \phi_{\sig}^{\prime}(u) p_{SY}(w-u)\, \mathrm{d}u.
$$
Therefore we have by \eqref{eq:int4} that
\bea\label{eq:SteinIBP1}
\E\bigg[ G\cdot \envdx{\widehat\ell}{W_{\sigma} }{1}  \bigg] = -\sigma \,\E\bigg[\envdx{\widehat{\ell}}{W_{\sigma} }{1}\ksi_{W_{\sig}}(W_{\sig})\bigg].
\eea
This combined with \eqref{eq:lambda_main2} gives $\E\left[\envdx{\widehat{\ell}}{W_{\sigma} }{1}\ksi_{W_{\sig}}(W_{\sig})\right]=-1/\delta.$
Second, multiplying \eqref{eq:lambda_main2} with $\sigma^2$ and adding it to \eqref{eq:mu_main2} yields that,
\begin{align}\label{eq:w}
 \Exp\left[W_\sig \cdot\envdx{\widehat\ell}{W_{\sig}}{1} \right]&=\sig^2/\delta ,
\end{align}
Putting these together we conclude with the following system of equations which is equivalent to \eqref{eq:eq_main2},
\begin{subequations}\label{eq:eq_main4}
\bea
&\Exp\bigg[W_\sig \cdot\envdx{\widehat\ell}{W_\sig}{1} \bigg]=\sig^2/\delta\,, \label{eq:mu_main4}\\
&\Exp\bigg[\,\left(\envdx{\widehat\ell}{W_\sig }{1}\right)^2\,\bigg]=\sig^2/\delta\,,
\label{eq:alpha_main4}\\
&\E\bigg[\envdx{\widehat{\ell}}{W_\sig}{1}\ksi_{W_\sig}(W_\sig)\bigg]=-1/\delta\,.
\label{eq:lambda_main4}
\eea
\end{subequations}
Note that for $\sig > 0$, $\ksi_{W_\sig}=p'_{W_\sig}/p_{W_\sig}$ exists everywhere. This is because for all $w\in\mathbb{R}$: $p_{W_\sig}(w) >0$ and $p_{W_{\sig}}(\cdot)$ is continuously differentiable. 
Combining \eqref{eq:mu_main4} and \eqref{eq:lambda_main4} we derive the following equation which holds for $\alpha_1, \alpha_2 \in \mathbb{R}$,
\begin{align*}
 \Exp\left[(\alpha_1W_\sig+\alpha_2\ksi_{W_\sig}(W_\sig)) \cdot\envdx{\widehat\ell}{W_\sig}{1} \right]&=\alpha_1\sig^2/\delta-\alpha_2/\delta.
\end{align*}
By Cauchy-Schwartz inequality we have that
\begin{align}\label{eq:cs}
 \left(\Exp\left[(\alpha_1W_\sig+\alpha_2\ksi_{W_\sig}(W_\sig)) \cdot\envdx{\widehat\ell}{W_\sig}{1} \right]\right)^2 
 \le\, \Exp\bigg[(\alpha_1W_\sig+\alpha_2\ksi_{W_\sig}(W_\sig))^2 \bigg]\Exp\bigg[\left(\envdx{\widehat\ell}{W_\sig}{1} \right)^2\bigg].
\end{align}
Using the fact that $\Exp[W_{\sig}\ksi_{W_\sig}(W_{\sig})]=-1$ (by integration by parts), $\Exp[(\ksi_{W_\sig}(W_\sig))^2]=\Ic(W_\sig)$, $\Exp[W_\sig^2] = \sig^2+1$ and \eqref{eq:alpha_main4}, the right hand side of \eqref{eq:cs} is equal to
\begin{align*}
\left(\alpha_1^2(\sig^2+1)+ \alpha_2^2\,\Ic(W_\sig)-2\alpha_1\alpha_2\right)\,\sig^2/\delta.
\end{align*}
Therefore, we have concluded with the following inequality for $\sigma$,
\begin{align}\label{eq:16}
\delta \sig^2\left(\alpha_1^2(\sig^2+1)+\alpha_2^2\,\Ic(W_\sig)-2\alpha_1\alpha_2\right) \ge (\alpha_1\sig^2-\alpha_2)^2,
\end{align}
which holds for all $\alpha_1,\alpha_2\in\R$. In particular, \eqref{eq:16} holds for the following choice of values for $\alpha_1$ and $\alpha_2$:
\begin{align*} 
\alpha_1=  \frac{1-\sig^2\Ic(W_\sig)}{\delta(\sig^2\Ic(W_\sig)+\Ic(W_\sig)-1)}, \;\;\;\;\; \alpha_2 = \frac{1}{\delta(\sig^2\Ic(W_\sig)+\Ic(W_\sig)-1)}.
\end{align*} 
(The choice above is motivated by the result of Section \ref{sec:opt_func}; see Theorem \ref{thm:opt_loss}).
Rewriting \eqref{eq:16} with the chosen values of $\alpha_1$ and $\alpha_2$ yields the following inequality,
\begin{align}\label{eq:ineq_sig}
\frac{1}{\delta}\le\frac{\sig^2(\sig^2\Ic(W_\sig)+\Ic(W_\sig)-1)}{1+\sig^2(\sig^2\Ic(W_\sig)-1)} = \kappa(\sig),
\end{align}
where in the right-hand side above, we recognize the function $\kappa$ defined in the theorem. 

Next, we use \eqref{eq:ineq_sig} to  show that $\sig_{\rm opt}$ defined in \eqref{eq:lem} yields a lower bound on the achievable value of $\sigma$. For the sake of contradiction, assume that $\sig < \sig_{\rm{opt}}$. By the above, $1/\delta\le \kappa(\sig)$. Moreover, by the definition of $\sig_{\rm{opt}}$ we must have that $1/\delta < \kappa(\sig).$ Since $\kappa(0) = 0$ and $\kappa(\cdot)$ is a continuous function we conclude that for some $\sig_1 \in (0,\sig)$ it holds that $\kappa(\sig_1)=1/\delta$. Therefore for $\sig_1 < \sig_{\rm{opt}}$ we have $\kappa(\sig_1)= 1/\delta$, which contradicts the definition of $\sig_{\rm{opt}}$. This proves that $\sig\geq\sig_{\rm{opt}}$, as desired. 

In order to complete the proof, it remains to show that the equation $\kappa(\sigma)=1/\delta$ admits a solution for all $\delta>1$. For this purpose, we use the continuous mapping theorem and the fact that fisher information is a continuous function \cite{Costa1985ANE}. Recall that for two independent and non-constant random variables it holds that $\Ic(X+Y) < \Ic(X)$ \cite[Eq. 2.18]{barron1984monotonic}. Since $G$ and $SY$ are independent random variables we find that $\Ic(\sig G+SY) < \Ic(SY)$ which implies that $\Ic(\sig G+SY)$ takes finite values for all values of $\sig$. Therefore,
$$
\lim_{\sig\rightarrow 0}\kappa(\sig) = \lim_{\sig\rightarrow 0} \frac{\sig^2\left(\sig^2 \Ic(W_\sig)+\Ic(W_\sig)-1\right)}{1+\sig^2\left(\sig^2\Ic(W_\sig)-1\right)} = 0.
$$ 
Furthermore $\sig^2\Ic(\sig G+SY)= \Ic(G+\frac{1}{\sig}SY)\rightarrow\Ic(G)=1$ when $\sig\rightarrow\infty$. Hence,
$$
\lim_{\sig\rightarrow \infty}\kappa(\sig) = \lim_{\sig\rightarrow \infty} \frac{\sig^2\left(\sig^2 \Ic(W_\sig)+\Ic(W_\sig)-1\right)}{1+\sig^2\left(\sig^2\Ic(W_\sig)-1\right)} = 1.
$$
Note that $\sig^2\Ic(\sig G+ SY)<\sig^2\Ic(\sig G) = 1$, which further yields that $\kappa(\sig)<1$ for all $\sig \ge 0$.
Finally since $\Ic(\cdot)$ is a continuous function, we deduce that range of $\kappa:\mathbb{R^+}\cup 0 \rightarrow\mathbb{R}$ is $[0,1)$, implying the existence of a solution to \eqref{eq:lem} for all $\delta>1$.
\end{proof}

\vp
\noindent\textbf{A useful closed-form bound on the best achievable performance:}~In general, determining $\sig_{\rm opt}$ requires computing the Fisher information of the random variable $\sigma G+SY$ for $\sigma>0$. If the probability distribution of $SY$ is continuously differentiable (e.g., logistic model; see Section \ref{sec:psy}), then we obtain the following simplified bound.
\begin{cor}[Closed-form lower bound on $\sig_{\rm opt}$]\label{cor:Stam}
Let $p_{SY}:\mathbb{R}\rightarrow\mathbb{R}$ be the probability distribution of $SY$. If $p_{SY}(x)$ is differentiable for all $x\in\mathbb{R}$, then,
\begin{align}\label{eq:cor_stam}
\sig_{\rm{opt}}^2\ge\frac{1}{(\delta-1)(\Ic(SY)-1)}\,.
\end{align}
\end{cor}
The proof of the corollary reveals that \eqref{eq:cor_stam} holds with equality when $SY$ is Gaussian. In Section \ref{sec:psy}, we compute $p_{SY}$ for the Logistic and the Probit models and  numerically show that it is close to the density of a Gaussian random variable.  Consequently, the lower bound of Corollary \ref{cor:Stam} is almost exact when measurements are obtained according to the Logistic and Probit models; see Figure \ref{fig:p_SY} in the appendix. 

\begin{proof}[Proof of Corollary \ref{cor:Stam}]
Based on Theorem \ref{sec:lem}, the following equation holds for $\sig=\sig_{\rm opt}$
$$
\frac{1}{\delta}=\kappa(\sigma),
$$
or equivalently, by rewriting the right-hand side,
\begin{align}\label{eq:frac}
\frac{1}{\delta}=1- \frac{1}{\frac{1}{1-\sig^2\Ic(W_\sig)}-\sig^2}\,.
\end{align}
Define the following function
\begin{align*}
h(x):=1-\frac{1}{\frac{1}{1-\sig^2 x}-\sig^2}\,.
\end{align*}
The function $h$ is increasing in the region $\mathcal{R}_{\sig}   = \{ z: z> \sig^{-2}-\sig^{-4} \}.$
According to Stam's inequality \cite{stam}, for two independent random variables $X$ and $Y$ with continuously differentiable $p_X$ and $p_Y$  it holds that
$$
\Ic(X+Y)\le\frac{\Ic(X)\cdot\Ic(Y)}{\Ic(X) + \Ic(Y)},
$$
where equality is achieved if and only if $X$ and $Y$ are independent Gaussian random variables. Therefore since by assumption $p_{SY}$ is differentiable on the real line, Stam's inequality yields 
\begin{align}\label{eq:eq_stam}
{\Ic(W_\sig)}={\Ic(\sig\,G+SY)}\le\frac{\Ic(\sig\,G)\cdot\Ic(SY)}{\Ic(\sig\,G)+\Ic(SY)}\,.
\end{align}
Next we prove that for all $\sig > 0$, both sides of \eqref{eq:eq_stam} are in the region $\mathcal{R}_\sig$. First, we prove that $\Ic(W_\sig)\in\mathcal{R}_\sig$. By Cramer-Rao bound (e.g. see \cite[Eq. 2.15]{barron1984monotonic}) for Fisher information of a random variable $X$, we have that $\Ic(X) \ge 1/(\text{Var}\left[X\right])$. Also for the random variable $W_\sig$, we know that $\text{Var}\left[W_\sig\right] = 1+\sig^2-(\Exp[SY])^2$, thus
\begin{align}\label{eq:corr1}
\Ic(W_\sig) \ge \frac{1}{1+\sig^2-(\E[SY])^2} .
\end{align}
Using the relation $(\Exp[SY])^2 \le \Exp[S^2]\Exp[Y^2] = 1$, one can check that the following inequality holds :
\bea\label{eq:corr2}
 \frac{1}{1+\sig^2-(\E[SY])^2} \ge \sig^{-2}-\sig^{-4}.
\eea
Therefore from \eqref{eq:corr1} and \eqref{eq:corr2} we derive that $\Ic(W_\sig)\in\Rc_\sig$ for all $\sig>0$. Furthermore by the inequality in \eqref{eq:eq_stam} and the definition of $\Rc_\sig$ it directly follows that for all $\sig > 0$
$$
\frac{\Ic(\sig\,G)\,\Ic(SY)}{\Ic(\sig\,G)+\Ic(SY)} \in \mathcal{R}_\sig\,.
$$
Finally noting that $h(\cdot)$ is increasing in $\Rc_\sig$, combined with \eqref{eq:eq_stam} we have
$$
\frac{1}{\delta}=h\left(\Ic(W_\sig)\right) \le h\left(\frac{\Ic(\sig\,G)\cdot\Ic(SY)}{\Ic(\sig\,G)+\Ic(SY)}\right),
$$
 which after using the relation $\Ic(\sig\,G) = \sig^{-2}$ and further simplification yields the inequality in the statement of the corollary.
\end{proof}

\subsection{On the optimal loss function}\label{sec:opt_func}

It is natural to ask whether there exists a loss function that attains the bound of Theorem \ref{sec:lem}. If such a loss function exists, then we say it is \emph{optimal} in the sense that it maximizes the correlation performance among all convex loss functions in \eqref{eq:gen_opt}.

Our next theorem derives a candidate for the optimal loss function, which we denote $\ell_{\rm{opt}}$. Before stating the result, we provide some intuition about the proof which builds on Theorem \ref{sec:lem}. The critical observation in the proof of Theorem \ref{sec:lem} is that the effective noise $\sigma_{\widehat\ell}$ of $\widehat\ell$ is minimized (i.e., it attains the value $\sig_{\rm opt}$) if the Cauchy-Schwartz inequality in \eqref{eq:cs} holds with equality. Hence, we seek $\widehat\ell = \ell_{\rm{opt}}$ so that  for some $c\in\mathbb{R}$, 
\begin{align}\label{eq:MorDer}
\envdx{{\ell}_{\rm{opt}}} {w} {1} = c(\alpha_1w+\alpha_2\cdot\ksi_{W_{\rm{opt}}}(w)).
\end{align}
By choosing $c=-1$, integrating and ignoring constants irrelevant to the minimization of the loss function, the previous condition is equivalent to the following $
\env{\ell_{\rm{opt}} }{w}{1} = -\alpha_1w^2/2 - \alpha_2 \log(p_{W_{\rm{opt}} }(w)).
$
It turns out that this condition can be ``inverted" to yield the following explicit formula for $\ell_{\rm{opt}}$ (see Proposition \ref{propo:inverse}) 
$
\ell_{\rm{opt}} (w) = -\env{\alpha_1q + \alpha_2 \log(p_{W_{\rm{opt}} })}{w}{1}.\label{eq:optloss0}
$
Of course, one has to properly choose $\alpha_1$ and $\alpha_2$ to make sure that this function satisfies the system of equations in \eqref{eq:eq_main4} with $\sigma=\sigma_{\rm opt}$. The correct choice is specified in the theorem below.


\begin{thm}[Optimal loss function]\label{thm:opt_loss}
Recall the definition of $\sig_{\rm{opt}}$ in \eqref{eq:lem}. Define the random variable $W_{\rm{opt}} :=\sig_{\rm{opt}}  \,G+ SY$ and let $p_{W_{\rm{opt}} }$ denote its density. Consider the following loss function $\ell_{\rm{opt}}:\R\rightarrow\R$
\begin{align}\label{eq:opt_loss}
\ell_{\rm{opt}}(w) = -\env{\alpha_1q+\alpha_2\log(p_{W_{\rm{opt}} })}{w}{1},
\end{align}
where $q(x) = x^2/2$ and 
\begin{align}
\begin{split}
\alpha_1&=  \frac{1-\sig_{\rm{opt}}^2\Ic(W_{\rm{opt}})}{\delta(\sig_{\rm{opt}}^2\Ic(W_{\rm{opt}})+\Ic(W_{\rm{opt}})-1)},\\ \alpha_2 &= \frac{1}{\delta(\sig_{\rm{opt}}^2\Ic(W_{\rm{opt}})+\Ic(W_{\rm{opt}})-1)}.
\end{split}
\label{eq:alphas}
\end{align}
If $\ell_{\rm opt}$ defined as in \eqref{eq:opt_loss} is convex and the equation $\kappa(\sig) = 1/\delta$ has a unique solution, then $\sig_{\ell_{\rm{opt}}}=\sig_{\rm{opt}} $. 
\end{thm}

In general, there is no guarantee that the function $\ell_{\rm{opt}}(\cdot)$ as defined in \eqref{eq:opt_loss} is convex. However, if this is the case, the theorem above guarantees that it is optimal \footnote{Strictly speaking, the performance is optimal among all convex loss functions $\ell$ for which \eqref{eq:eq_main} has a unique solution as required by Theorem \ref{sec:lem}.}. A \emph{sufficient} condition for $\ell_{\rm{opt}}(w)$ to be convex, is provided in Section \ref{sec:opt_FL}.  Importantly, in Section \ref{sec:opt_special} we show that this condition holds for observations following the Signed model.  Thus, for this case the resulting function is convex. Although we do \emph{not} prove the convexity of optimal loss function for the Logistic and Probit models, our numerical results (e.g., see Figure \ref{fig:opt_loss_logistic}) suggest that this is the case. Concretely, we conjecture that the loss function $\ell_{\rm opt}$ is convex for Logistic and Probit models, and therefore by Theorem \ref{thm:opt_loss} its performance is optimal.

\begin{proof}[Proof of Theorem \ref{thm:opt_loss}]
We will show that the triplet $(\mu=1,\alpha=\sig_{\rm opt},\la=1)$ is a solution to the equations \eqref{eq:eq_main} for $\ell$ chosen as in \eqref{eq:opt_loss}. 
Using Proposition \ref{propo:FL} in the appendix we rewrite $\ell_{\rm opt}$ using the Fenchel-Legendre conjugate as follows :
\bea\label{eq:FLopt}
 \ell_{\rm opt} (w)= \left(q + \alpha_1 q + \alpha_2 \log p_{W_{\rm opt}}\right)^\star(w) - q(w),   
\eea
where $q(w) = w^2/2$, and for a function $f$, its Fenchel-Legendre conjugate is defined as:
$$
f^\star(x) = \max_y\, xy-f(y). 
$$
Next we use the fact that for any proper, closed and convex function $f$ it holds that, $(f^\star)^\star = f$ \cite[theorem 12.2]{rockafellar1970convex}. Therefore noting that $q+\alpha_1 q+ \alpha_2 \log p_{W_{\rm opt}}$ is a convex function (see the proof of Lemma \ref{lem:convex} in the appendix), combined with \eqref{eq:FLopt} yields that
\bea\label{eq:FLsecond}
(\ell_{\rm opt} + q)^\star = q+\alpha_1 q+ \alpha_2 \log p_{W_{\rm opt}}.
\eea
Additionally using Proposition \ref{propo:FL} we find that $\env{\ell_{\rm opt}}{w}{1} = q(w) - (q+\ell_{\rm opt})^\star(w),$ which by \eqref{eq:FLsecond} reduces to :
$$
\env{\ell_{\rm opt}}{w}{1} = -\alpha_1q(w) -\alpha_2\log p_{W_{\rm opt}}(w).
$$
Thus, by differentiation, we find that $\ell_{\rm opt}$ satisfies \eqref{eq:MorDer} with $c=-1$, i.e.,
\bea \label{eq:MorDer2} 
\envdx{{\ell}_{\rm{opt}}} {w} {1} = -\alpha_1w-\alpha_2\cdot\ksi_{W_{\rm{opt}}}(w).
\eea
Next, we establish the desired by directly substituting \eqref{eq:MorDer2} into the system of equations in \eqref{eq:eq_main4}. 
First, using the values of $\alpha_1$ and $\alpha_2$ in \eqref{eq:alphas}, as well as, the fact that  $\kappa(\sig_{\rm{opt}})=1/\delta$, we have the following chain of equations:
\begin{align}
\nn\Exp\bigg[\,\left(\envdx{\ell_{\rm{opt}}}{W_{\rm{opt}} }{1}\right)^2\,\bigg] &= \Exp\bigg[(\alpha_1W_{\rm{opt}} + \alpha_2\,\ksi_{W_{\rm{opt}}}(W_{\rm{opt}}))^2\bigg] \\
\nn&= \alpha_1^2\,(\sig_{\rm{opt}}^2+1) + \alpha_2^2\,\Ic(W_{\rm{opt}}) + 2\, \alpha_1\alpha_2\,\Exp\left[W_{\rm{opt}}\cdot\ksi_{W_{\rm{opt}}}(W_{\rm{opt}})\right] \\
\nn&= \frac{1+\sig_{\rm{opt}}^2\,\left(\sig_{\rm{opt}}^2\,\Ic(W_{\rm{opt}})-1\right)}{\delta^2\,\left(\sig_{\rm{opt}}^2\,\Ic(W_{\rm{opt}})+\Ic(W_{\rm{opt}})-1\right)} 
=\frac{\sig_{\rm{opt}}^2}{\delta^2 \,\kappa(\sig_{\rm{opt}})}\\
&={\sig_{\rm{opt}}^2}/{\delta}.\label{eq:3.3-2}
\end{align}
This shows \eqref{eq:alpha_main}. Second, using again the specified values of $\alpha_1$ and $\alpha_2$, a similar calculation yields
\begin{align}
\nn\Exp\bigg[\envdx{\ell_{\rm{opt}}}{W_{\rm{opt}}}{1}\ksi_{W_{\rm{opt}}}(W_{\rm{opt}})\bigg] &= -\Exp\left[\left(\alpha_1W_{\rm{opt}}+\alpha_2\,\ksi_{W_{\rm{opt}}}(W_{\rm{opt}})\right)\,\ksi_{W_{\rm{opt}}}(W_{\rm{opt}})\right] = \alpha_1 - \alpha_2\,\Ic(W_{\rm{opt}}) \\
&= -{1}/{\delta}.\label{eq:3.3-2}
\end{align}
Recall from \eqref{eq:SteinIBP1} that $\E\bigg[ G\cdot \envdx{\ell_{\rm{opt}}}{W_{{\rm{opt}}} }{1}  \bigg] = -\sigma_{\rm{opt}} \,\E\left[\envdx{\ell_{\rm{opt}}}{W_{{\rm{opt}}} }{1}\ksi_{W_{{\rm{opt}}}}(W_{{\rm{opt}}})\right].$ This combined with \eqref{eq:3.3-2} yields \eqref{eq:lambda_main}.
Finally, we use again \eqref{eq:MorDer2} and the specified values of $\alpha_1$ and $\alpha_2$ to find that
 \begin{align}
\nn \Exp\bigg[W_{\rm{opt}} \cdot\envdx{\ell_{\rm{opt}}}{W_{\rm{opt}}}{1} \bigg] &= \Exp\bigg[W_{\rm{opt}} \cdot (-\alpha_1 W_{\rm{opt}}- \alpha_2\,\ksi_{W_{\rm{opt}}}(W_{\rm{opt}})) \bigg] \\&=  -\alpha_1\,\Exp\left[W_{\rm{opt}}^2\right]-\alpha_2\,\Exp\left[W_{\rm{opt}}\,\ksi_{W_{\rm{opt}}}(W_{\rm{opt}})\right]  \\
\nn &= -\alpha_1(\sig_{\rm{opt}}^2+1) - \alpha_2 \int_{-\infty}^{\infty}{w\,p^{\prime}_{W_{\rm{opt}}}(w)\,\text{d}w} = -\alpha_1(\sig_{\rm{opt}}^2+1)+\alpha_2 \\
 &= {\sig_{\rm{opt}}^2}/{\delta}. \label{eq:3.3-1}
\end{align}
But, using \eqref{eq:SteinIBP1} it holds that
\bea
\Exp\bigg[W_{\rm{opt}} \cdot\envdx{\ell_{\rm{opt}}}{W_{\rm{opt}}}{1} \bigg] = -\sigma_{\rm{opt}}^2\, \Exp\bigg[\envdx{\ell_{\rm{opt}}}{W_{\rm{opt}}}{1}\ksi_{W_{\rm{opt}}}(W_{\rm{opt}})\bigg] + \Exp\bigg[Y\, S \cdot\envdx{\ell_{\rm{opt}}}{W_{\rm{opt}}}{\la}  \bigg].\nn
\eea
This combined with \eqref{eq:3.3-1} and \eqref{eq:3.3-2} shows that $\Exp\bigg[Y\, S \cdot\envdx{\ell_{\rm{opt}}}{W_{\rm{opt}}}{\la}  \bigg]=0$, as desired to satisfy \eqref{eq:mu_main}. This completes the proof of the theorem. 
%
%
%
\end{proof} 


\section{Special cases}\label{sec:cases}
\subsection{Least-Squares}\label{sec:LS}

By choosing $\ell(t)=(t-1)^2$ in \eqref{eq:gen_opt}, we obtain the standard least-squares estimate. To see this, note that since $y_i=\pm 1$, it holds for all $i$ that
$
(y_i\ab_i^T\x-1)^2 = (y_i-\ab_i^T\x)^2.
$
Thus, $\xh$ is minimizing the sum of squares of the residuals:
\bea\label{eq:LS}
\xh=\arg\min_\x \sum(y_i-\ab_i^T\x)^2.
\eea
For this choice of a loss function, we can solve the equations in \eqref{eq:eq_main} in closed form. Furthermore, the equations have a (unique, bounded) solution for any $\delta>1$ provided that $\E[SY]>0$. The final result is summarized in the corollary below. See Section \ref{sec:corLS} for the proof.

\begin{cor}[Least-squares]\label{cor:LS}
 Let Assumption \ref{ass:Gaussian} hold and $\delta>1$. For the label function assume that  $\E[SY]>0$ in the notation of \eqref{eq:GSY}. Let $\xh$ be as in \eqref{eq:LS}. Then, in the limit of $m,n\rightarrow+\infty$, $m/n\rightarrow\delta$, Equations \eqref{eq:corr_thm} and \eqref{eq:norm_thm} hold with probability one with $\alpha$ and $\mu$ given as follows:
%
\bea
\mu &= \E[SY] \label{eq:mu_LS}, \\
\alpha &=\sqrt{1-\left(\E[SY]\right)^2}\cdot\sqrt{\frac{1}{\delta-1}}\;.\label{eq:alpha_LS}
\eea
\end{cor}

Corollary \ref{cor:LS} appears in \cite{NIPS} (see also \cite{Bri,PV15,genzel2017high} and Section \ref{sec:LS_app} for an interpretation of the result). However, these previous works obtain results that are limited to least-squares loss. In contrast, our results are general and LS prediction is obtained as a simple corollary of our general Theorem \ref{thm:main}.  
Moreover, our study of fundamental limits allows us to quantify the sub-optimality gap of least-square (LS) as follows. 

\vp
\noindent\textbf{On the optimality of LS.}~On the one hand, Corollary \ref{cor:LS} derives an explicit formula for the effective noise variance $\sigma_{\rm{LS}}=\alpha/\mu$ of LS in terms of $E[YS]$ and $\delta$.  On the other hand, Corollary \ref{cor:Stam} provides an explicit lower bound on the optimal value $\sig_{\rm opt}$ in terms of $\Ic(SY)$ and $\delta$. Combining the two, we conclude that 
$$
\frac{\sigma_{\rm{LS}}^2}{\sig_{\rm opt}^2} \leq \xi:=(\Ic(SY)-1)\frac{1-(\E[SY])^2}{(\E[SY])^2}\,.
$$
In terms of correlation,
\begin{align*}
\frac{\rm{corr}_{opt}}{\rm{corr}_{\rm{LS}}} = \sqrt{\frac{1+\sig_{\rm{LS}}^2}{1+\sig_{\rm opt}^2}} \le \frac{\sig_{\rm{LS}}}{\sig_{\rm opt}}\le \sqrt{\ksi}\,,
\end{align*}
where the first inequality follows from the fact that $\sig_{\rm{LS}} \ge\sig_{\rm opt}.$ 
\begin{center}
\emph{Therefore, 
the performance of LS is at least as good as $\frac{1}{\sqrt{\ksi}}$ times the optimal one. \\ In particular, for Logistic and Probit models (for which Corollary \ref{cor:Stam} holds), \\ we can explicitly compute $\frac{1}{\sqrt{\ksi}} = 0.9972 \text{ and } 0.9804$,  respectively.} 
\end{center}

Another interesting consequence of combining Corollary \ref{cor:LS} with Corollary \ref{cor:Stam} is that LS would be optimal if $SY$ were a a Gaussian random variable. To see this, recall from Corollary \ref{cor:Stam} that if  $SY$ is Gaussian then:
$$
\sig_{\rm opt}^2=\frac{1}{(\delta-1)(\Ic(SY)-1)}.
$$ 
But, for $SY$ Gaussian, we can explicitly compute $\Ic(SY) = 1/\text{Var}[SY]$, which leads to $$\sig_{\rm opt}^2 = \frac{1-(\Exp[SY])^2}{(\Exp[SY])^2(\delta-1)}.$$ The right hand side is exactly $\sig_{\text{LS}}^2$. Therefore the optimal performance is achieved by the square-loss function if $SY$ is a Gaussian random variable.  In particular, the result described above applies to the following binary Gaussian-mixtures model:
$$
y_i = \pm1 ~~\Leftrightarrow~~ \ab_i\sim\Nn(y_i \x_0, \Id_n),\quad i\in[m].
$$
For this model, using the method introduced above (with appropriate modifications), we prove in \cite{Hossein_ISIT2020} that LS is optimal for $\delta>1$ among all choices of convex loss functions yielding a unique solution to the equations.


\subsection{Logistic \& Hinge loss}\label{sec:logloss}

 Theorem \ref{thm:main} only holds in regimes for which the set of minimizers of \eqref{eq:gen_opt} is bounded. As we show here, this is $\emph{not}$ always the case. Specifically, consider non-negative loss functions $\ell(t)\geq 0$ with the property $\lim_{t\rightarrow+\infty} \ell(t)=0$. For example, the hinge, exponential and logistic loss functions all satisfy this property. Now, we show that for such loss functions the set of minimizers is unbounded if $\delta<\delta^\star_f$ for some appropriate $\delta^\star_f>2$. First, note that the set of minimizers is unbounded if the following condition holds:
\bea\label{eq:sep}
\exists~\x_s\neq \mathbf{0} \quad\text{such that}\quad y_i\ab_i^T\x_s \geq 0, \quad \forall~i\in[m].
\eea
Indeed, if \eqref{eq:sep} holds then $\x=c\cdot\x_s$ with $c\rightarrow+\infty$, attains zero cost in \eqref{eq:gen_opt}; thus, it is optimal and the set of minimizers is unbounded. To proceed, we rely on a recent result by Candes and Sur \cite{candes2018phase} who prove that \eqref{eq:sep} holds iff \footnote{To be precise, \cite{candes2018phase} proves the statement for measurements $y_i,~i\in[m]$ that follow a logistic model. Close inspection of their proof shows that this requirement can be relaxed by appropriately defining the random variable $Y$ in \eqref{eq:GSY}. See also \cite{montanari2019generalization} and \cite{deng2019model}.} 
\bea\label{eq:threshold}
\delta \leq \delta^\star_{f}:= \left(\min_{c\in\R}\E\left[\left(G+c\,S\,Y\right)_{-}^2\right]\right)^{-1},
\eea
where $G,S$ and $Y$ are random variables as in \eqref{eq:GSY} and $(t)_{-}:=\min\{0,t\}$. We highlight that Logistic and Hinge losses give unbounded solutions in the Noisy-Signed model with $\eps=0$, since the condition \eqref{eq:sep} holds for $\x_s=\x_0$. However their performances are comparable to the optimal performance in both Logistic and Probit models (see Figures \ref{fig:figure2} and \ref{fig:figure3}).

%
\section{Numerical experiments} \label{sec:numerical}
In this section, we present numerical simulations that validate the predictions of Theorems \ref{thm:main}, \ref{sec:lem} and \ref{thm:opt_loss}.  We use the following three popular models as our case study: Signed, Logistic and Probit. We generate random measurements according to \eqref{eq:gen_model} and Assumption \ref{ass:Gaussian}. Without loss of generality (due to rotational invariance of the Gaussian measure) we set $\x_0=[1,0,...,0]^T$. We then obtain estimates $\xh_\ell$ of $\x_0$ by numerically solving \eqref{eq:gen_opt} and measure performance by the correlation value $\corr{\xh_\ell}{\x_0}$. Throughout the experiments, we set $n=128$ and the recorded values of correlation are averages over $25$ independent realizations. For each label function we first provide plots that compare results of Monte Carlo simulations to the asymptotic predictions for loss functions discussed in Section \ref{sec:cases}, as well as, to the optimal performance of Theorem \ref{sec:lem}. We next present numerical results on optimal loss functions. In order to empirically derive the correlation of optimal loss function, we run gradient descent-based optimization with 1000 iterations. As a general comment, we note that despite being asymptotic, our predictions appear accurate even for relatively small problem dimensions. 
For the analytical predictions we apply Theorem \ref{thm:main}. In particular for solving the system of non-linear equations in \eqref{eq:gen_opt}, we empirically observe (see also \cite{Master,salehi2019impact} for similar observation) that if a solution exists, then it can be efficiently found by the following fixed-point iteration method. Let $\vb := [\mu,\alpha,\la]^T$ and $\Fc:\R^3\rightarrow\R^3$ be such that \eqref{eq:gen_opt} is equivalent to $\vb = \Fc(\vb)$. With this notation, we initialize $\vb=\vb_0$ and for $k\geq 1$ repeat the iterations $\vb_{k+1} = \Fc(\vb_k)$ until convergence. \\  

\noindent\textbf{Logistic model.}~ For the logistic model, comparison between the predicted values  and the numerical results is illustrated in Figure \ref{fig:figure2}. Results are shown for LS, logistic and hinge loss functions. Note that minimizing the logistic loss corresponds to the maximum-likelihood estimator (MLE) for logistic model. An interesting observation in Figure \ref{fig:figure2} is that in the high-dimensional setting (finite $\delta$) LS has comparable (if not slightly better) performance to MLE. Additionally we observe that in this model, performance of LS is almost the same as the best possible performance derived according to Theorem \ref{sec:lem}. This confirms the analytical conclusion of Section \ref{sec:LS}. The comparison between the optimal loss function as in Theorem \ref{thm:opt_loss} and other loss functions is illustrated in Figure \ref{fig:opt_loss_logistic}. We note the obvious similarity between the shapes of optimal loss functions and LS which further explains the similarity between their performance.
\\

\noindent\textbf{Probit model.}~ Theoretical predictions for the performance of hinge and LS loss functions are compared with the empirical results and optimal performance of Theorem \ref{sec:lem} in Figure \ref{fig:figure3}. Similar to the Logistic model, in this model LS also outperforms hinge-loss and its performance resembles the performance of optimal loss function derived according to Theorem \ref{thm:opt_loss}. Figure \ref{fig:opt_probit} illustrates the shapes of LS, hinge-loss and the optimal loss functions for the Probit model. The obvious similarity between the shape of LS and optimal loss functions for all values of $\delta$ explains the close similarity of their performance.
\vp

Additionally by comparing the LS performance for the three models in Figures \ref{fig:figure1}, \ref{fig:figure2} and \ref{fig:figure3}, it is clear that higher (resp., lower) correlation values are achieved for signed (resp., logistic) measurements. This behavior is indeed predicted by Corollary \ref{cor:LS}: correlation performance is higher for higher values of $\mu=\E[SY]$. It can be shown that for  the signed, probit and logistic models we have $\mu= \sqrt{2/\pi}, \sqrt{1/\pi} \text{ and } 0.4132$, respectively.
 \begin{figure*}[t!]
\centering
	\begin{subfigure}{0.48\textwidth}
		\centering
    		\includegraphics[width=6.7cm, height=5.5cm]{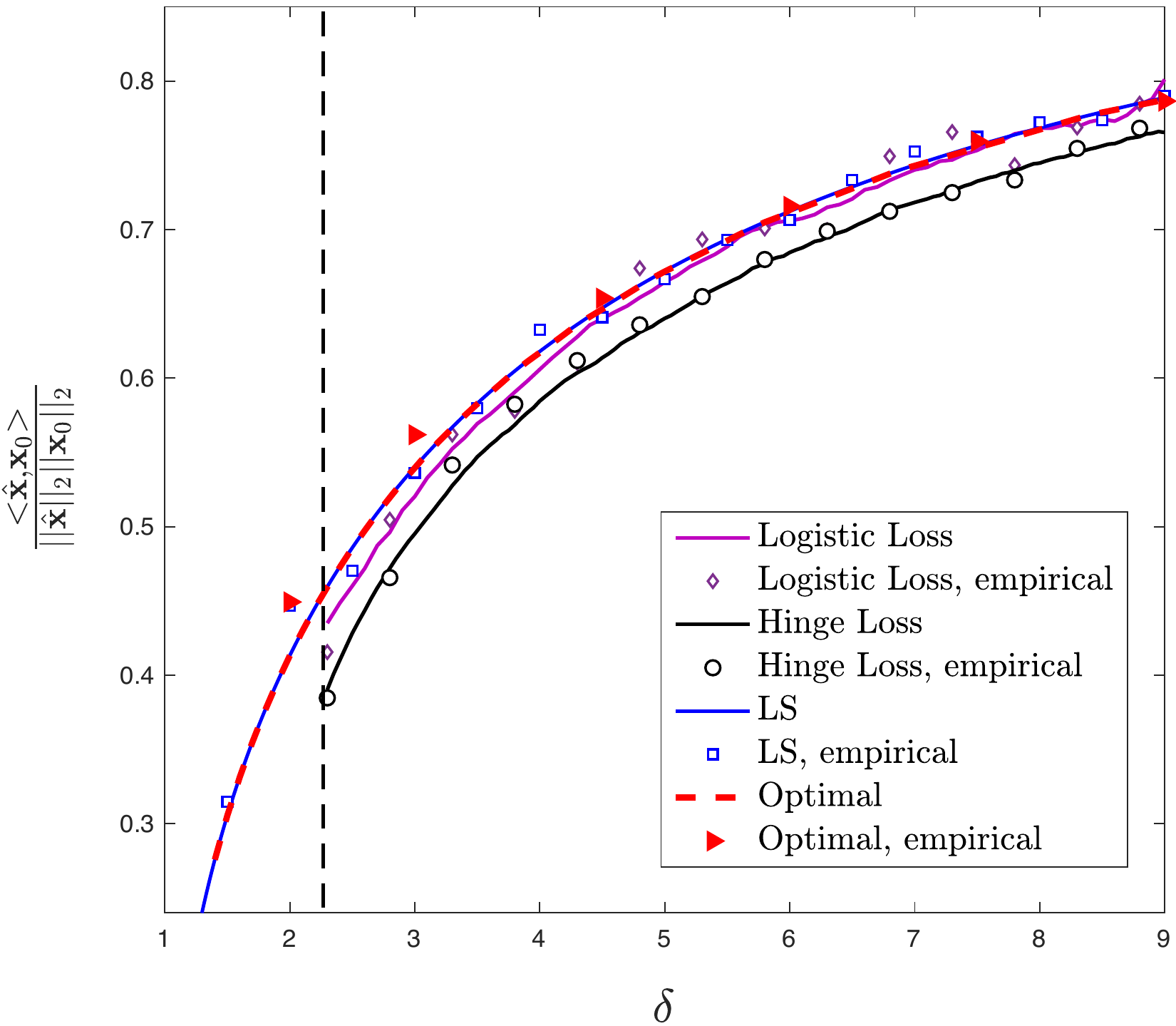}
    		\caption{{\footnotesize}}
    		\label{fig:figure2}
    \end{subfigure}
    	\begin{subfigure}{0.48\textwidth}
		\centering
    		\includegraphics[width=6.7cm, height=5.5cm]{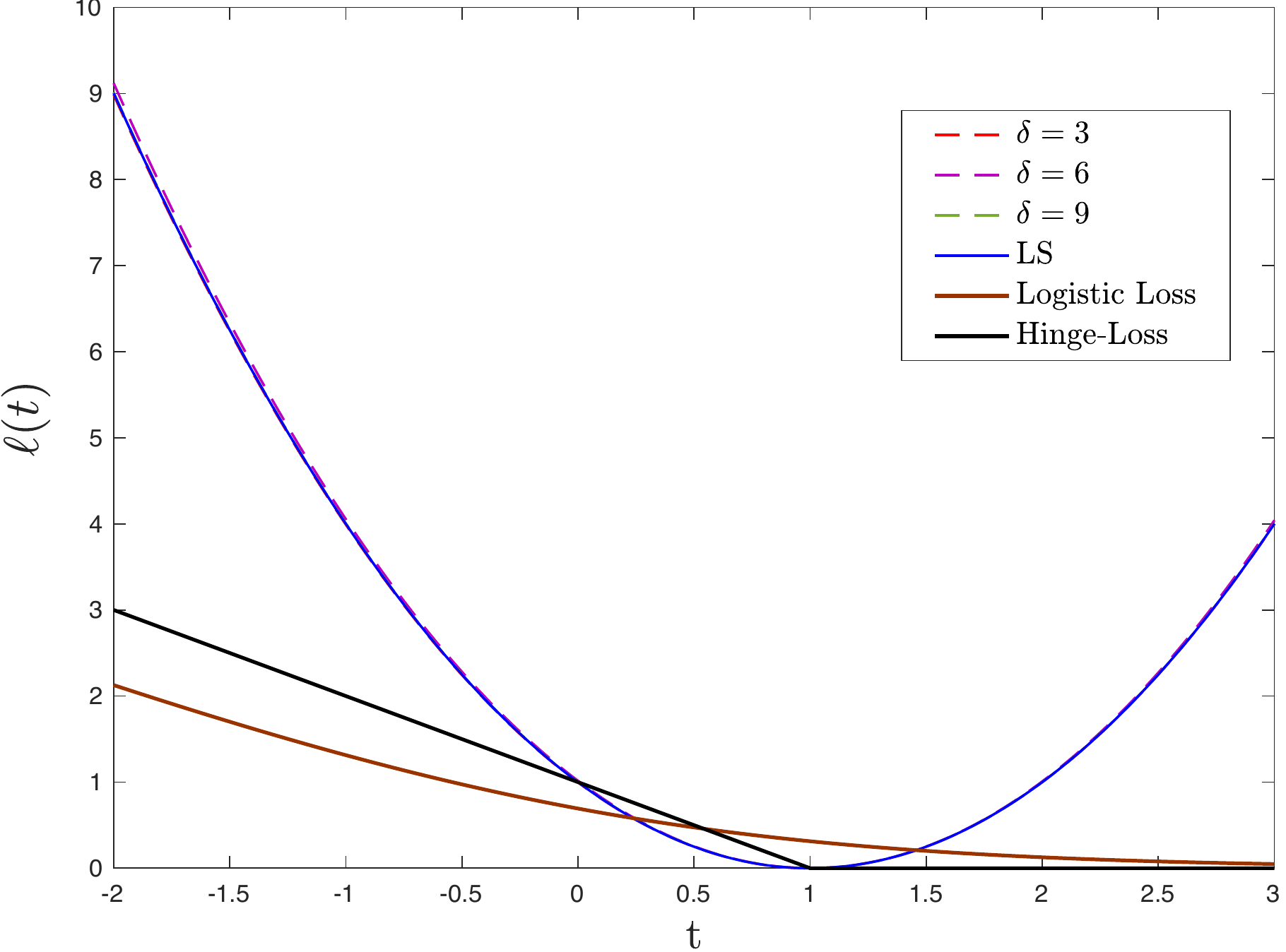}
    		\caption{{\footnotesize}}
    		\label{fig:opt_loss_logistic}
    \end{subfigure}
    \\
    \caption{\footnotesize Left: Comparison between analytical and empirical results for the performance of LS, Logistic loss, Hinge-loss and optimal loss function for Logistic model. The vertical dashed line represents $\delta^\star_f \approx  2.275$, as evaluated by \eqref{eq:threshold}. Right: Illustrations of optimal loss functions for different values of $\delta$, derived according to Theorem \ref{thm:opt_loss} for Logistic model. In order to signify the similarity of optimal loss function to the LS loss, the optimal loss functions (hardly visible) are scaled such that $\ell(1)=0$ and $\ell(2)=1$ .}
    \end{figure*}
    
\begin{figure*}[t!]
\centering
	\begin{subfigure}{0.48\textwidth} 
		\centering
    		\includegraphics[width=6.9cm, height=5.6cm]{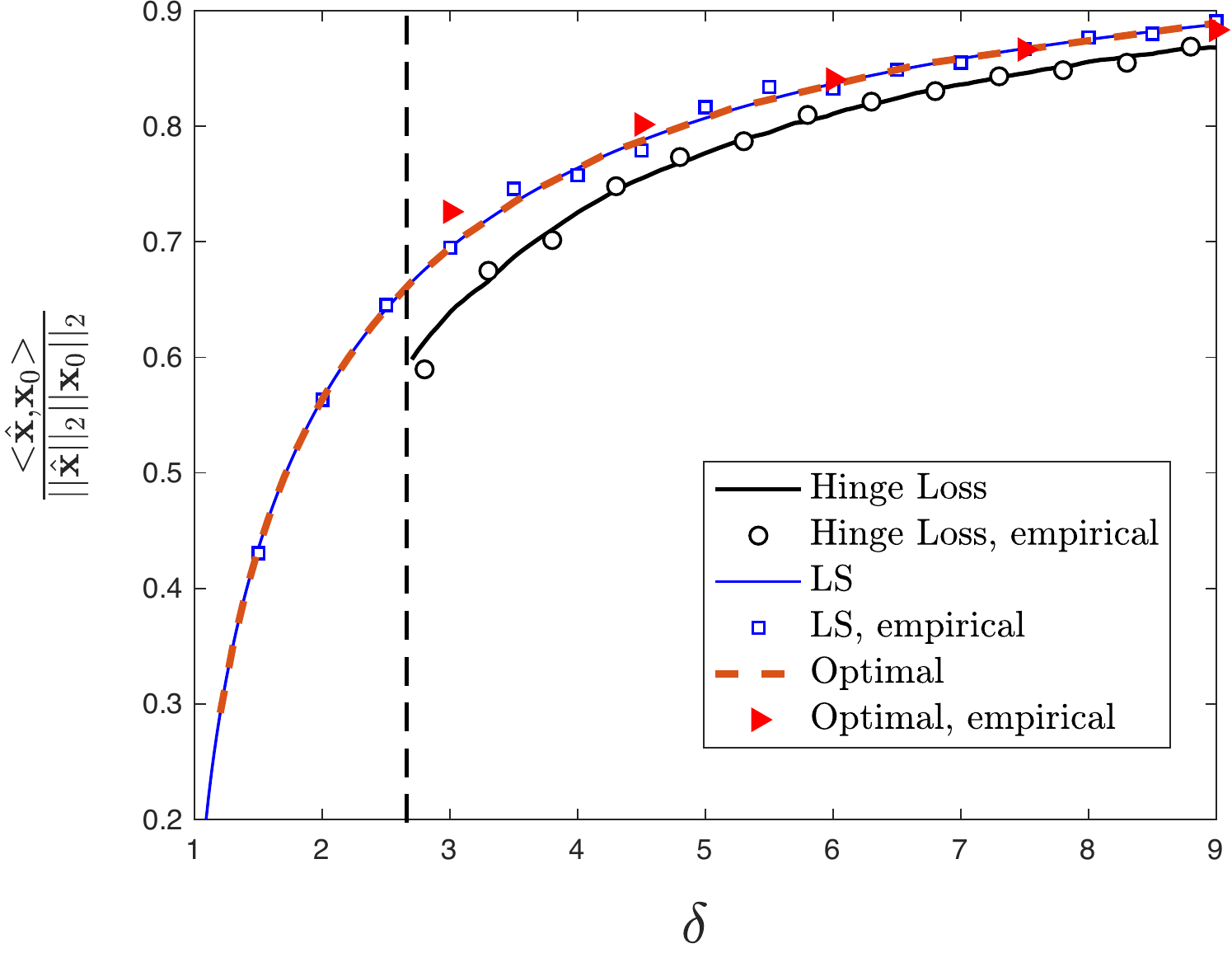}
		\caption{}
		\label{fig:figure3}
    \end{subfigure}
    	\begin{subfigure}{0.48\textwidth}
		\centering
    		\includegraphics[width=6.7cm, height=5.5cm]{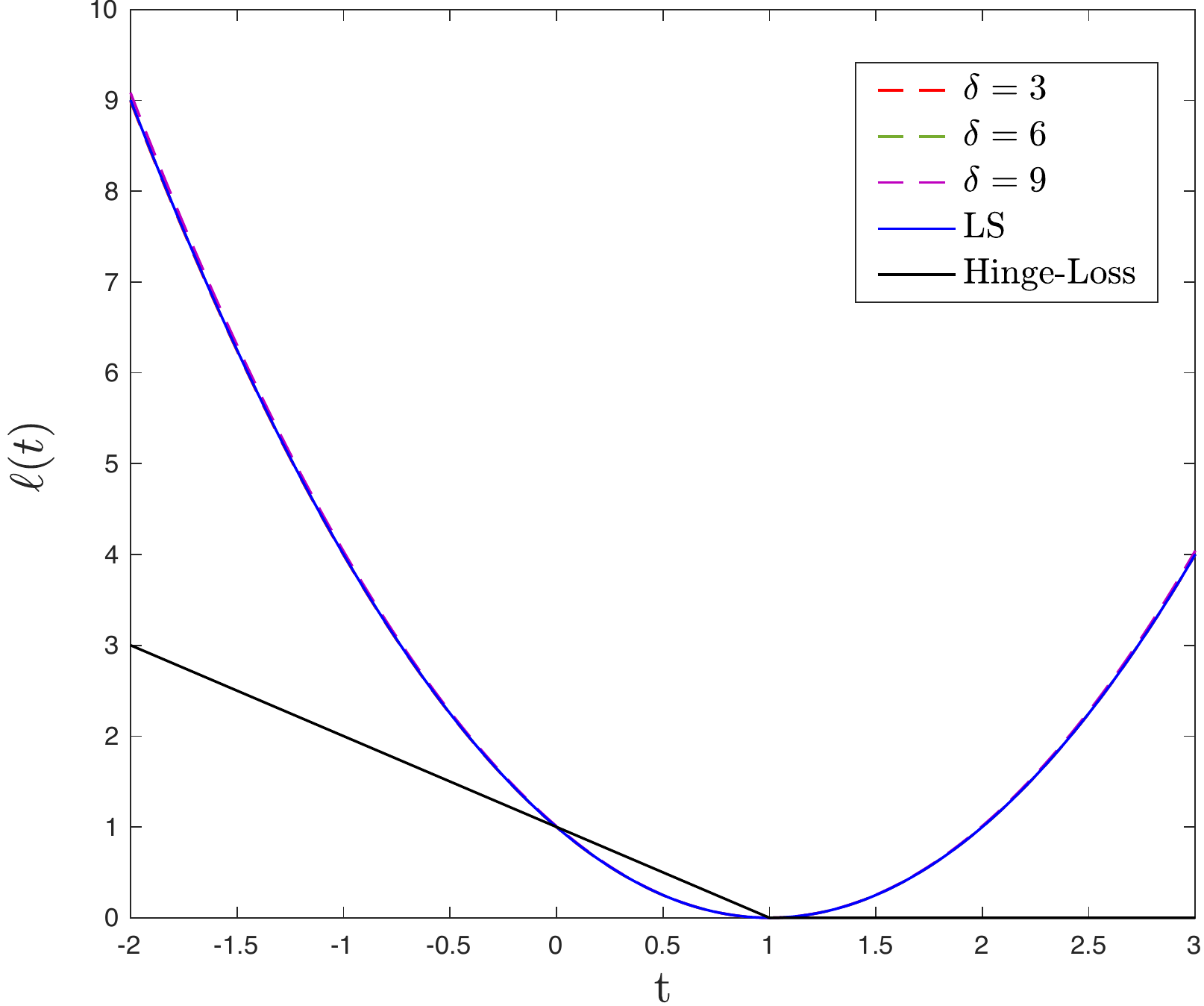}
		\caption{}
	   \label{fig:opt_probit}

    \end{subfigure}
    \caption{\footnotesize Left: Comparison between analytical and empirical results for the performance of LS, Hinge-loss and optimal loss function for Probit model. The vertical dashed line represents $\delta^\star_f \approx  2.699$, as evaluated by \eqref{eq:threshold}. Right: Illustrations of optimal loss functions for different values of $\delta$ derived according to Theorem \ref{thm:opt_loss} for Probit model. In order to signify the similarity of optimal loss function to the LS loss, the optimal loss functions (hardly visible) are scaled such that $\ell(1)=0$ and $\ell(2)=1$ }
 \end{figure*}

\vp
\noindent \textbf{Optimal loss function.}~~By putting together Theorems \ref{sec:lem} and \ref{thm:opt_loss}, we obtain a method on deriving the optimal loss function. This requires the following steps.
\\
\noindent$1$.\;\;Find $\sig_{\rm{opt}}$ by solving \eqref{eq:lem}.\\
\noindent$2$.\;\;Compute the density of $W_{\rm{opt}} = \sig_{\rm{opt}}G+SY$.\\
\noindent$3$.\;\;Compute $\ell_{\rm{opt}}$ according to \eqref{eq:opt_loss}.
\\
Note that computing $\sig_{\rm{opt}}$ needs the density function $p_W$ of the random variable $W = \sig\,G+SY$.  In principle $p_W$ can be calculated as the convolution of the Gaussian density with the pdf $p_{SY}$ of $SY$. Moreover, it follows from the recipe above that the optimal loss function depends on $\delta$ in general. This is because $\sig_{\rm opt}$ itself depends on $\delta$ via \eqref{eq:lem}.

\section{Conclusion}\label{sec:conc}
This paper derives \emph{sharp} asymptotic performance guarantees for a wide class of convex optimization based estimators for recovering a signal from binary observation models. 
  We further provide a theoretical upper bound on the best achievable performance among all convex loss functions. Using this, we develop a procedure for computing the optimal loss function. Finally, we provide numerical studies that show tight agreement with our theoretical results. Interesting future directions include studying the generalized linear measurement model beyond binary observations and characterizing the optimal loss function for such general models. 

\section*{Acknowledgment}
This work was supported by NSF Grants CCF-1755808 and CCF-1909320 and an Academic Senate Research Grant from UCSB.

%

\bibliographystyle{alpha}
\bibliography{compbib}

\newpage
\appendix
\section*{Appendix}
\section{Properties of Moreau envelopes}\label{sec:ME}
\subsection{Derivatives}
Recall the definition of the Moreau envelope $\env{\ell}{x}{\la}$ and proximal operator $\prox{\ell}{x}{\la}$ of a function $\ell$: 
\bea\label{eq:M_def2}
\env{\ell}{x}{\la} = \min_y ~\frac{1}{2\la}{(x-y)^2}+ \ell(y),
\eea
and
$\prox{\ell}{x}{\la} = \arg\min_{y}~\frac{1}{2\la}{(x-y)^2}+ \ell(y)$. 

\begin{propo}[Basic properties of $\Mc_{\ell}$ and $\rm{prox}_\ell$~\cite{rockafellar2009variational}]
\label{propo:der}
Let $\ell:\R\rightarrow\R$ be lower semi-continuous (lsc), proper and convex. The following statements hold for any $\la>0$.

\vp
\noindent{(a)} The proximal operator $\prox{\ell}{x}{\la}$ is unique and continuous. In fact, $\prox{\ell}{x}{\la}\rightarrow\prox{\ell}{x'}{\la'}$ whenever $(x,\la)\rightarrow(x',\la')$ with $\la'>0$.

\vp
\noindent{(b)} The  value $\envl{x}{\la}$ is finite and depends continuously on $(\la,x)$, with $\envl{x}{\la} \rightarrow f(x)$ for all $x$ as $\la\rightarrow 0_+$.

\vp
\noindent{(c)} The Moreau envelope function is differentiable with respect to both arguments. Specifically, for all $x\in\R$, the following properties are true:
\bea
\envdx{\ell}{x}{\la} &= \frac{1}{\la}{(x-\prox{\ell}{x}{\la})}, \label{eq:envdx} \\
\envdla{\ell}{x}{\la} &= -\frac{1}{2\la^2}{(x-\prox{\ell}{x}{\la})^2}.
\eea
If in addition $\ell$ is differentiable and $\ell^{'}$\,denotes its derivative, then
\bea
\envdx{\ell}{x}{\la}&= \elld(\proxl{x}{\la})\label{eq:envdxp},\\
\envdla{\ell}{x}{\la}&= -\frac{1}{2}(\elld(\proxl{x}{\la})^2.
\eea
\end{propo}

\subsection{Alternative representations of \eqref{eq:eq_main}}
Replacing the above relations for derivative of $\Mc_{\ell}$ in \eqref{eq:eq_main}, we can write the equations in terms of the proximal operator. If $\ell$ is differentiable then the Equations \eqref{eq:eq_main}
 can be equivalently written as follows:
\begin{subequations}\label{eq:eq_main22}
\bea
 \Exp\bigg[Y\, S \cdot\elld\left( \proxl{\ourx}{\la} \right)  \bigg]\label{eq:mu_main22}&=0,\\
{\la^2}\,{\delta}\,\Exp\bigg[\left(\elld\left( \proxl{\ourx}{\la} \right)\right)^2\bigg]&=\alpha^2,  \label{eq:alpha_main22}\\
\lambda\,\delta\,\E\bigg[ G\cdot \elld\left(\proxl{\ourx}{\la}\right)  \bigg]&=\alpha. \label{eq:lambda_main22} 
\eea
\end{subequations}
Finally, if $\ell$ is two times differentiable then applying integration by parts  in Equation \eqref{eq:lambda_main2} results in the following reformulation of  \eqref{eq:lambda_main}:
\bea\label{eq:lambda_main3}
1 &= \lambda\,\delta\, \Exp\left[\frac{\elldd\left( \proxl{\ourx}{\la} \right)}{1+\la\, \elldd\left( \proxl{\ourx}{\la} \right)}\right].
\eea

\subsection{Examples of proximal operators}
%
\paragraph{LAD.} For $\ell(t) = |t-1|$ the proximal operator admits a simple expression, as follows:
\bea
\prox{\ell}{x}{\lambda} = 1+ \soft{x-1}{\la},
\eea
where 
$$
\soft{x}{\la} =   \begin{cases}
      x-\lambda, & \text{if}\ x>\lambda, \\
      x+\lambda, & \text{if}\ x<-\lambda, \\
      0,&\text{otherwise}.
    \end{cases}
$$
is the standard soft-thresholding function.

\paragraph{Hinge-Loss.} When $\ell(t) = \max\{0,1-t\}$, the  proximal operator can be expressed in terms of the soft-thresholding function as follows:
\begin{align*}
\prox{\ell}{x}{\la} = 1 +\soft{x+\frac{\la}{2}-1}{\frac{\la}{2}}.
\end{align*}

\subsection{Fenchel-Legendre conjugate representation}
For a function $h:\mathbb{R}\rightarrow\mathbb{R}$, its Fenchel-Legendre conjugate, $h^\star:\mathbb{R}\rightarrow\mathbb{R}$ is defined as :
\begin{align*}
h^\star(x) = \max_{y}\left[xy-h(y)\right].
\end{align*}
The following proposition relates Moreau Envelope of a function to its Fenchel-Legendre conjugate. 
\begin{propo} \label{propo:FL} 
For $\la >0$ and a function $h$, we have:\\
\begin{align}\label{eq:FLME}
\env{h}{x}{\la} = \frac{q(x)}{\la}-\frac{1}{\la}\left(q+\la h\right)^\star(x),
\end{align}
where $q(x)= x^2/2.$
\end{propo}
\begin{proof}
\begin{dmath*}
\env{h}{x}{\la} = \frac{1}{2\la}\min_y \left[(x-y)^2 + 2\la h(y)\right]  \\
=\frac{x^2}{2\la} + \frac{1}{2\la}\min_y\left[ y^2-2xy+2\la h(y)\right] \\
=\frac{x^2}{2\la} - \frac{1}{\la}\max_y \left[xy-\left(y^2/2+\la h(y)\right)\right] \\
=\frac{q(x)}{\la} -\frac{1}{\la} \left(q+\la h\right)^\star(x). 
\end{dmath*}
\end{proof}

\subsection{Convexity of the Moreau envelope}

\begin{lem}\label{lem:H_cvx} The function $H:\R^3\rightarrow\R$ defined as follows
\bea\label{eq:H_def}
H(x,v,\la) = \frac{1}{2\la}(x-v)^2,
\eea
is jointly convex in its arguments.
\end{lem}
\begin{proof}
Note that the function $h(x,v)=(x-v)^2$ is jointly convex in $(x,v)$. Thus its perspective function  $$\la \,h(x/\la,v/\la)=(x-v)^2/\la = 2H(x,v,
\la)$$ is jointly convex in $(x,v,\la)$ \cite[Sec.~2.3.3]{boyd2009convex}, which completes the proof.
\end{proof}

\begin{propo}\label{lem:min_cvx} \noindent{(a)} [Prop.~2.22\cite{rockafellar2009variational}] Let $f(x,y)$ be jointly convex in its arguments. Then, the function $g(x)=\min_y f(x,y)$ is convex.\\[4pt] 
\noindent{(b)} [Sec.~3.2.3\cite{boyd2009convex}] Suppose \(f_{i}: \R \rightarrow \mathbb{R}\) is a set of concave functions, with \(i \in A\) an index set. Then the function $f:\R\rightarrow\R$ defined as \(f(x):=\inf _{i \in A} f_{i}(x)\) is concave.
\end{propo}

\begin{lem}\label{lem:Moreau_cvx}
Let $\ell:\R\rightarrow\R$ be a lsc, proper, convex function. Then, $\env{\ell}{x}{\la}$ is \emph{jointly} convex in $(x,\la)$.
\end{lem}
\begin{proof}
Recall that 
\bea
\env{\ell}{x}{\la} = \min_v \;G(\ab):= {\frac{1}{2\la}(x-v)^2 + \ell(v)},\label{eq:defn_mor_lem}
\eea
where for compactness, we let $\ab\in\R^3$ denote the triplet $(x,v,\la)$.   Now, let $\ab_i=(x_i,v_i,\la_i),~i=1,2$, $\theta\in (0,1)$ and $\thetao:=1-\theta$. With this notation, we may write
\bea\nn
G(\theta \ab_1 + \thetao \ab_2) &= H\left(\theta x_1+\thetao x_2, \theta\la_1 + \thetao \la_2, \theta v_1 + \thetao v_2 \right) + \ell(\theta v_1 + \thetao v_2) \\[4pt]
&\leq \theta H(x_1,v_1,\la_1)+ \thetao  H(x_2,v_2,\la_2) + \theta\ell(v_1)  + \thetao \ell(v_2) \nn\\[4pt]
&= \theta G(\ab_1) + \thetao G(\ab_2).\nn
\eea
For the first equality above we recalled the definition of $H:\R^3\rightarrow\R$ in \eqref{eq:H_def} and the inequality right after follows from Lemma \ref{lem:H_cvx} and convexity of $\ell$. Thus, the function $G$ is jointly convex in its arguments. Using this fact, as well as \eqref{eq:defn_mor_lem}, and applying Proposition \ref{lem:min_cvx}(a) completes the proof.
\end{proof}


\subsection{The expected Moreau-envelope (EME) function and its properties}\label{sec:strict_EME}

The performance of the ERM estimator \eqref{eq:gen_opt} is governed by the system of equations \eqref{eq:eq_main} in which the Moreau envelope function $\env{\ell}{x}{\la}$ of the loss function $\ell$ plays a central role. More precisely, as already hinted by \eqref{eq:eq_main} and will become clear in Appendix \ref{sec:proof}, what governs the behavior is the function
\bea\label{eq:EME_YS}
(\alpha>0,\mu,\tau>0,\gamma>0) \mapsto \E\big[\env{\ell}{\alpha G + \mu SY}{\tau/\gamma}\big],
\eea 
which we call the expected Moreau envelope (EME). Recall here that $Y=f(S)$. Hence, the EME is the key summary parameter that captures the role of both the loss function $\ell:\R\rightarrow\R$ and of the link function $f:\R\rightarrow\{\pm1\}$ on the statistical performance of \eqref{eq:gen_opt}.

In this section, we study several favorable properties of the EME. In \eqref{eq:EME_YS} the expectation is over $G,S\simiid\Nn(0,1)$. We first study the EME under more general distribution assumptions in Sections \ref{sec:EME_der}--\ref{sec:EME_conc} and we then specialize our results to Gaussian random variables $G$ and $S$ in Section \ref{sec:EME_Gauss}.

\subsubsection{Derivatives}\label{sec:EME_der}

\begin{propo}\label{propo:der_EME}
Let $\ell:\R\rightarrow\R$ be a lsc, proper and convex function. Further let $X,Z$ be independent random variables with bounded second moments $\E[X^2]<\infty$, $\E[Z^2]<\infty$. Then the expected Moreau envelope function $\E\left[\env{\ell}{c X +  Z}{\la}\right]$, is differentiable with respect to both $c$ and $\la$ and the derivatives are given as follows:
\bea
\frac{\partial}{\partial c}  \E\Big[\env{\ell}{c X +  Z}{\la}\Big] &= \E\Big[X\envdx{\ell}{cX+Z}{\la}\Big], \label{eq:eenvdx} \\
\frac{\partial}{\partial \la} \E\Big[\env{\ell}{c X +  Z}{\la}\Big] &= \E\Big[\envdla{\ell}{cX+Z}{\la}\Big]. \label{eq:eenvdla} 
\eea
\end{propo}
\begin{proof}
The proof is an application of the Dominated Convergence Theorem (DCT). First, by Proposition \ref{propo:der}(b), for every $c\in\R$ and any $\la>0$ the function $\E[\env{\ell}{c X +  Z}{\la}]$ takes a finite value. Second, by Proposition \ref{propo:der}(c) $\env{\ell}{c x +  z}{\la}$ is continuously differentiable with respect to both $c$ and $\la$:
\bea\nn
\frac{\partial}{\partial c}\env{\ell}{c X +  Z}{\la} &=  X\envdx{\ell}{cX+Z}{\la} = X\frac{1}{\la}\big(cX+Z - \prox{\ell}{cX+Z}{\la}\big),\\
\frac{\partial}{\partial \la}\env{\ell}{c X +  Z}{\la} &=  \envdla{\ell}{cX+Z}{\la} = -\frac{1}{2\la^2}\big(cX+Z - \prox{\ell}{cX+Z}{\la}\big)^2.\nn
\eea
From this, note that Cauchy-Schwarz inequality gives
$$
\E\Big[\frac{\partial}{\partial c}\env{\ell}{c X +  Z}{\la} \Big] \leq  \left(\E[X^2])^{1/2}\right) \Big(\E\Big[\frac{1}{\la^2}\underbrace{\big(cX+Z - \prox{\ell}{cX+Z}{\la}\big)^2}_{:=A}\Big]\Big)^{1/2},
$$
Therefore, the remaining condition to check so that DCT can be applied is that the term $A/\la^2$ above is integrable. To begin with, we can easily bound $A$ as:
$
A \leq 2 (cX+Z)^2 + 2 (\prox{\ell}{cX+Z}{\la})^2.
$
Next, by non-expansiveness (Lipschitz property) of the proximal operator \cite[Prop.~12.19]{rockafellar2009variational} we have that
$
|\prox{\ell}{cX+Z}{\la}| \leq |cX+Z| + |\prox{\ell}{0}{\la}|.
$
Putting together, we find that 
$$
A \leq 6 (cX+Z)^2 + 2  |\prox{\ell}{0}{\la}|^2 \leq 12c^2X^2 +  12Z^2 + 2  |\prox{\ell}{0}{\la}|^2.
$$
We consider two cases. First, for fixed $\la>0$ and any compact interval $\Ic$, we have that
$$
\E\sup_{c\in\Ic}[A] \leq 12(\sup_{c\in\Ic}c^2)\E[X^2]+12\E[Z]^2 + 2 |\prox{\ell}{0}{\la}|^2 < \infty.
$$
Similarly, for fixed $c$ and any compact interval $\Jc$ on the positive real line, we have that
$$
\E\sup_{\la\in\Jc}[A/\la^2] \leq 12\sup_{\la\in\Jc}\frac{c^2\E[X^2] + \E[Z]^2}{\la^2} + 2 \sup_{\la\in\Jc}\frac{|\prox{\ell}{0}{\la}|^2}{\la^2} < \infty,
$$
where we also used boundedness of the proximal operator (cf. Proposition \ref{propo:der}(a)). This completes the proof.
\end{proof}

\subsubsection{Strict convexity}
We study convexity properties of the \emph{expected Moreau envelope function} $\Psi:\R^3\rightarrow\R$: 
\bea\label{eq:EM_def}
\Psi(\vb) := \Psi(\al,\mu,\la):=\E\Big[\env{\ell}{\al X + \mu Z}{\la}\Big],
\eea
for a lsc, proper, convex function $\ell$ and independent random variables $X$ and $Z$ with positive densities. 
Here and onwards,  we let $\vb\in\R^3$ denote a triplet $(\al,\mu,\la)$ and the expectation is over the randomness of $X$ and $Z$. 
From Lemma \ref{lem:Moreau_cvx}, it is easy to see that $\Psi(\vb)$ is convex. In this section, we prove a stronger claim: 

\begin{center}
``\,If $\ell$ is strictly convex and does not attain its minimum at $0$, then $\Psi(\vb)$ is also strictly convex.\,''
\end{center}

\noindent This is summarized in Proposition \ref{propo:strict_EME} below.

\begin{propo}[Strict convexity]\label{propo:strict_EME}
 Let $\ell:\R\rightarrow\R$ be a function with the following properties: (i) it is proper strictly convex; (ii) it is continuously differentiable and its derivative $\ellp$ is such that $\ellp(0)\neq 0$. Further let $X,Z$ be independent random variables with strictly positive densities. Then, the function $\Psi:\R^3\rightarrow\R$ in \eqref{eq:EM_def} is jointly strictly convex in its arguments.
\end{propo}
\begin{proof}
Let $\vb_i=(\alpha_i,\mu_i,\la_i),~i=1,2$, $\theta\in(0,1)$ and $\thetao=1-\theta$. Further assume that $\vb_1\neq \vb_2$ and define the proximal operators
$$
\proxri{i}{X}{Z} := \prox{\ell}{\alpha_i X+ \mu_i Z}{\la_i} = \arg\min_{v}  \frac{1}{2\la_i}\left(\alpha_i X+ \mu_i Z - v\right)^2 + \ell(v),
$$
for $i=1,2$. Finally, denote $\la_\theta:=\theta\la_1+\thetao\la_2, \al_\theta:=\theta\alpha_1+\thetao\alpha_2$ and $\mu_\theta:=\theta\mu_1+\thetao\mu_2$. With this notation, 
\bea
&\Psi(\theta\vb_1+\thetao\vb_2)  \leq \E\left[\,  \frac{1}{2\la_\theta}\Big(\al_\theta X+ \mu_\theta Z - (\theta \proxri{1}{X}{Z} + \thetao \proxri{2}{X}{Z}) \Big)^2 + \ell\Big(\theta \proxri{1}{X}{Z} + \theta \proxri{2}{X}{Z}\Big)   \,\right]\nn \\[4pt]
&\qquad\qquad\qquad=
\E\Big[\, H\Big( \al_\theta X+ \mu_\theta Z, \theta \proxri{1}{X}{Z} + \theta \proxri{2}{X}{Z}, \la_\theta \Big) + \ell\Big(\theta \proxri{1}{X}{Z} + \thetao \proxri{2}{X}{Z}\Big)   \,\Big]\nn \\[4pt]
&\leq
\E\left[\, \theta H\Big(\al_1 X+ \mu_1 Z,  \proxri{1}{X}{Z}, \la_1\Big) + \thetao H\Big(\al_2 X+ \mu_2 Z,  \proxri{2}{X}{Z}, \la_2\Big)  + \ell\Big(\theta \proxri{1}{X}{Z} + \thetao \proxri{2}{X}{Z}\Big)   \,\right].\label{eq:EME_step1}
\eea
The first inequality above follows by the definition of the Moreau envelope in \eqref{eq:M_def2}. The equality in the second line uses the definition of the function $H:\R^3\rightarrow\R$ in \eqref{eq:H_def}. Finally, the last inequality follows from convexity of $H$ as proved in Lemma \ref{lem:H_cvx}.\\
Continuing from \eqref{eq:EME_step1}, we may use convexity of $\ell$ to find that 
\bea
&\Psi(\theta\vb_1+\thetao\vb_2) \leq \eqref{eq:EME_step1} \nn\\[4pt]
&\leq \E\Big[\, \theta H(\al_1 X+ \mu_1 Z, \la_1,  \proxri{1}{X}{Z}) + \thetao H(\al_2 X+ \mu_2 Z, \la_2,  \proxri{2}{X}{Z})  + \theta\ell(\proxri{1}{X}{Z}) + \thetao \ell(\proxri{2}{X}{Z}) \Big] \label{eq:EME_make_strict}\\[4pt]
&= \theta \Psi(\vb_1) + \thetao \Psi(\vb_2).\nn
\eea
This already proves convexity of \eqref{eq:EM_def}. In what follows, we will argue that the inequality in \eqref{eq:EME_make_strict} is in fact strict under the assumption of the lemma.

Specifically, in Lemma \ref{lem:prox_diff}, we prove that under the assumptions of the proposition, for $\vb_1\neq \vb_2$, it holds
$$
\E\Big[ \ell\left( \theta \proxri{1}{X}{Z} + \thetao \proxri{2}{X}{Z} \right) \Big] < \theta\, \E\Big[ \ell\left( \proxri{1}{X}{Z} \right) \Big] +  \thetao\, \E\Big[ \ell\left( \proxri{2}{X}{Z} \right) \Big].
$$
Using this in \eqref{eq:EME_step1} completes the proof of the proposition. The idea behind the proof of Lemma \ref{lem:prox_diff} is as follows. First, we use the fact that $\vb_1\neq\vb_2$ and $\ell'(0)\neq 0$ to argue that there exists a non-zero measure set of $(x,z)\in\R^2$ such that $\proxri{1}{x}{z}\neq\proxri{2}{x}{z}$. Then, the desired claim follows by \emph{strict} convexity of $\ell$.
\end{proof}

\begin{lem}\label{lem:prox_diff} Let $\ell:\R\rightarrow\R$ be  a proper strictly convex function that is continuously differentiable with $\ell '(0)\neq 0$. Further assume independent continuous random variables $X,Z$ with strictly positive densities. Fix arbitrary triplets $\vb_i=(\alpha_i,\mu_i,\la_i),~i=1,2$ such that $\vb_1\neq \vb_2$. Further denote
\bea\label{eq:prox_short}
\proxri{i}{X}{Z} := \prox{\ell}{\alpha_i X+ \mu_i Z}{\la_i},~~~i=1,2.
\eea
Then, there exists a ball $\Sc\subset\R^2$ of nonzero measure, i.e. $\Pro\left( (X,Z)\in\Sc \right)>0$, such that 
$\proxri{1}{x}{z}\neq \proxri{2}{x}{z}$, for all $(x,z)\in\Sc$. Consequently,  for any $\theta\in(0,1)$ and $\thetao=1-\theta$, the following strict inequality holds,
\bea\label{eq:ell_prox_strict}
\E\Big[ \ell\left( \theta \proxri{1}{X}{Z} + \thetao \proxri{2}{X}{Z} \right) \Big] < \theta\, \E\Big[ \ell\left( \proxri{1}{X}{Z} \right) \Big] +  \thetao\, \E\Big[ \ell\left( \proxri{2}{X}{Z} \right) \Big].
\eea
\end{lem}
\begin{proof}
Note that \eqref{eq:ell_prox_strict} holds trivially with ``$<"$ replaced by ``$\leq"$ due to the convexity of $\ell$. To prove that the inequality is strict, it suffices, by strict convexity of $\ell$, that there exists subset $\Sc\subset\R^2$ that satisfies the following two properties:
\begin{enumerate}
\item $\proxri{1}{x}{z}\neq \proxri{2}{x}{z}$, for all $(x,z)\in\Sc$.
\item $\Pro\left( (X,Z)\in\Sc \right)>0$.
\end{enumerate}
Consider the following function $f:\R^2\rightarrow\R$: 
\bea
f(x,z):= \proxri{1}{x}{z}- \proxri{2}{x}{z}.
\eea
By lemma \ref{lem:key_prox}, there exists $(x_0,z_0)$ such that 
\bea\label{eq:x0y0}
 f(x_0,z_0) \neq 0.
 \eea 
Moreover, by continuity of the proximal operator (cf. Proposition \ref{propo:der}(a)), it follows that $f$ is continuous. From this and \eqref{eq:x0y0}, we conclude that for sufficiently small $\zeta>0$ there exists a $\zeta$-ball $\Sc$ centered at $(x_0,z_0)$, such that property 1 holds. Property $2$ is also guaranteed to hold for $\Sc$, since both $X,Z$ have strictly positive densities and are independent.
\end{proof}

\begin{lem}\label{lem:key_prox}
Let $\ell:\R\rightarrow\R$ be  a proper, convex function. Further assume that $\ell:\R\rightarrow\R$ is continuously differentiable  and $\ell '(0)\neq 0$. Let $\al_1,\al_2>0$, $\la_1,\la_2>0$. Then, the following statement is true
\bea
(\alpha_1,\mu_1,\la_1) \neq (\alpha_2,\mu_2,\la_2) \quad\Longrightarrow\quad \exists (x,z)\in\R^2: \prox{\ell}{\alpha_1 x+ \mu_1 z}{\la_1} \neq \prox{\ell}{\alpha_2 x+ \mu_2 z}{\la_2}.
\eea
\end{lem}
\begin{proof}
We prove the claim by contradiction, but first, let us set up some useful notation. Let $\vb\in\R^3$ denote triplets $(\al,\mu,\la)$ and further define
$$
\proxri{\al,\mu,\la}{x}{z} := \prox{\ell}{\alpha x+ \mu z}{\la},
$$
and 
$$
\ellprox{\al,\mu,\la}{x}{z} := \ellp\left(\prox{\ell}{\alpha x+ \mu z}{\la}\right).
$$
By Proposition \ref{propo:der}, the following is true:
\bea\label{eq:L_eqv}
\ellprox{\al,\mu,\la}{x}{z}= \frac{1}{\la}\left( \alpha x+ \mu z - \proxri{\al,\mu,\la}{x}{z} \right).
\eea
For the sake of contradiction, assume that the claim of the lemma is false. Then, 
\bea\label{eq:prox_eq}
\proxri{\al_1,\mu_1,\la_1}{x}{z} =  \proxri{\al_2,\mu_2,\la_2}{x}{z}, \quad\forall (x,z)\in\R^2.
\eea
From this, it also holds that
\bea\label{eq:der_eq}
\ellprox{\al_1,\mu_1,\la_1}{x}{z} =  \ellprox{\al_2,\mu_2,\la_2}{x}{z},\quad\forall (x,z)\in\R^2.
\eea
Recalling \eqref{eq:L_eqv} and applying \eqref{eq:prox_eq}, we derive the following from \eqref{eq:der_eq}:
\bea\label{eq:use1}
(\la_2-\la_1)\proxri{\al_1,\mu_1,\la_1}{x}{z} = (\la_2\al_1-\la_1\al_2) x + (\la_2\mu_1-\la_1\mu_2) z,\quad\forall(x,z)\in\R^2.
\eea

We consider the following two cases separately.

\vp
\noindent\underline{Case 1:~$\la_1=\la_2$}\,:  Since $\vb_1\neq\vb_2$, it holds that 
\bea\label{eq:case1}
\exists (x,z) \in \R^2 : \quad\al_1 x + \mu_1 z \neq \al_2 x + \mu_2z.
\eea
However, from \eqref{eq:use1} we have that $(\al_1-\al_2) x + (\mu_1-\mu_2) z=0$ for all $(x,z)\in\R^2$. 
This contradicts \eqref{eq:case1} and completes the proof for this case.
\\

\noindent\underline{Case 2:~$\la_1\neq\la_2$}\,: Continuing from \eqref{eq:use1} we can compute that for all $(x,z)\in\R^2$
\bea\label{eq:use2}
\ellp(\proxri{\al_1,\mu_1,\la_1}{x}{z}) &=  \frac{1}{\la_1}(\al_1x+\mu_1z-\proxri{\al_1,\mu_1,\la_1}{x}{z})\nn \\
&= \frac{\al_2-\al_1}{\la_2-\la_1} x + \frac{\mu_2-\mu_1}{\la_2-\la_1} z.
\eea
By replacing $\proxri{\al_1,\mu_1,\la_1}{x}{z}$ from \eqref{eq:use1} we derive that:
\bea\label{eq:use3}
\ellp( \eps_1 x + \eps_2 z) = \eps_3 x + \eps_4 z,\quad\forall (x,z)\in\R^2,
\eea
where 
\begin{align*}
\eps_1 &= \frac{\la_2\al_1-\la_1\al_2}{\la_2-\la_1},\qquad
\eps_2 = \frac{\la_2\mu_1-\la_1\mu_2}{\la_2-\la_1},\\
\eps_3 &= \frac{\al_2-\al_1}{\la_2-\la_1},\qquad\quad\quad
\eps_4 = \frac{\mu_2-\mu_1}{\la_2-\la_1}.
\end{align*}
By replacing $x=z=0$ in \eqref{eq:use3} we find that $\ell '(0) = 0$. This contradicts the assumption of the lemma and completes the proof.
\end{proof}

\subsubsection{Strict concavity}\label{sec:EME_conc}
In this section, we study the following variant $\Gamc:\R_+\rightarrow\R$ of the expected Moreau envelope:
\bea\label{eq:Gamc}
\Gamc(\gamma):=\E\left[\envl{X}{1/\gamma}\right],
\eea
for a lower semi-continuous, proper, convex function $\ell$ and continuous random variable $X$. The expectation above is over the randomness of $X$. In Section \ref{sec:B4} we show that the function $\Gamc$ is concave in $\gamma$. Here, we prove the following statement regarding \emph{strict}-concavity of $\Gamc$ :

\begin{center}
``\,If $\ell$ is convex, continuously differentiable, and $\elld(0)\neq 0$, then $\Gamc$ is \emph{strictly} concave.\,''
\end{center}

\noindent This is summarized in Proposition \ref{propo:strict_EME2} below.

\begin{propo}[Strict concavity]\label{propo:strict_EME2} Let $\ell:\R\rightarrow\R$ be a convex, continuously differentiable function for which $\elld(0)\neq 0$. Further let $X$ be a continuous random variable in $\R$ with strictly positive density in the real line. Then, the function $\Gamc$ in \eqref{eq:Gamc} is \emph{strictly} concave in $\R_+$.
\end{propo}
\begin{proof}
Before everything, we introduce the following convenient notation:
$$\widetilde{\Gamc}_x(\gamma) : = \envl{x}{1/\gamma} \quad\text{and}\quad  p^x_\gamma:=\prox{\ell}{x}{1/\gamma}.$$
 
\noindent Note from Proposition \ref{propo:der} that $\widetilde{\Gamc}_x$ is differentiable with derivative
\bea\label{eq:Gam_der_exp}
\widetilde{\Gamc}'_x(\gamma) = \frac{1}{2} \Big(x-\prox{\ell}{x}{1/\gamma}\Big)^2.
\eea
We proceed in two steps as follows. 
First, for fixed $x\in\R$ and $\gammatwo>\gammaone$ we prove in Lemma \ref{lem:prox_der} below that 
\bea\label{eq:Gam_der}
(x-p^x_{\gammatwo})^2 - (x-p^x_{\gammaone})^2 \leq - \frac{\gammaone}{\gammatwo-\gammaone}(p^x_\gammaone-p^x_\gammatwo)^2,
\eea
 This shows that for all $x \in \R$
\begin{align}\label{eq:gamctilde}
\widetilde\Gamc_x'(\gammatwo) - \widetilde\Gamc_x'(\gammaone)\leq 0.
\end{align}

Second, we use Lemma \ref{lem:prox_diff}  to argue that the inequality is in fact strict for all $x\in\Sc$ where $\Sc\subset\R$ and  $\Pro(X\in\Sc)>0$. To be concrete, apply Lemma \ref{lem:prox_diff} for $\vb_i=(1,0,1/\gamma_i),~i=1,2$. Notice that all the assumptions of the lemma are satisfied, hence there exists interval $\Sc\subset\R$ for which $\Pro(X\in\Sc)>0$ and 
$$p^x_\gammaone\neq p^x_\gammatwo \Rightarrow (p^x_\gammaone- p^x_\gammatwo)^2>0,\quad\forall x\in\Sc.$$ Hence, from  \eqref{eq:Gam_der} it follows that 
$$(x-p^x_{\gammatwo})^2 - (x-p^x_{\gammaone})^2 < 0,\quad\forall x\in\Sc.$$
From this, and \eqref{eq:Gam_der_exp} we conclude that 
\begin{align}\label{eq:gamctildest}
\widetilde\Gamc_x'(\gammatwo) - \widetilde\Gamc_x'(\gammaone) < 0,\quad\forall x\in\Sc.
\end{align}
Thus from \eqref{eq:gamctilde} and \eqref{eq:gamctildest}, as well as the facts that $\Gamc(\gamma) = \E \left[\widetilde{\Gamc}_X(\gamma)\right]$ and $\Pro(X\in\Sc)>0$, we conclude that $\Gamc$ is strictly concave in $\R_+$.

\end{proof}

\begin{lem}\label{lem:prox_der}
 Let $\ell:\R\rightarrow\R$ be a convex, continuously differentiable function. Fix $x\in\R$ and denote 
$p_\gamma:=\prox{\ell}{x}{1/\gamma}$. Then, for any $\gamma,\gammat>0$, it holds that
\bea\label{eq:p_gamma}
(\gammat-\gamma)(p_\gammat-p_\gamma)(p_\gamma-x) + {\gammat}(p_\gammat-p_\gamma)^2 \leq 0.
\eea
Moreover, for $\gammatwo>\gammaone$, the following statement is true : 
\bea\label{eq:Gam_der_lem}
(x-p_{\gammatwo})^2 - (x-p_{\gammaone})^2 \leq - \frac{\gammaone}{\gammatwo-\gammaone}(p_\gammaone-p_\gammatwo)^2.
\eea
\end{lem}
\begin{proof}First, we prove \eqref{eq:p_gamma}. Then, we use it to prove \eqref{eq:Gam_der_lem}.

\vp
\noindent\underline{Proof of \eqref{eq:p_gamma}:} 
Consider function $g:\R\rightarrow\R$ defined as follows $g(p)=\frac{\gammat}{2}(x-p)^2+\ell(p)$. By assumption, $g$ is differentiable with derivative
$g^\prime(p)=\gammat (p-x) + \elld(p)$. Moreover, $g$ is $\gamma_2$-strongly convex. Finally, by optimality of the proximal operator (cf. Proposition \ref{propo:der}), it holds that $\gamma(x-p_\gamma) = \elld(p_\gamma)$ and $\gammat(x-p_\gammat) = \elld(p_\gammat)$. Using these, it can be computed that $g^\prime(p_\gammat)=0$ and $g^\prime(p_\gamma) = (\gammat-\gamma)(p_\gamma-x)$. 

In the following inequalities, we combine all the aforementioned properties of the function $g$ to find that
\bea
g(p_\gamma) &\geq g(p_\gammat) + \frac{\gammat}{2}(p_\gamma-p_\gammat)^2 \geq g(p_\gamma) + (\gammat-\gamma)(p_\gamma-x)(p_\gammat-p_\gamma) + {\gammat}(p_\gamma-p_\gammat)^2.\nn
\eea
This leads to the desired statement and completes the proof of \eqref{eq:p_gamma}.

\vp
\noindent\underline{Proof of \eqref{eq:Gam_der_lem}:} 
We fix $\gammatwo > \gammaone$ and apply \eqref{eq:p_gamma} two times as follows. First, applying \eqref{eq:p_gamma} for $(\gammat,\gamma)=(\gammatwo,\gammaone)$ and using the fact that $\gammatwo>\gammaone$ we find that 
\bea\label{eq:concave_1}
(p_\gammatwo-p_\gammaone)(p_\gammaone-x) \leq - \frac{\gammatwo}{\gammatwo-\gammaone}(p_\gammatwo-p_\gammaone)^2.
\eea
Second, applying \eqref{eq:p_gamma} for $(\gammat,\gamma)=(\gammaone,\gammatwo)$ and using again the fact that $\gammatwo>\gammaone$ we find that 
\bea
&(\gammaone-\gammatwo)(p_\gammaone-p_\gammatwo)(p_\gammatwo-x) + {\gammaone}(p_\gammaone-p_\gammatwo)^2 \leq 0\nn \\[4pt]
\Rightarrow\qquad&(p_\gammatwo-p_\gammaone)(p_\gammatwo-x) \leq - \frac{\gammaone}{\gammatwo-\gammaone}(p_\gammaone-p_\gammatwo)^2. \label{eq:concave_2}
\eea
Adding \eqref{eq:concave_1} and \eqref{eq:concave_2}, we have shown the desired property as follows: 
$$
(p_\gammatwo-p_\gammaone)(p_\gammatwo-x) + (p_\gammatwo-p_\gammaone)(p_\gammaone-x)  \leq - \frac{\gamma_2+\gamma_1}{\gamma_2-\gamma_1} (p_\gammaone-p_\gammatwo)^2.
$$

\end{proof}

\subsubsection{Summary of properties of \eqref{eq:EME_YS}}\label{sec:EME_Gauss}

\begin{propo}\label{propo:EME_sum}
Let $\ell:\R\rightarrow\R$ be a lsc, proper, convex function. Let $G,S\simiid\Nn(0,1)$ and function $f:\R\rightarrow\{\pm1\}$ such that the random variable $YS=f(S)S$ has a continuous strictly positive density on the real line. Then the following properties are true for the expected Moreau envelope function
\bea\label{eq:Omega}
\Omega:~(\alpha>0,\mu,\tau>0,\gamma>0) \mapsto \E\big[\env{\ell}{\alpha G + \mu SY}{\tau/\gamma}\big] : 
\eea

\noindent~(a)~ The function $\Omega$ is differentiable and its derivatives are given as follows:
\begin{align*}
\frac{\partial}{\partial \alpha} \Omega(\alpha,\mu,\tau,\gamma) &= \E\Big[G\,\envdx{\ell}{\alpha G+\mu SY}{\tau/\gamma}\Big],  \\
\frac{\partial}{\partial \mu} \Omega(\alpha,\mu,\tau,\gamma) &= \E\Big[SY\,\envdx{\ell}{\alpha G+\mu SY}{\tau/\gamma}\Big],\\
\frac{\partial}{\partial \tau} \Omega(\alpha,\mu,\tau,\gamma) &= \frac{1}{\gamma}\E\Big[\envdla{\ell}{\alpha G+\mu SY}{\tau/\gamma}\Big],\\
\frac{\partial}{\partial \gamma} \Omega(\alpha,\mu,\tau,\gamma) &= -\frac{\tau}{\gamma^2}\E\Big[\envdla{\ell}{\alpha G+\mu SY}{\tau/\gamma}\Big].
\end{align*}

\noindent~(b)~ The function $\Omega$ is jointly convex $(\alpha,\mu,\tau)$ and concave on $\gamma$.

\noindent~(c)~  The function $\Omega$ is increasing in $\alpha$.

\noindent For the statements below, further assume that $\ell$ is strictly convex and continuously differentiable with $\ellp(0)\neq 0$.

\noindent~(d)~  The function $\Omega$ is strictly convex in $(\alpha,\mu,\tau)$ and strictly concave in $\la$.

\noindent~(e)~  The function $\Omega$ is strictly increasing in $\alpha$.

\end{propo}

\begin{proof}
The statements (a),(b) and (d) follow directly by Propositions \ref{propo:der_EME},  \ref{propo:strict_EME} and  \ref{propo:strict_EME2}. It remains to prove statements (c) and (e). Let $\alpha_2>\alpha_1$. Then, there exist independent copies $G^\prime,G^{\prime\prime}$ of $G$ and $\tilde\alpha>0$ such that 
$\alpha_2G = \alpha_1G^{\prime} + \tilde\alpha G^{\prime\prime}$.
Hence, we have the following chain of inequalities:
\begin{align*}
 \Omega(\alpha_2,\mu,\tau,\gamma) &= \E\big[\env{\ell}{\alpha_1G^{\prime} + \tilde\alpha G^{\prime\prime} + \mu SY}{\tau/\gamma}\big]\nn \geq \E\big[\env{\ell}{\alpha_1G^{\prime} + \tilde\alpha \E[G^{\prime\prime}] + \mu SY}{\tau/\gamma}\big] \\
 &= \E\big[\env{\ell}{\alpha_1G^{\prime}  + \mu SY}{\tau/\gamma}\big] =  \Omega(\alpha_1,\mu,\tau,\gamma),
\end{align*}
where the inequality follows from Jensen and convexity of $\Omega$ with respect to $\alpha$ (see Statement (b) of the Proposition). This proves Statement (c). For Statement (e), note that the inequality is strict provided that $\Omega$ is strictly convex (see Statement (d) of the Proposition).
\end{proof}


\section{Proof of Theorem \ref{thm:main}}\label{sec:proof}
In this section we provide a proof sketch of Theorem \ref{thm:main}. The main technical tool that facilitates our analysis is the convex Gaussian min-max theorem (CGMT), which is an extension of Gordon's Gaussian min-max inequality (GMT). We introduce the necessary background on the CGMT in \ref{sec:CGMT}.

The CGMT has been mostly applied to linear measurements \cite{StoLASSO,OTH13,COLT,Master,miolane2018distribution}. The simple, yet central idea, which allows for this extension, is a certain projection trick inspired by \cite{PV15}. Here, we apply a similar trick, but in our setting, we recognize that it suffices to simply rotate $\x_0$ to align with the first basis vector. The simple rotation decouples the measurements $y_i$ from the last $n-1$ coordinates of the measurement vectors $\ab_i$ (see Section \ref{sec:mainproof}). While this is sufficient for LS in \cite{NIPS}, in order to study more general loss functions, we further need to combine this with a duality argument similar to that in \cite{COLT}. Second, while the steps that bring the ERM minimization to the form of a PO (see \eqref{eq:normopt_PO}) bear the aforementioned similarities to \cite{NIPS,COLT}, the resulting AO is different from the one studied in previous works. Hence, the mathematical derivations in Sections \ref{sec:B3} and \ref{sec:B4} are different. This also leads to a different system of equations characterizing the statistical behavior of ERM. Finally, in Section \ref{sec:unique}, we prove uniqueness of the solution of this system of equations using the properties of the expected Moreau envelope function studied in Section \ref{sec:strict_EME}.

\subsection{Technical tool: CGMT}\label{sec:CGMT}

\subsubsection{Gordon's Min-Max Theorem (GMT)}
The Gordon's Gaussian comparison inequality \cite{Gor88} compares the min-max value of two doubly indexed Gaussian processes based on how their autocorrelation functions compare. The inequality is quite general (see \cite{Gor88}), but for our purposes we only need its application to the following two Gaussian processes:
\begin{subequations}
\begin{align}
X_{\w,\ub} &:= \ub^T \G \w + \psi(\w,\ub),\\
Y_{\w,\ub} &:= \norm{\w}_2 \g^T \ub + \norm{\ub}_2 \h^T \w + \psi(\w,\ub),
\end{align}
\end{subequations}
where: $\G\in\mathbb{R}^{m\times n}$, $\g \in \mathbb{R}^m$, $\h\in\mathbb{R}^n$, they all have entries iid Gaussian; the sets $\mathcal{S}_{\w}\subset\R^n$ and $\mathcal{S}_{\ub}\subset\R^m$ are compact; and, $\psi: \mathbb{R}^n\times \mathbb{R}^m \to \mathbb{R}$. For these two processes, define the following (random) min-max optimization programs, which we refer to as the \emph{primary optimization} (PO) problem and the \emph{auxiliary optimization} (AO).  
\begin{subequations}
\begin{align}\label{eq:PO_loc}
\widetilde{\Phi}(\G)&=\min\limits_{\w \in \mathcal{S}_{\w}} \max\limits_{\ub\in\mathcal{S}_{\ub}} X_{\w,\ub},\\
\label{eq:AO_loc}
\phi(\g,\h)&=\min\limits_{\w \in \mathcal{S}_{\w}} \max\limits_{\ub\in\mathcal{S}_{\ub}} Y_{\w,\ub}.
\end{align}
\end{subequations}
According to Gordon's comparison inequality, for any $c\in\R$, it holds:
\begin{equation}\label{eq:gmt}
\mathbb{P}\left( \widetilde{\Phi}(\G) < c\right) \leq 2 \mathbb{P}\left(  \phi(\g,\h) < c \right).
\end{equation}
In other words, a high-probability lower bound on the AO is a high-probability lower bound on the PO. The premise is that it is often much simpler to lower bound the AO rather than the PO. To be precise, \eqref{eq:gmt} is a slight reformulation of Gordon's original result proved in \cite{COLT}.
\subsubsection{Convex Gaussian Min-Max Theorem (CGMT)}
The proof of Theorem \ref{thm:main} builds on the CGMT \cite{COLT}. 
For ease of reference we summarize here the essential ideas of the framework following the presentation in \cite{Master}; please see \cite[Section~6]{Master} for the formal statement of the theorem and further details.
The CGMT is an extension of the GMT and it asserts that the AO in \eqref{eq:AO_loc} can be used to tightly infer properties of the original (PO) in \eqref{eq:PO_loc}, including the optimal cost and the optimal solution.
According to the CGMT \cite[Theorem 6.1]{Master}, if the sets $\mathcal{S}_{\w}$ and $\mathcal{S}_{\ub}$ are convex and $\psi$ is continuous \emph{convex-concave} on $\mathcal{S}_{\w}\times \mathcal{S}_{\ub}$, then, for any $\nu \in \mathbb{R}$ and $t>0$, it holds
\begin{equation}\label{eq:cgmt}
\mathbb{P}\left(\, \abs{\widetilde{\Phi}(\G)-\nu} > t\right) \leq 2 \mathbb{P}\Big(\,\abs{\phi(\g,\h)-\nu} > t \Big).
\end{equation}
In words, concentration of the optimal cost of the AO problem around $\mu$ implies concentration of the optimal cost of the corresponding PO problem around the same value $\mu$.  Moreover, starting from \eqref{eq:cgmt} and under strict convexity conditions, the CGMT shows that concentration of the optimal solution of the AO problem implies concentration of the optimal solution of the PO to the same value. For example, if minimizers of \eqref{eq:AO_loc} satisfy $\norm{\w^\ast(\g,\h)}_2 \to \zeta^\ast$ for some $\zeta^\ast>0$, then, the same holds true for the minimizers of \eqref{eq:PO_loc}: $\norm{\w^\ast(\G)}_2 \to \zeta^\ast$ \cite[Theorem 6.1(iii)]{Master}. Thus, one can analyze the AO to infer corresponding properties of the PO, the premise being of course that the former is simpler to handle than the latter.

\subsection{Applying the CGMT to ERM for binary classification} \label{sec:mainproof}


In this section, we show how to apply the CGMT to \eqref{eq:gen_opt}. For convenience, we drop the subscript $\ell$ from $\xh_\ell$ and simply write
\begin{equation} \label{eq:opt}
\xh = \arg\min_{\x} \frac{1}{m} \sum_{i=1}^{m} \ell(y_i \ab_i^T \x), 
\end{equation}
where the measurements $y_i,~i\in[m]$ follow \eqref{eq:gen_model}. By rotational invariance of the Gaussian distribution of the measurement vectors $\ab_i,~i\in[m]$, we assume without loss of generality that $\x_0 = [1,0,...,0]^T$. Denoting $y_i\ab_i^T\x$ by $u_i$, \eqref{eq:opt} is equivalent to the following min-max optimization: 
\begin{align}\label{eq:mmbn}
\min_{\ub,\x} \max_{\pmb{\beta}} \frac{1}{m} \sum_{i=1}^{m}\ell(u_i) + \frac{1}{m}\sum_{i=1}^{m}\beta_i u_i 
- \frac{1}{m}\sum_{i=1}^{m}\beta_i y_i \ab_i^T \x .
\end{align}
Now, let us define 
$$\ab_i=[s_i;\tilde{\ab}_i],~i\in[m]\quad\text{ and }\quad \x=[x_1;\widetilde{\x}],$$ such that $s_i$ and $x_1$ are the first entries of $\ab_i$ and $\x$, respectively. Note that in this new notation \eqref{eq:gen_model} becomes:
\bea\label{eq:y_pf}
y_i =f(s_i),
\eea
and 
\bea\label{eq:corr_pf}
\corr{\xh}{\x_0} = \frac{\widehat{x}_1}{\sqrt{\widehat{x}_1^2+\|\widetilde{\xh}\|_2^2}},
\eea
where we have decomposed $\xh=[\widehat{x}_1;{\widetilde{\xh}}]$.
Also, \eqref{eq:mmbn} is written as 
\begin{dmath*}
\min_{\ub,\x}\max_{\pmb{\beta}}  \frac{1}{m}\sum_{i=1}^{m} \ell(u_i) + \frac{1}{m}\sum_{i=1}^{m} \beta_i u_i+ \frac{1}{m}\sum_{i=1}^{m} \beta_i y_i \tilde{\ab}_i^T \widetilde{\x} - \frac{1}{m}\sum_{i=1}^{m} \beta_i y_i s_i x_1,   
\end{dmath*}
or, in matrix form:
\begin{dmath}\label{eq:normopt_PO}
\min_{\ub,\x}\max_{\pmb{\beta}}~ \frac{1}{m}\betab^T\mathbf{D}_y\widetilde{\A}\widetilde{\x} + \frac{1}{m}x_1\pmb{\beta}^T\mathbf{D}_y\s+ \frac{1}{m}\pmb{\beta}^T\ub +\frac{1}{m}\sum_{i=1}^m \ell (u_i).
\end{dmath}
where $\mathbf{D}_\y := {\rm{diag}}(y_1,y_2,...,y_m)$ is a diagonal matrix with $y_1,y_2,...y_m$ on the diagonal, $\s=[s_1,\ldots,s_m]^T$ and $\widetilde{\A}$ is an $m\times (n-1)$ matrix with rows $\tilde{\ab}_i^T,~i\in[m]$.

In \eqref{eq:normopt_PO} we recognize that the first term has the bilinear form required by the GMT in \eqref{eq:PO_loc}. The rest of the terms form the function $\psi$ in \eqref{eq:PO_loc}: they are independent of $\widetilde\A$ and convex-concave as desired by the CGMT. Therefore, we have expressed \eqref{eq:opt} in the desired form of a PO and for the rest of the proof we will analyze the probabilistically equivalent  AO problem. In view of \eqref{eq:AO_loc}, this is given as follows,
\begin{dmath}\label{eq:normopt}
\min_{\ub,\x}\max_{\pmb{\beta}} \frac{1}{m} \left\lVert\widetilde{\x}\right\rVert_2 \mathbf{g}^T \mathbf{D}_y \pmb{\beta} + \frac{1}{m} \left\lVert \mathbf{D}_y\pmb{\beta}\right\rVert_2\mathbf{h}^T\widetilde{\x} - \frac{1}{m}x_1\pmb{\beta}^T\mathbf{D}_y\s+ \frac{1}{m}\pmb{\beta}^T\ub + \frac{1}{m}\sum_{i=1}^m \ell (u_i) \,,
\end{dmath}
where as in \eqref{eq:AO_loc} $\mathbf{g}\sim\mathcal{N}(0,I_m)$ and $\mathbf{h}\sim\mathcal{N}(0,I_{n-1})$.

\subsection{Analysis of the Auxiliary Optimization}\label{sec:B3}
Here, we show how to analyze the AO in \eqref{eq:normopt}. To begin with, note that $y_i\in\{\pm1\}$, therefore $\mathbf{D}_\y\mathbf{g} \sim\mathcal{N}(0,I_m)$ and $\left\lVert\mathbf{D}_\y\pmb{\beta}\right\rVert_2 = \left\lVert\pmb{\beta}\right\rVert_2$. Also, let us denote the first entry $x_1$ of $\x$ as 
$$
\mu:=x_1.
$$ 
From \cite[Lem.~A.3]{Master}, instead of the AO in \eqref{eq:normopt}, it suffices to analyze the following version
\begin{dmath}\label{eq:normopt2}
\min_{\ub,\mu,\alpha\geq 0}\max_{\pmb{\beta}} \min_{\|\widetilde{\x}\|_2=\alpha} \frac{1}{m} \left\lVert\widetilde{\x}\right\rVert_2 \mathbf{g}^T \mathbf{D}_y \pmb{\beta} + \frac{1}{m} \left\lVert \mathbf{D}_y\pmb{\beta}\right\rVert_2\mathbf{h}^T\widetilde{\x} - \frac{1}{m}\mu\pmb{\beta}^T\mathbf{D}_y\s+ \frac{1}{m}\pmb{\beta}^T\ub + \frac{1}{m}\sum_{i=1}^m \ell (u_i) \,,
\end{dmath}
\noindent Note that the ``$\min_{\ub,\x} \max_{\betab} $'' problem in \eqref{eq:normopt} is equivalent to a ``$\min_{\ub,\mu,\alpha\geq 0} \min_{\|\widetilde{x}\|_2=\alpha} \max_{\betab}$'' problem. Compared to that, the order of min-max in \eqref{eq:normopt2} is now flipped; see the discussion in \cite[Sec.~A.6]{Master}\footnote{Here we skip certain technical details in this argument regarding boundedness of the constraint sets in \eqref{eq:normopt}. While they are not trivial, they can be handled with the same techniques used in \cite{Master,PhaseLamp}.}. Now it is now possible to optimize over the direction of $\widetilde{\x}$, which leads to the following:
\begin{dmath}\label{eq:alpha_pf}
\min_{\alpha\ge0,\mu,\ub}~\max_{\pmb{\beta}} ~ \frac{1}{m}\alpha\mathbf{g}^T\pmb{\beta} - \frac{\alpha}{m}\left\lVert\pmb{\beta}\right\rVert_2\left\lVert\mathbf{h}\right\rVert_2 - \frac{1}{m}\mu \mathbf{s}^T\mathbf{D}_{\y} \pmb{\beta} + \frac{1}{m}\pmb{\beta}^T \mathbf{u} + \frac{1}{m} \sum_{i=1}^m\ell(u_i). 
\end{dmath}
Next, let $\gamma := \frac{\left\lVert\pmb{\beta}\right\rVert_2}{\sqrt{m}}$ and optimize over the direction of $\betab$ to yield
\begin{dmath}\label{eq:gamma}
\min_{\alpha\ge0,\ub,\mu}~\max_{\gamma\ge0}~\frac{\gamma}{\sqrt{m}}\left\lVert \alpha\mathbf{g}-\mu\mathbf{D}_\y\mathbf{s}+\ub\right\rVert_2 - \frac{\alpha}{\sqrt{m}}\gamma\left\lVert\mathbf{h}\right\rVert_2 + \frac{1}{m}\sum_{i=1}^m\ell(u_i).
\end{dmath}
To continue, we utilize the fact that for all $x\in\R$, 
$\min_{\tau>0}\frac{\tau}{2} + \frac{x^2}{2\tau m} = \frac{x}{\sqrt{m}}$. Hence 
\bea\nn
\frac{\gamma}{\sqrt{m}}\left\lVert \alpha\mathbf{g}-\mu\mathbf{D}_\y\mathbf{s}+\ub\right\rVert_2 =  \min_{\tau>0}~\frac{\gamma\tau}{2} + \frac{\gamma}{2\tau m}{\left\lVert  -\alpha\mathbf{g}+\mu\mathbf{D}_\y\mathbf{s}-\ub \right\rVert}_2 ^2\,.   
\eea
With this trick, the optimization over $\ub$ becomes separable over its coordinates $u_i,~i\in[m]$. By inserting this in \eqref{eq:gamma} we have
\begin{dmath*}
\min_{\alpha\ge0,\tau>0,\ub,\mu}\max_{\gamma\ge0}\frac{\gamma\tau}{2} - \frac{\alpha}{\sqrt{m}}\gamma\left\lVert\mathbf{h}\right\rVert_2 +  \frac{\gamma}{2\tau m}\sum_{i=1}^{m}(-\alpha g_i +\mu y_is_i- u_i) ^2+ \frac{1}{m}\sum_{i=1}^m\ell(u_i),
\end{dmath*}
Now, we show that the objective function above is convex-concave. Clearly, the function is linear (thus, concave in $\gamma$). Moreover, from Lemma \ref{lem:H_cvx}, the function $\frac{1}{2\tau}(\alpha g_i +\mu y_is_i- u_i) ^2$ is jointly convex in $(\alpha,\mu,u_i,\tau)$. The rest of the terms are clearly convex and this completes the argument. Hence, with a permissible change in the order of min-max, we arrive at  the following convenient form:
\begin{dmath} \label{eq:binfty}
\min_{\mu,\alpha\ge0,\tau>0}~\max_{\gamma\ge0}~\frac{\gamma \tau}{2}- \frac{\alpha}{\sqrt{m}}\gamma\left\lVert\mathbf{h}\right\rVert_2  + \frac{1}{m}\sum_{i=1}^m \env{\ell}{-\alpha g_i+\mu s_iy_i}{\frac{\tau}{\gamma}},
\end{dmath}
where recall the definition of the Moreau envelope in \eqref{eq:M_def2}.
As to now, we have reduced the AO into a random min-max optimization over only four scalar variables in \eqref{eq:binfty}. For fixed $\mu,\alpha,\tau,\gamma$, direct application of the weak law of large numbers, shows that the objective function of \eqref{eq:binfty} converges in probability to the following as  $m,n\rightarrow\infty$ and $\frac{m}{n}=\delta$:
$$
\gamma\frac{\tau}{2}-\frac{\alpha\gamma}{\sqrt{\delta}}+ \mathbb{E}\left[\env{\ell}{\alpha G+\mu YS}{\frac{\tau}{\gamma}} \right],
$$
where $G,S\sim\mathcal{N}(0,1)$ and $Y\sim f(S)$ (in view of \eqref{eq:y_pf}). 
Based on that, it can be shown (similar arguments are developed in \cite{Master,PhaseLamp}) that the random optimizers $\alpha_n$ and $\mu_n$ of \eqref{eq:binfty} converge to the deterministic optimizers $\alpha$ and $\mu$ of the following (deterministic) optimization problem (whenever these are bounded as the statement of the theorem requires):
\begin{dmath}\label{eq:det}
\min_{\alpha\ge0,\mu,\tau>0}\max_{\gamma\ge0}~ \gamma\frac{\tau}{2}-\frac{\alpha\gamma}{\sqrt{\delta}}+ \mathbb{E}\left[\env{\ell}{\alpha G+\mu YS}{\frac{\tau}{\gamma}} \right].
\end{dmath}
At this point, recall that $\alpha$ represents the norm of $\tilde\x$ and $\mu$ the value of $x_1$. Thus, in view of (i) \eqref{eq:corr_pf}, (ii) the equivalence between the PO and the AO, and, (iii) our derivations thus far we have that with probability approaching 1,
\bea
\lim_{n\rightarrow+\infty}\corr{\xh}{\x_0} = \frac{\mu}{\sqrt{\mu^2+\alpha^2}},\nn
\eea
where $\mu$ and $\alpha$ are the minimizers in \eqref{eq:det}. The three equations in \eqref{eq:eq_main} are derived by the first-order optimality conditions of the optimization in \eqref{eq:det}. We show this next.

\subsection{Convex-Concavity and First-order Optimality Conditions}\label{sec:B4}
First, we prove that the objective function in \eqref{eq:det} is convex-concave. For convenience define the function $F:\R^4\rightarrow\R$ as follows
\bea\label{eq:det_2}
 F(\alpha,\mu,\tau,\gamma) := \frac{\gamma\tau}{2}-\frac{\alpha\gamma}{\sqrt{\delta}}+ \mathbb{E}\left[\env{\ell}{\alpha G+\mu YS}{\frac{\tau}{\gamma}} \right].
\eea
Based on Lemma \ref{lem:Moreau_cvx}, it immediately follows that if $\ell$ is convex, $F$ is jointly convex in $(\al,\mu,\tau)$. To prove concavity of $F$ based on $\gamma$  it suffices to show that $\env{\ell}{x}{1/\gamma}$ is concave in $\gamma$ for all $x\in\R$. To show this we note that
\begin{align*}
\env{\ell}{x}{1/\gamma} = \min_{u} \frac{\gamma}{2} (x-u)^2 + \ell(u),
\end{align*}
which is the point-wise minimum of linear functions of $\gamma$. Thus, using Proposition \ref{lem:min_cvx}(b), we conclude that  $\env{\ell}{x}{1/\gamma}$ is concave in $\gamma$. This completes the proof of convex-concavity of the function $F$ in \eqref{eq:det_2} when $\ell$ is convex. 
By direct differentiation and applying Proposition \ref{propo:EME_sum}(a), the first order optimality conditions of the min-max optimization in \eqref{eq:det} are as follows:
\begin{subequations}\label{eq:four_eq}
\bea
\label{eq:1}
\mathbb{E}\left[SY\cdot\envdx{\ell}{\ourx}{\frac{\tau}{\gamma}}\right] &= 0,
\\
\label{eq:2}
\mathbb{E}\left[G\cdot \envdx{\ell}{\ourx}{\frac{\tau}{\gamma}} \right]&=\frac{\gamma}{\sqrt{\delta}},
\\
\label{eq:3}
\frac{\gamma}{2} + \frac{1}{\gamma}\mathbb{E}\left[\envdla{\ell}{\ourx}{\frac{\tau}{\gamma}}\right]&=0,
\\
\label{eq:4}
-\frac{\alpha}{\sqrt{\delta}}-\frac{\tau}{\gamma^2}\mathbb{E}\left[\envdla{\ell}{\ourx}{\frac{\tau}{\gamma}}\right]+\frac{\tau}{2} &= 0.
\eea
\end{subequations}
Next, we show how these equations simplify to the following system of equations (same as \eqref{eq:eq_main}: 
\begin{subequations}\label{eq:reg_main}
\bea
 \Exp\bigg[Y\, S \cdot\envdx{\ell}{\ourx}{\la}  \bigg]&=0 , \label{eq:mureg_main}\\
 {\la^2}\,{\delta}\,\Exp\bigg[\,\left(\envdx{\ell}{\ourx}{\la}\right)^2\,\bigg]&=\alpha^2 ,
\label{eq:alphareg_main}\\
\lambda\,\delta\,\E\bigg[ G\cdot \envdx{\ell}{\ourx}{\la}  \bigg]&=\alpha .
\label{eq:lambdareg_main}
\eea
\end{subequations}
Let $\lambda := \frac{\tau}{\gamma}$. 
First, \eqref{eq:mureg_main} is immediate from equation \eqref{eq:1}.
Second, substituting $\gamma$ from \eqref{eq:3} in \eqref{eq:4} yields $\tau=\frac{\alpha}{\sqrt{\delta}}$ or $\gamma=\frac{\alpha}{\lambda \sqrt{\delta}}$, which together with \eqref{eq:2} leads to \eqref{eq:lambdareg_main}. Finally, \eqref{eq:alphareg_main} can be obtained by substituting $\gamma = \frac{\alpha}{\lambda\sqrt{\delta}}$ in \eqref{eq:3} and using the fact that (see Proposition \ref{propo:der}):
\begin{equation*}
\envdla{\ell}{\ourx}{\lambda} = -\frac{1}{2}(\envdx{\ell}{\ourx}{\lambda})^2.
\end{equation*}

\subsection{On the uniqueness of solutions to \eqref{eq:reg_main}: Proof of Proposition \ref{propo:Unique}}\label{sec:unique}
Here we prove the claim of Proposition \ref{propo:Unique} through the following lemmas. As we discussed in Remark \ref{rem:unique}, the main part of the proof is showing strict convex-concavity of $F$ in \eqref{eq:DO}. Lemma \ref{lem:F} proves that this is the case, and Lemmas \ref{lem:unique_1} and \ref{lem:unique_2} show that this is sufficient for the uniqueness of solutions to \eqref{eq:reg_main}. When put together, these complete the proof of Proposition \ref{propo:Unique}.

\begin{lem}[Strict convex-concavity of \eqref{eq:det_2}]\label{lem:F}
Let $\ell:\R\rightarrow\R$ be proper and strictly convex function. Further assume that $\ell$ is continuously differentiable with $\ellp(0)\neq 0$. Also assume that $SY$ has positive density in the real line. Then, the function $F:\R^4\rightarrow\R$ defined in \eqref{eq:det_2} is strictly convex in $(\alpha,\mu,\tau)$ and strictly concave in $\gamma$.
\end{lem}

\begin{proof}
The claim follows directly from the strict convexity-concavity properties of the expected Moreau-envelope proved in Proposition \ref{propo:strict_EME} and \ref{propo:strict_EME2}. Specifically, we apply Proposition \ref{propo:EME_sum}.
\end{proof}

\begin{lem}\label{lem:unique_1}
 If the objective function in \eqref{eq:det_2} is strictly convex in $(\alpha,\mu,\tau)$ and strictly concave in $\gamma$, then \eqref{eq:four_eq} has a unique solution $(\alpha,\mu,\tau,\gamma)$.
\end{lem}

\begin{proof}
Let $(\al_i,\mu_i,\tau_i,\gamma_i),\; i=1,2,$ be two different saddle points of \eqref{eq:det_2}. For convenience, let $\x_i := (\al_i,\mu_i,\tau_i)$ for $i=1,2$. By strict-concavity in $\gamma$, for fixed values of $\x:=(\alpha,\mu,\tau)$, the value of $\gamma$ maximizing $F(\x,\gamma)$ is unique. Thus, if $\x_1 = \x_2$ then it must hold that $\gamma_1=\gamma_2$, which is a contraction to our assumption of $(\x_1,\gamma_1) \neq (\x_2,\gamma_2)$.  Similarly, we can use strict-convexity to derive that $\gamma_1\neq\gamma_2$. Then based on the definition of the saddle point, and strict convexity-concavity, the following two relations hold for $i=1,2$:
\begin{align*}
 F(\x_i,\gamma)< F(\x_i,\gamma_i) <F(\x,\gamma_i), \quad \text{for all  } \x \neq \x_i, \gamma\neq\gamma_i.
\end{align*}
We choose $\x =\x_2, \gamma= \gamma_2$ for $i=1$ and $\x =\x_1, \gamma= \gamma_1$ for $i=2$ to find
\begin{align*}
 F(\x_1,\gamma_2)< F(\x_1,\gamma_1) <F(\x_2,\gamma_1), \\[4pt]
 F(\x_2,\gamma_1)< F(\x_2,\gamma_2) <F(\x_1,\gamma_2).
\end{align*}
From the above, it follows that $F(\x_1,\gamma_1) < F(\x_2,\gamma_2)$ and $F(\x_1,\gamma_1) > F(\x_2,\gamma_2)$, which is a contradiction. This completes the proof.
\end{proof}

\begin{lem}\label{lem:unique_2} If \eqref{eq:four_eq} has a unique solution $(\alpha^\star,\mu^\star,\tau^\star,\gamma^\star)$ then \eqref{eq:reg_main} has a unique solution $(\alpha^\star,\mu^\star,\lambda^\star)$.
\end{lem}

\begin{proof}
First, following the same approach of deriving the equations \eqref{eq:reg_main} from \eqref{eq:four_eq} in Section \ref{sec:B4}, it is easy to see that existence of solution $(\alpha_1,\mu_1,\tau_1,\gamma_1)$ to \eqref{eq:four_eq} implies existence of solution $(\alpha_1,\mu_1,\la_1:=\frac{\tau_1}{\gamma_1})$ to \eqref{eq:reg_main}. Now, for the sake of contradiction to the statement of the lemma, assume that there are two different triplets $\vb_1:=(\alpha_1,\mu_1,\la_1)$ and $\vb_2:=(\alpha_2,\mu_2,\la_2)$ with $\al_1,\al_2,\la_1,\la_2 >0$ and satisfying \eqref{eq:reg_main}. Then, we can show that both $\w_i := (\alpha_i,\mu_i,\tau_i,\gamma_i)\,~i =1,2,$ such that: 
\begin{align*}
\tau_i := \frac{\al_i}{\sqrt{\delta}}, \quad \gamma_i = \frac{\al_i}{\la_i\sqrt{\delta}},\quad\quad i=1,2,
\end{align*}
satisfy the system of equations in \eqref{eq:four_eq}. However since $\vb_1\neq\vb_2$, it must be that $\w_1\neq\w_2$. This contradicts the assumption of uniqueness of solutions to \eqref{eq:four_eq} and completes the proof.
\end{proof}

\section{On Theorem \ref{sec:lem}}\label{sec:prooflem}
\begin{figure}
  \centering
    \includegraphics[width=9cm,height=7cm]{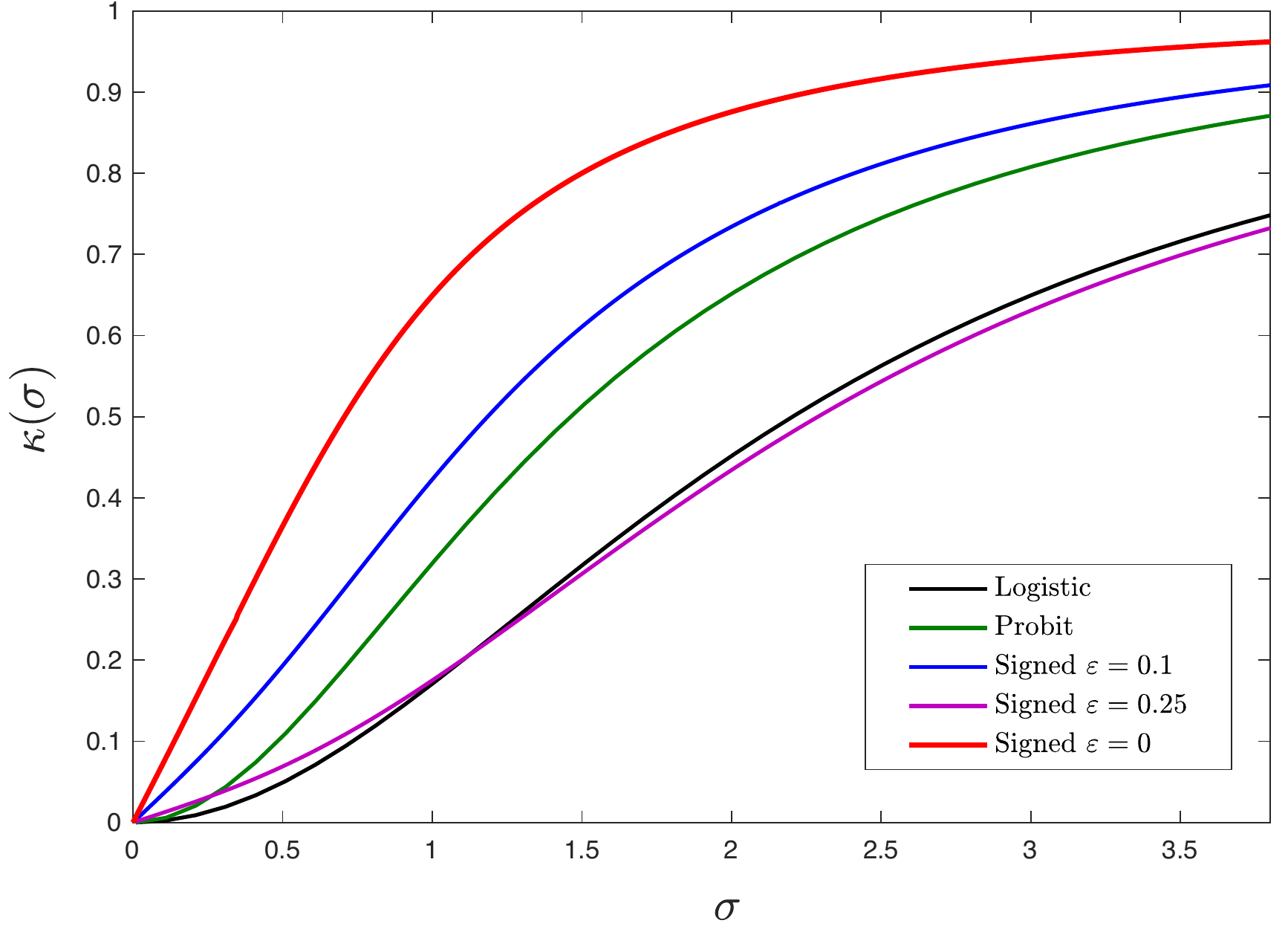}
      \caption{\footnotesize The value of $\kappa(\sig)$ as in Theorem \ref{sec:lem} for various measurement models. Since $\kappa(\sig)$ is a monotonic function of $\sigma$, the solution to $\kappa(\sigma) = 1/\delta$ determines the minimum possible value of $\sigma$. }
          \label{fig:sigma}
 \end{figure}
\subsection{On the uniqueness of solutions to the Equation $\kappa(\sig)=\frac{1}{\delta}$}
The existence of a solution to the equation $\kappa(\sig)=\frac{1}{\delta}$ was proved in the previous Section. However it is not clear if the solution to this equation is unique i.e., for any $\delta>1$ there exists only one $\sig_{\rm opt}>0$ such that $\kappa(\sig_{\rm opt}) = \frac{1}{\delta}.$ If this is the case then the Equation \eqref{eq:lem} in Theorem \ref{sec:lem} can be equivantly written as $$\sig_{\rm opt} = \sig, \text{ s.t. }\kappa(\sig) = \frac{1}{\delta}.$$
Although we do not prove this claim, our numerical experiments in Figure \ref{fig:sigma} show that $\kappa(\cdot)$ is a monotonic function for Noisy-signed (see Section \ref{sec:noisysigned} for the definition), Logistic and Probit measurements, implying the uniqueness of solution to the equation $\kappa(\sig)=\frac{1}{\delta}$ for all $\delta>1$. 


\subsection{Distribution of $SY$ in special cases}\label{sec:psy}
We derive the following densities for $SY$ for the special cases : \\

\noindent$\bullet~$  {\small\emph{Signed}}: $p_{SY}(w) = \sqrt{\frac{2}{\pi}}\exp(-w^2/2)\,\mathbbm{1}_{\{w\ge0\}}.$\\

\noindent$\bullet~$  {\small\emph{Logistic}}: $p_{SY} (w)= \sqrt{\frac{2}{\pi}}\,\frac{\exp({-w^2/2}) }{1+\exp(-w)}.$ \\

\noindent$\bullet~$  {\small\emph{Probit}}: $p_{SY} (w)= \sqrt{\frac{2}{\pi}}\,\Phi(w)\exp({-w^2/2}).$ \\

In particular we numerically observe that for Logistic and Probit models, the resulting densities are similar to the density of a gaussian distribution derived according to $\mathcal{N}(\Exp[SY],\text{Var}[SY])$. Figure \ref{fig:p_SY} illustrates this similarity for these two models. As it was discussed in Corollary \ref{cor:Stam} this similarity results in the tightness of the lower bound achieved for $\sig_{\rm opt}$ in Equation \eqref{eq:cor_stam}.

\begin{figure}
  \centering
    \includegraphics[width=7.8cm,height=6.4cm]{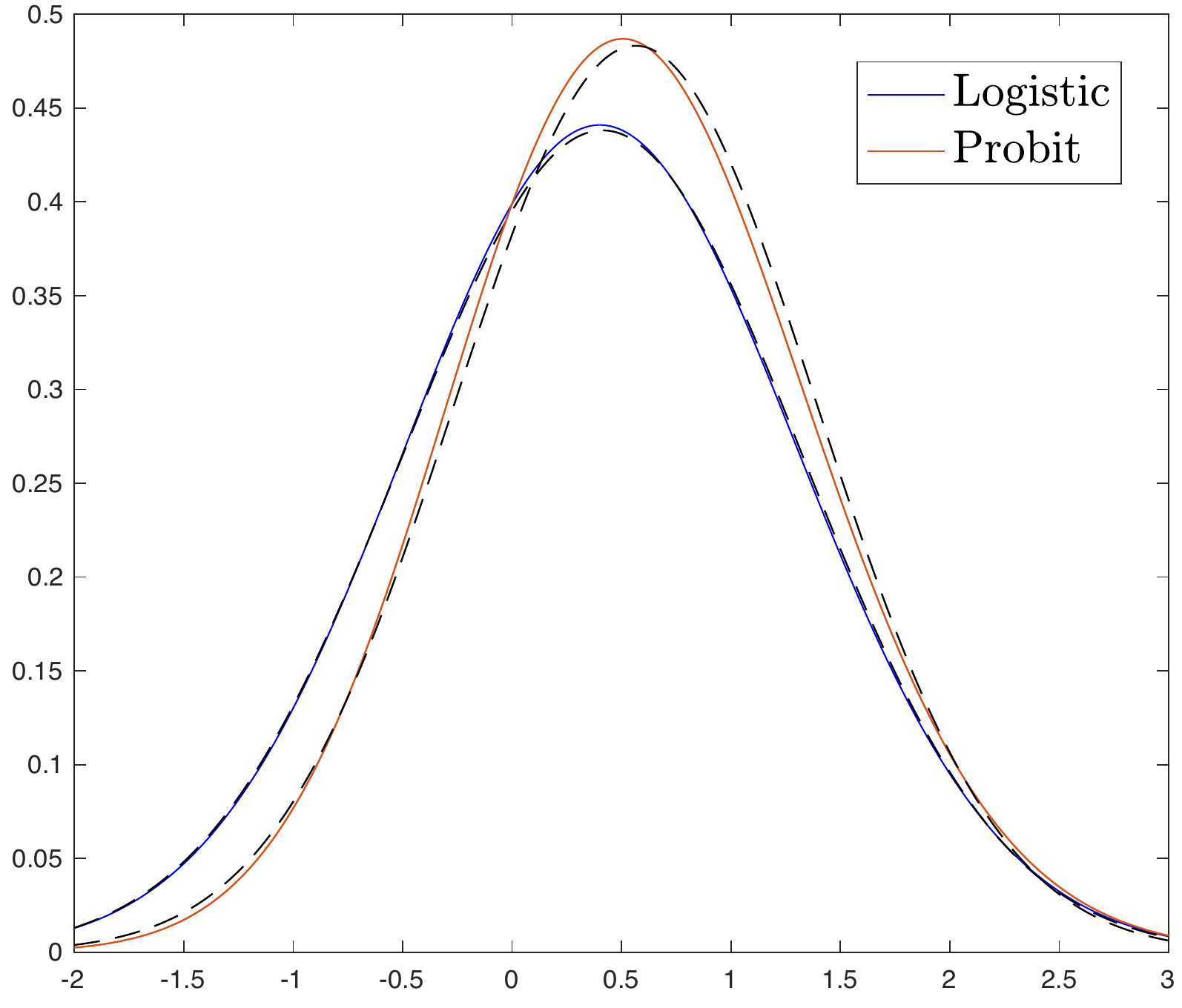}
      \caption{\footnotesize Probability distribution function of $SY$ for the Logistic and Probit models compared with the probability distribution function of the Gaussian random variable (dashed lines) with the same mean and variance i.e., $\mathcal{N}(\Exp[SY], \text{Var}[SY])$.}
          \label{fig:p_SY}
 \end{figure}
 
\section{On Theorem \ref{thm:opt_loss}}
\subsection{Completing the proof}
The following proposition gives a recipe to invert Moreau envelope functions and was used in the proof of Theorem \ref{thm:opt_loss}.
\begin{propo}[Inverse of the Moreau envelope]\cite[Result \,23 in the appendix]{advani2016statistical}\label{propo:inverse}
For $\la>0$ and $f$ a convex, lower semi-continuous function such that $g(\cdot) =\env{f}{\cdot}{\la}$, the Moreau envelope can be inverted so that $f(\cdot) = -\env{-g}{\cdot}{\la}.$
\end{propo}

\subsection{On the convexity of optimal loss function}\label{sec:opt_FL}
Here we provide a sufficient condition for $\ell_{\rm{opt}}(w)$ to be convex. 
\begin{lem} \label{lem:convex}
The optimal loss function as defined in Theorem \ref{thm:opt_loss} is convex if 
$$(\log(p_{W_{\sig}}))^{\prime\prime}(w)\le-\frac{1}{\sig^2+1},  \;\;\;\text{\;for all }w\in\mathbb{R} \;\text{and}\; \sig \ge 0.
$$
\end{lem}
\begin{proof}
Using \eqref{eq:FLME} optimal loss function is written in the following form
\begin{align}\label{eq:fen_lossapp}
\ell_{\rm{opt}}(w) = \left(q+\alpha_1q+\alpha_2\log(p_{W_{\rm opt}})\right)^\star(w) - q(w).
\end{align}
Next we prove that $q+\alpha_1q+\alpha_2\log(p_{W_{\rm opt}})$ is a convex function. We first show that both $\alpha_1$ and $\alpha_2$ are positive numbers for all values of $\sig_{\rm opt}$. We first note that since $G$ and $SY$ are independent random variables $\sig_{\rm opt}^2\Ic(W_{\rm opt}) < \sig_{\rm opt}^2\Ic(\sig_{\rm opt}\,G) = 1$. Therefore 
\begin{align}\label{eq:pos_first}
1-\sig_{\rm opt}^2\Ic(W_{\rm opt})>0.
\end{align}
Additionally following Cramer-Rao bound \cite{barron1984monotonic} for fisher information, it yields that :
\begin{align*}
\Ic(W_{\rm opt}) &> \frac{1}{\Exp\left[(W_{\rm opt}-\Exp[W_{\rm opt}])^2\right]}\\
&=\frac{1}{1+\sig_{\rm opt}^2-(\Exp[SY])^2}.
\end{align*}
Using this inequality for $\Ic(W
_{\rm opt})$ we derive that 
\begin{align}\label{eq:pos_second}
\sig_{\rm opt}^2\Ic(W_{\rm opt})+\Ic(W_{\rm opt})-1 > 0.
\end{align}
From \eqref{eq:pos_first} and \eqref{eq:pos_second} it follows that $\alpha_1, \alpha_2 >0$.\\
Based on the definition of the random variable $W_{\rm opt}$: 
\begin{align*}
\log p_{W_{\rm opt}}(w) = -w^2/(2\sig_{\rm opt}^2)  \,+  \log \int_{-\infty}^{\infty}{\exp\left((2wz-z^2)/{2\sig_{\rm opt}^2}\right)}\,p_{SY}(z)\, \text{d} z + c,
\end{align*}
where $c$ is a constant independent of $w$. By differentiating twice we see that $$ \log \int_{-\infty}^{\infty}{\exp\left((2wz-z^2)/{2\sig_{\rm opt}^2}\right)}\,p_{SY}(z)\, \text{d} z$$ is a convex function of $w$. Therefore for proving that $q+\alpha_1q+\alpha_2\log(p_{W_{\rm opt}})$ is a convex function it is sufficient to prove that $(1+\alpha_1-\alpha_2/\sig_{\rm opt}^2)q$ is a convex function or equivalently
$1+\alpha_1-\alpha_2/\sig_{\rm opt}^2 \ge 0$.  By replacing values of $\alpha_1,\alpha_2$ and recalling the equation for $\sig_{\rm opt}$ it yields that 
$$
1+\alpha_1-\alpha_2/\sig_{\rm opt}^2 = 0,
$$
which implies the convexity of $q+\alpha_1q+\alpha_2\log(p_{W_{\rm opt}})$. For obtaining the derivative of $\ell_{opt}$, we use the result in \cite[Cor. 23.5.1]{rockafellar1970convex} which states that for a convex function $f$
 $$(f^\star)^{\prime} = (f^{\prime})^{-1}.$$ Therefore following \eqref{eq:fen_lossapp}
 \begin{align}\label{eq:eq_lprime}
 \ell_{\rm opt}^\prime (w)=(q^\prime+\alpha_1q^\prime+\alpha_2(\log(p_{W_{\rm opt}}))^\prime)^{-1}(w)-w.
 \end{align}
By differentiating again and using the properties of inverse function it yields that
\begin{align}\label{eq:lprimeprime}
\ell_{\rm opt}^{\prime\prime}(w) = \frac{1}{1+\alpha_1+\alpha_2(\log(p_{W_{\rm opt}}))^{\prime\prime}(g(w))}-1,
\end{align}
where
$$
g(w) := (q^\prime+\alpha_1q^\prime+\alpha_2(\log(p_{W_{\rm opt}}))^\prime)^{-1}(w).
$$
Note that denominator of \eqref{eq:lprimeprime} is nonnegative since it is second derivative of a convex function. Therefore it is evident from \eqref{eq:lprimeprime} that a sufficient condition for the convexity of $\ell_{\rm opt}$ is that 
$$
\alpha_1+\alpha_2(\log(p_{W_{\rm opt}}))^{\prime\prime}(w) \le0, \;\;\;\; \text{for all}\;w\in\mathbb{R},
$$
or
$$
1-\sig_{\rm opt}^2\Ic(W_{\rm opt}) + (\log(p_{W_{\rm opt}}))^{\prime\prime}(w) \le 0.
$$
This condition is satisfied if the statement of the lemma holds for $\sig = \sig_{\rm opt}$ : 
\begin{align*}
&1-\sig_{\rm opt}^2\Ic(W_{\rm opt}) + (\log(p_{W_{\rm opt}}))^{\prime\prime}(w) \le 1-\sig_{\rm opt}^2\Ic(W_{\rm opt})  -\frac{1}{1+\sig_{\rm opt}^2} <0,
\end{align*}
where we used \eqref{eq:pos_second} in the last inequality. This concludes the proof.
\end{proof}
\subsubsection{Optimal loss function for Signed model}\label{sec:opt_special}




In the case of Signed model, it can be proved that the conditions of Lemma \ref{lem:convex} is satisfied. Since $W_\sig=\sig G+SY$, we derive the probability density of $W_{\sig}$ as follows :
\begin{align*}
p_{W_{\sig}}(w) = p_{_{\sig G}}(w) * p_{SY}(w)= \frac{\text{exp}(-w^2/(2+2\sig^2))}{\sqrt{2\pi(1+\sig^2)}}\cdot f(w),
\end{align*}
where $$f(w)=2-2\text{Q}(w/(\sig\sqrt{2+2\sig^2})).$$
Direct calculation shows that $f$ is a log-concave function for all $w \in \mathbb{R}$. Therefore 
\begin{align*}
(\log(p_{W_{\sig}}))^{\prime\prime}(w)&=-\frac{1}{\sig^2+1}+(\log(f))^{\prime\prime}(w)\\
&\le -\frac{1}{\sig^2+1}.
\end{align*}
This proves the convexity of optimal loss function derived according to Theorem \ref{thm:opt_loss} when measurements follow the Signed model.

\section{Noisy Signed Measurement Model}\label{sec:noisysigned}
Consider a noisy-signed label function as follows:
\begin{align*}
y_i =f_{\eps}(\ab_i^T\x_0) =  \begin{cases}  \sign(\ab_i^T\x_0) &, \text{w.p.}~~1-\eps, \\  -\sign(\ab_i^T\x_0) &, \text{w.p.}~~\eps,  \end{cases}
\end{align*}
where $\eps \in [0,1/2]$. 
\begin{figure}
  \centering
    \includegraphics[width=7.8cm,height=6.4cm]{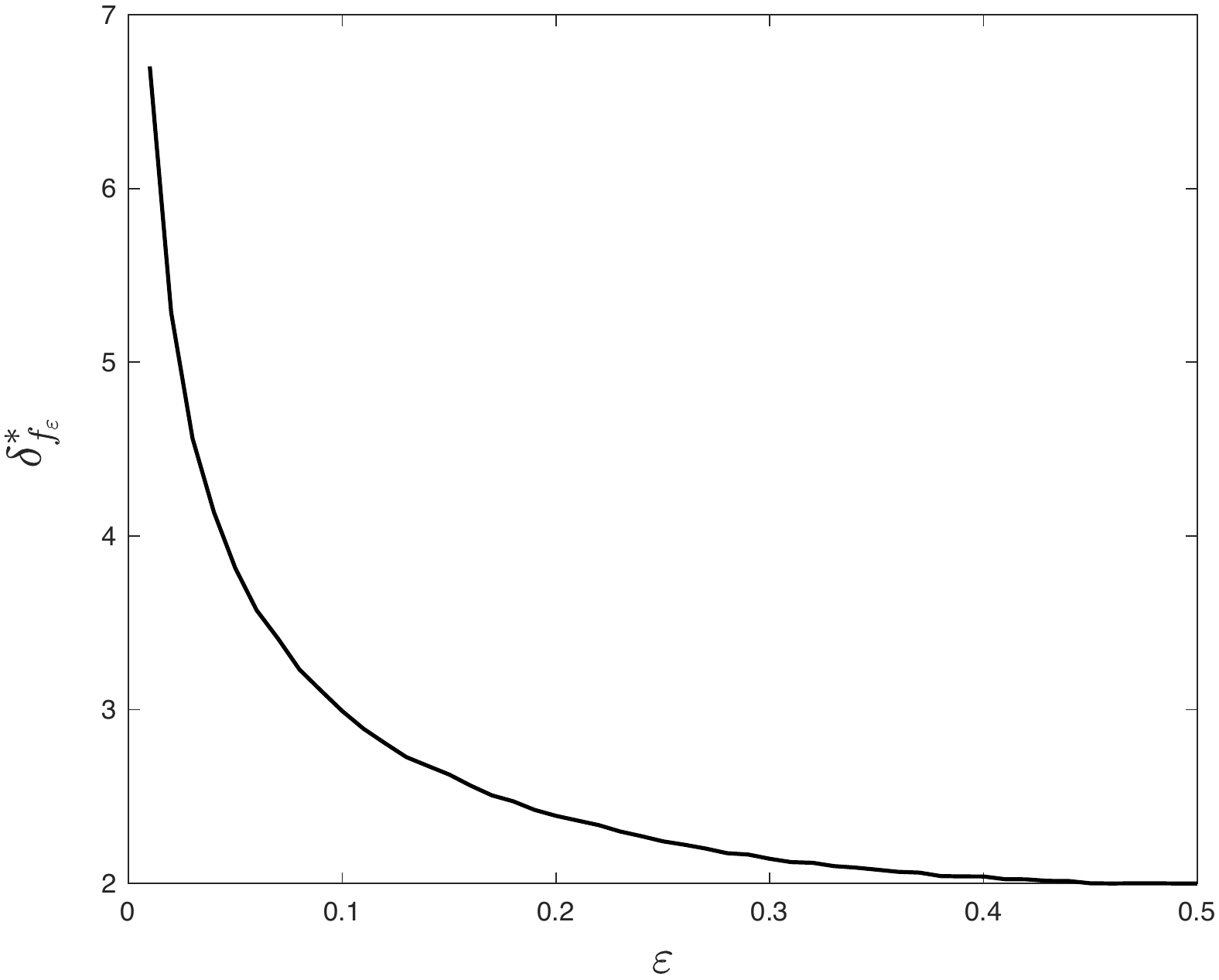}
      \caption{\footnotesize The value of the threshold $\delta^\star_{f_\eps}$ in \eqref{eq:thcandes} as a function of probability of error \small$\eps\in[0,1/2]$. For Logistic and Hinge-loss, the set of minimizers in \eqref{eq:gen_opt} is bounded (as required by Theorem \ref{thm:main}) iff $\delta>\delta^\star_{f_\eps}$.}
          \label{fig:thresholdfig}
 \end{figure}
In the case of signed measurements i.e., \small$y_i = \sign(\ab_i^T\x_0)$, \normalsize it can be observed that for all possible values of $\delta$, the condition \eqref{eq:sep} in Section \ref{sec:logloss} holds for $\x_s = \x_0$. This implies the separability of data and therefore the solution to the optimization problem \eqref{eq:gen_opt} is unbounded for all $\delta$. However in the case of noisy signed label function, boundedness or unboundedness of solutions to \eqref{eq:gen_opt} depends on $\delta$. As it was discussed in Section \ref{sec:logloss}, the minimum value of $\delta$ for bounded solutions is derived from the following:
\begin{equation}\label{eq:thcandes}
\delta^\star_{f_{\eps}}(\eps):= \left(\min_{c\in\R}\E\left[\left(G+c\,S\,Y\right)_{-}^2\right]\right)^{-1},
\end{equation}
where $Y=f_{\eps}(S)$. It can be checked analytically that $\delta^\star_{f_{\eps}}$ is a decreasing function of $\eps$ with $\delta^\star_{f_{\eps}}(0^+)=+\infty$ and $\delta^{\star}_{f_{\eps}}(1/2)=2$. \\
In Figure \ref{fig:thresholdfig}, we have numerically evaluated the threshold value $\delta^\star_{f_\eps}$ as a function of the probability of error $\eps$. For $\delta<\delta^\star_{f_\eps}$, the set of minimizers of the \eqref{eq:gen_opt} with logistic or hinge loss is unbounded. \\
The performances of LS, LAD and Hinge loss functions for Noisy-signed measurement model with $\eps= 0.1$ and $\eps=0.25$ are demonstrated in Figures \ref{fig:eppointone} and \ref{fig:eppointtwentyfive}, respectively. Comparing performances of Least-Squares and Hinge-loss functions suggest that hinge-loss is robust to measurement corruptions, as for moderate to large values of $\delta$ it outperforms the LS  estimator.  Theorem \ref{thm:main} opens the way to analytically confirm such conclusions, which is an interesting future direction. 

\begin{figure}
\centering
	\begin{subfigure}{0.47\textwidth}
    		\includegraphics[width=7.2cm,height=6.3cm]{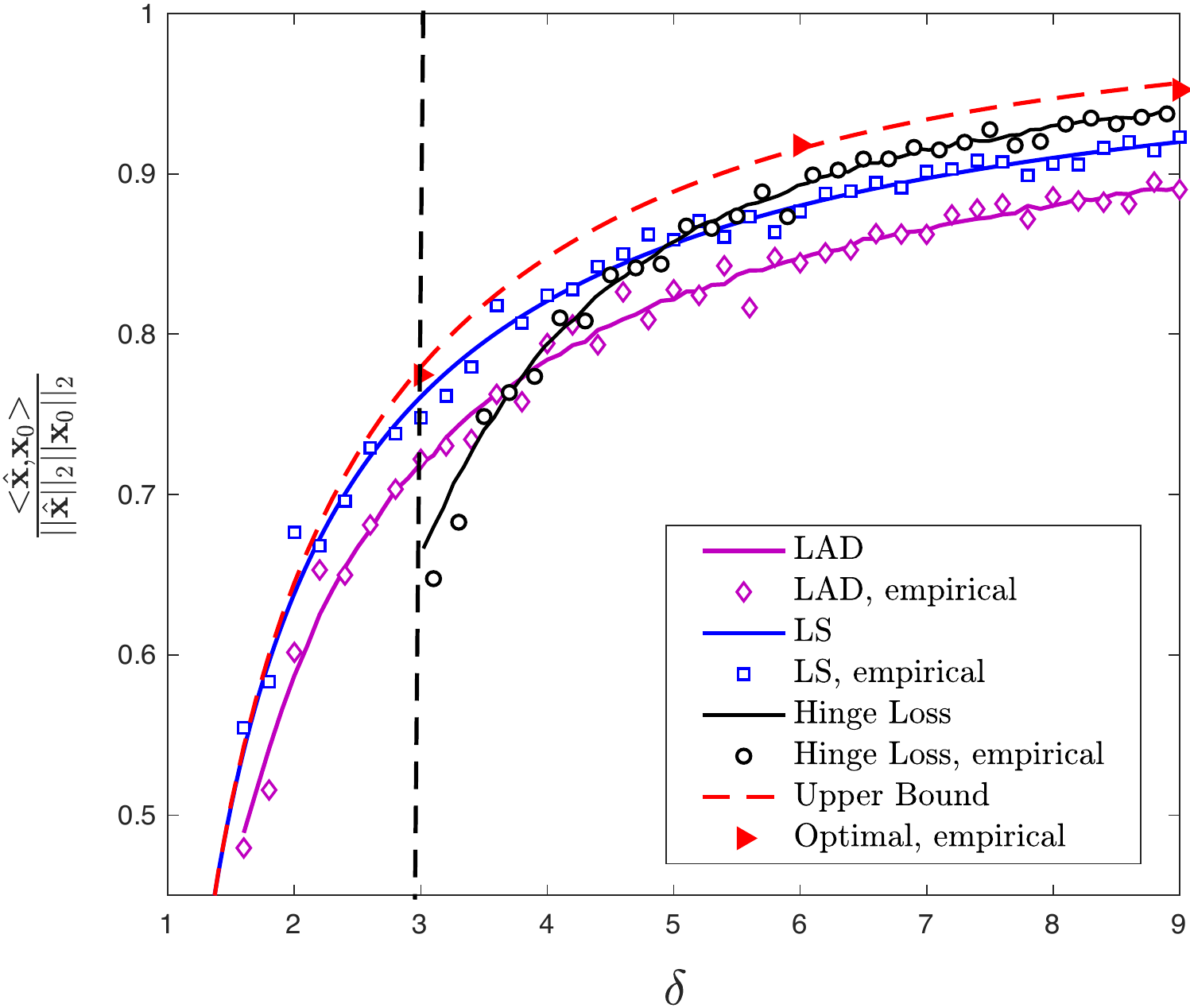}
		\caption{{\footnotesize $\eps = 0.1$}}
    		\label{fig:eppointone}
\end{subfigure}
    	\begin{subfigure}{0.47\textwidth}
	\centering
    		\includegraphics[width=7.2cm,height=6.39cm]{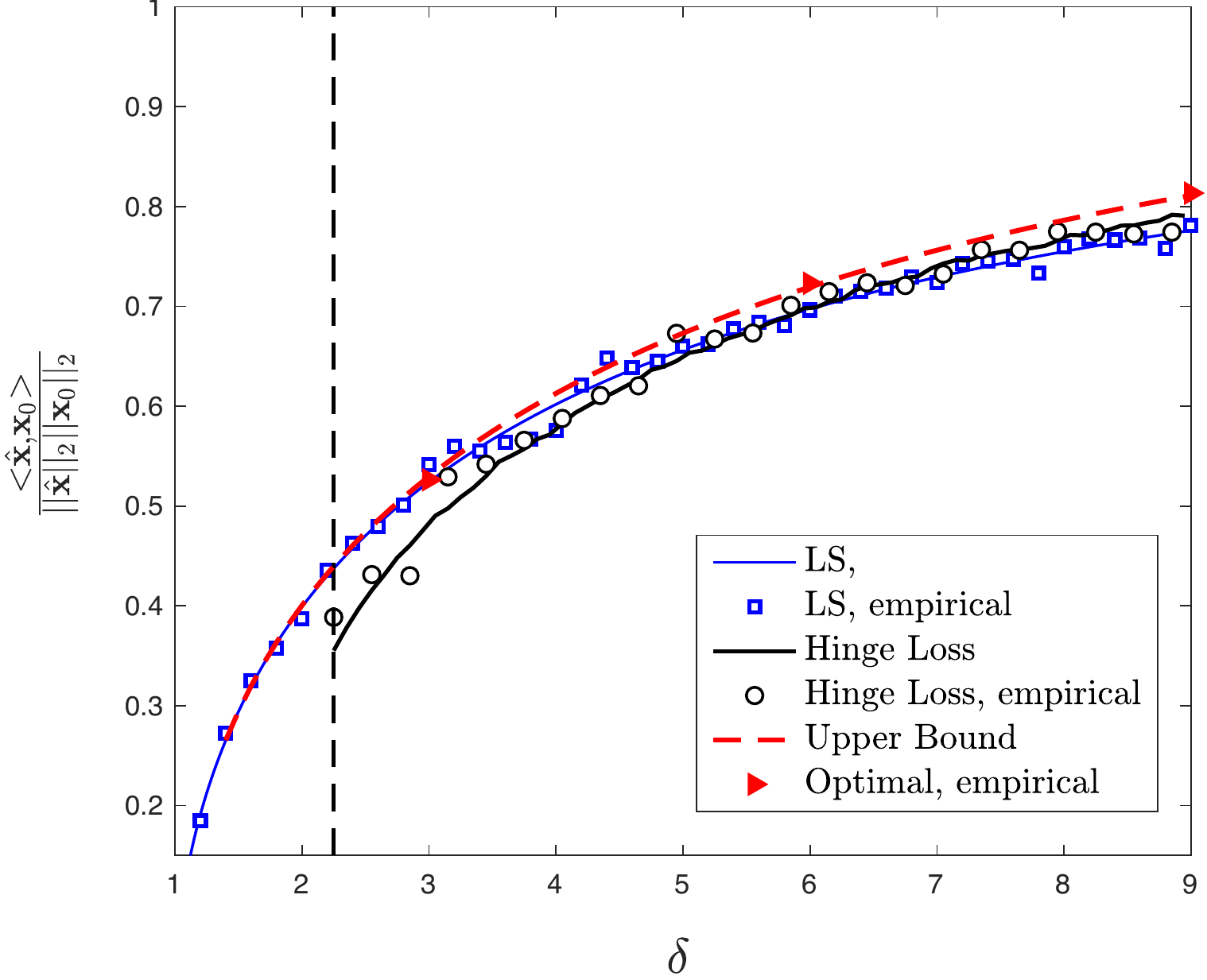}
		\caption{{\footnotesize $\eps = 0.25$}}
    		\label{fig:eppointtwentyfive}
\end{subfigure}
\caption{\footnotesize Comparisons between analytical and empirical results for the least-squares (LS), least-absolute deviations and Hinge loss functions along with the upper bound on performance and the empirical performance of optimal loss function as in Theorem \ref{thm:opt_loss}, for Noisy-signed measurement model with $\eps = 0.1$ (left) and $\eps=0.25$ (right). The vertical dashed lines are evaluated by \eqref{eq:thcandes} and represent $\delta^\star_{f_\eps} \approx 3$ and $2.25$ for $\eps=0.1$ and $0.25$, respectively.}
 \end{figure}

\section{Proofs and discussion on LS performance}\label{sec:LS_app}

\subsection{Proof of Corollary \ref{cor:LS}}\label{sec:corLS}
In order to get the values of $\alpha$ and $\mu$ as in the statement of the corollary, we show how to simplify Equations \eqref{eq:eq_main} for $\ell(t)=(t-1)^2$. In this case, the proximal operator admits a simple expression:
\begin{equation}
\prox{\ell}{x}{\la} = ({x+2\la})\Big/({1+2\la}).\nn
\end{equation}
Also, $\elld(t)=2(t-1)$.
Substituting these in \eqref{eq:mu_main2} gives the formula for $\mu$ as follows:
\bea\nn
0 &= \E\left[YS(\alpha G + \mu SY - 1)\right] = \mu\, \E[S^2] - \E[YS]\\
&\qquad\qquad\qquad \Longrightarrow
\mu = \E[YS], \nn
\eea
where we have also used from \eqref{eq:GSY} that $\E[S^2]=1$ and $G$ is independent of $S$.
Also, since $\elldd(t)=2$, direct application of \eqref{eq:lambda_main3} gives
\bea
1 = \lambda\delta\,\frac{2}{1+2\la}\Longrightarrow \la = \frac{1}{2(\delta-1)}\nn.
\eea
Finally, substituting the value of $\lambda$ in \eqref{eq:alpha_main2} we obtain the desired value for $\alpha$ as follows :
\begin{align*}
\alpha^2 &= 4 \lambda^2 \delta\,\mathbb{E}\left[(\prox{\ell}{\ourx}{\la}-1)^2\right] \\
&= \frac{4\lambda^2}{(1+2\la)^2}\delta\,\mathbb{E}\left[(\alpha G+\mu S Y -1)^2 \right] \\
&=\frac{4\la^2\delta}{(1+2\la)^2}(\alpha^2 +\mu^2 + 1 - 2\mu \E[SY])\\
&=\frac{1}{\delta}(\alpha^2+1-\left(\E[SY]\right)^2)\label{eq:alpha}\quad \\ 
&\Longrightarrow \alpha =\sqrt{1-\left(\E[SY]\right)^2}\cdot\sqrt{\frac{1}{\delta-1}}.
\end{align*}

\subsection{Discussion}

\paragraph{Linear vs binary.}~~On the one hand, Corollary \ref{cor:LS} shows that least-squares performance for binary measurements satisfies
\begin{equation}\label{eq:norm_LS}
\lim_{n\rightarrow \infty} \Big\|\xh-\frac{\mu}{\|\x_0\|_2}\cdot{\x_0}\Big\|_2^2 = \tau^2\cdot \frac{1}{\delta-1},
\end{equation}
where $\mu$ is as in \eqref{eq:mu_LS} and $\tau^2:=1-(\E[SY])^2$. On the other hand, it is well-known (e.g., see references in \cite[Sec.~5.1]{Master}) that least-squares for (scaled) linear measurements with additive Gaussian noise (i.e. $y_i= \rho \ab_i^T\x_0 + \sigma z_i$, $z_i\sim\Nn(0,1)$) 
leads to an estimator that satisfies
\bea
\lim_{n\rightarrow \infty} \left\|\xh-\rho\cdot{\x_0}\right\|_2^2 = \sigma^2\cdot \frac{1}{\delta-1}.\label{eq:norm_LS_lin}
\eea
Direct comparison of \eqref{eq:norm_LS} to \eqref{eq:norm_LS_lin} suggests that least-squares with binary measurements performs the same as if measurements were linear with scaling factor $\rho=\mu/\|\x_0\|_2$ and noise variance $\sigma^2=\tau^2=\alpha^2(\delta-1)$.  This worth-mentioning conclusion is not new as it was proved in \cite{Bri,PV15,NIPS,genzel2017recovering}. We include a short discussion on the relation to this prior work in the following paragraph. We highlight that all these existing results are limited to a least-squares loss unlike our general analysis. 

\paragraph{Prior work.}
There is a lot of recent work on the use of least-squares-type estimators for recovering signals from nonlinear measurements of the form $y_i=h(\ab_i^T\x_0)$ with Gaussian vectors $\ab_i$. The original work that suggests least-squares as a reasonable estimator in this setting is due to Brillinger \cite{Bri}. In his 1982 paper, Brillinger studied the problem in the classical statistics regime (aka $n$ is fixed not scaling with $m\rightarrow+\infty$) and he proved for the least-squares solution satisfies
$$
\lim_{m\rightarrow+\infty} \frac{1}{m}\left\|\xh-\frac{\mu}{\|\x_0\|_2}\cdot{\x_0}\right\|_2^2 = \tau^2,
$$
where
\begin{align}
\mu &= \E[S Y],\quad\quad\qquad S\sim\Nn(0,1),\nn\\
\tau^2 &= \E[(Y-\mu S)^2].\label{eq:Bri}
\end{align}
and the expectations are with respect to $S$ and possible randomness of $f$. Evaluating \eqref{eq:Bri} for $Y=f_{\eps}(S)$ leads to the same values for $\mu$ and $\tau^2$ in \eqref{eq:norm_LS}. In other works, \eqref{eq:norm_LS} for $\delta\rightarrow+\infty$ indeed recovers Brillinger's result. The extension of Brillinger's original work to the high-dimensional setting (both $m,n$ large) was first studied by Plan and Vershynin \cite{PV15}, who derived (non-sharp) non-asymptotic upper bounds on the performance of constrained least-squares (such as the Lasso). Shortly after, \cite{NIPS} extended this result to \emph{sharp} asymtpotic predictions and to regularized least-squares. In particular, Corollary \ref{cor:LS} is a special case of the main theorem in \cite{NIPS}. Several other interesting extensions of the result by Plan and Vershynin have recently appeared in the literature, e.g., \cite{genzel2017high,goldstein2018structured,genzel2017recovering,thrampoulidis2018generalized}. However, \cite{NIPS} is the only one to give results that are sharp in the flavor of this paper. Our work, extends the result of \cite{NIPS} to general loss functions beyond least-squares. The techniques of \cite{NIPS} that have guided the use of the CGMT in our context have also been recently applied in \cite{PhaseLamp} in the context of phase-retrieval.
%

%
%
%
%
%

\end{document}